\documentclass[hidelinks,onefignum,onetabnum]{siamonline220329}

\usepackage{lipsum}
\usepackage{amsfonts}
\usepackage{graphicx}
\usepackage{epstopdf}
\usepackage{amsmath,amssymb}
\usepackage{soul}
\usepackage{xcolor}
\usepackage{algorithm}
\usepackage{algpseudocode}
\usepackage{subcaption}
\usepackage{mathtools}
\usepackage{booktabs}
\usepackage{aligned-overset}
\ifpdf
  \DeclareGraphicsExtensions{.eps,.pdf,.png,.jpg}
\else
  \DeclareGraphicsExtensions{.eps}
\fi
\usepackage[shortlabels]{enumitem}
\setlist[enumerate]{leftmargin=.5in}
\setlist[itemize]{leftmargin=.5in}

\newsiamremark{remark}{Remark}
\newsiamremark{hypothesis}{Hypothesis}
\newsiamremark{assumption}{Assumption}
\crefname{hypothesis}{Hypothesis}{Hypotheses}
\newsiamthm{claim}{Claim}
\newsiamthm{prop}{Proposition}

\headers{Plug-and-play quasi-Newton methods}{H. Y. Tan, S. Mukherjee, J. Tang, and C.-B. Sch\"onlieb}

\title{Provably Convergent Plug-and-Play Quasi-Newton Methods}

\author{Hong Ye Tan\thanks{Department of Applied Mathematics and Theoretical Physics, University of Cambridge, UK (\email{hyt35@cam.ac.uk}, \email{sm2467@cam.ac.uk}, \email{cbs31@cam.ac.uk}).}
\and Subhadip Mukherjee\footnotemark[1] \thanks{Department of Computer Science, University of Bath, UK. Department of Electronics and Electrical Communication Engineering, Indian Institute of Technology, Kharagpur, India (\email{smukherjee@ece.iitkgp.ac.in}).}
\and Junqi Tang\footnotemark[1] \thanks{School of Mathematics, University of Birmingham, UK (\email{j.tang.2@bham.ac.uk}).}
\and Carola-Bibiane Sch\"onlieb\footnotemark[1]}

\usepackage{amsopn}

\newcommand{\R}{\mathbb{R}}
\DeclareMathOperator*{\argmin}{arg\,min}
\DeclareMathOperator*{\argmax}{arg\,max}

\DeclareMathOperator*{\zer}{zer}
\DeclareMathOperator*{\dom}{dom}
\DeclareMathOperator{\prox}{prox}
\DeclareMathOperator{\dist}{dist}

\usepackage{acro}

\ifpdf
\hypersetup{
  pdftitle={Plug and Play Quasi-Newton Methods},
  pdfauthor={Hong Ye Tan, Subhadip Mukherjee, Junqi Tang, and Carola-Bibiane Sch\"onlieb}
}
\fi

\newcommand{\Rev}[1]{{\color{black}{#1}}}
\newcommand{\RRev}[1]{{\color{black}{#1}}}
\begin{document}

\maketitle

\begin{abstract}
%
\Rev{Plug-and-Play (PnP) methods are a class of efficient iterative methods that aim to combine data fidelity terms and deep denoisers using classical optimization algorithms, such as ISTA or ADMM, \RRev{with applications in inverse problems and imaging}. Provable PnP methods are a subclass of PnP methods with convergence guarantees, such as fixed point convergence or convergence to critical points of some energy function. Many existing provable PnP methods impose heavy restrictions on the denoiser or fidelity function, such as \textit{nonexpansiveness} or \textit{strict convexity}, respectively. In this work, we propose a novel algorithmic approach incorporating quasi-Newton steps into a provable PnP framework based on proximal denoisers, resulting in greatly accelerated convergence while retaining light assumptions on the denoiser. By characterizing the denoiser as the proximal operator of a weakly convex function, we show that the fixed points of the proposed quasi-Newton PnP algorithm are critical points of a weakly convex function. Numerical experiments on image deblurring and super-resolution demonstrate \RRev{2--8x} faster convergence as compared to other provable PnP methods with similar \RRev{reconstruction quality}.}
\end{abstract}



\section{Introduction}
\Rev{Many image restoration problems can be formulated as reconstructing data $x\in\mathbb{R}^n$ from a noisy measurement $y=Ax+\varepsilon\in\mathbb{R}^m$, where $A$ is a linear forward operator, and $\varepsilon$ is some measurement noise. One common way to solve this is the variational formulation}
\begin{equation} \label{eq:regProb}
    \argmin_{x \in \R^n} \varphi(x) = f(x) + g(x),
\end{equation}%
where $f:\R^n \rightarrow \R$ is typically a continuously differentiable data fidelity term, and $g:\R^n \rightarrow \overline{\R}$ is a regularization term that controls the prior. In many cases, the fidelity term incorporates a forward operator $A:\R^n \rightarrow \R^m$, which may correspond to physical operators such as blurring operators or Radon transforms \cite{helgason1980radon}. For a noisy measurement $y = Ax + \varepsilon$ with additive white noise $\varepsilon \sim \mathcal{N}(0, \sigma^2 I)$, the fidelity term takes the form of the negative log likelihood $f(x) = \|Ax-y\|^2/(2\sigma^2)$. For many physical forward operators, such as blurring or down-sampling, the optimization problem $\min_x f(x)$ is ill-posed, thus a regularization term is needed \cite{kaipio2006statistical}. Classical examples for regularization include using Fourier spectra \Rev{(spectral regularization)} or total variation (TV) regularization on natural images \cite{ruderman94FourierReg,rudin1992nonlinear}, whereas recent works aim to learn a neural network regularizer \cite{lunz2018adversarial,mukherjee2020LearnedConvexreg}. 

Fully data-driven approaches have been shown to outperform explicitly defined regularizers \cite{zhang2021plug,wen2018survey,mukherjee2020LearnedConvexreg}. However, the outputs of these learned schemes often do not correspond to closed-form minimization problems of the form \eqref{eq:regProb}. This is particularly limiting in sensitive applications such as medical imaging, where interpretability is necessary \cite{vellido2020importance,tjoa2020survey}. \Rev{Recent lines of work consider combining iterative algorithms with generic denoisers, with notable examples including regularization by denoising (RED) \cite{cohen2021regularization,reehorst2018regularization}, consensus equilibrium \cite{buzzard2018plug}, and deep mean-shift priors \cite{arjomand2017deep}. In this work, we will focus on the line of Plug-and-Play (PnP) methods, which arise from replacing proximal steps with denoisers}. Under certain conditions on the fidelity and denoisers as detailed in Section \ref{ssec:pnp}, fixed point convergence \Rev{of certain PnP methods} can be established, characterized by critical points of a corresponding functional. 



The PnP framework of replacing the regularization proximal step with a denoiser is flexible in the choice of denoiser. In particular, it allows for the use of both classical denoisers such as NLM or BM3D \cite{buades2011NLM,dabov2007bm3d}, as well as data-driven denoisers \cite{zhang2017learning,zhang2021plug,ryu2019plug}. \RRev{This allows for extending the use of Gaussian denoisers to other image reconstruction tasks, such as super-resolution or image deblocking. Recently, PnP methods based on the half-quadratic splitting were able to achieve state-of-the-art performance for image reconstruction using a variable-strength Gaussian denoiser called DRUNet \cite{zhang2017learning}. Named the deep Plug-and-Play image restoration (DPIR) method, DPIR outperforms or is competitive with fully learned methods for applications such as image deblurring, super-resolution, and demosaicing while using only a single denoiser prior \cite{zhang2021plug}.} This work demonstrates the flexibility of PnP, using one prior for multiple reconstruction tasks.

While PnP methods can be used to achieve excellent performance, empirical convergence does not equate to traditional notions of convergence. Indeed, while DPIR is able to achieve state-of-the-art results in as few as eight PnP iterations, there are no associated theoretical results. Moreover, DPIR can diverge when more PnP iterations are applied \cite{hurault2021gradient}. This can be empirically alleviated using various stopping criteria, but this raises an additional issue for defining a notion of ``best reconstruction". In this work, we sidestep this by considering provable PnP methods. We use the term ``provable PnP" to refer to PnP methods equipped with some notion of convergence, such as fixed-point convergence, or the stronger notion of convergence to critical points of a function. 

\Rev{Various approaches for accelerating PnP methods have been proposed, including using classical accelerated optimization algorithms, block-coordinate methods, parallelization, and dimensionality reduction \cite{kamilov2017plug,gan2023block,tang2022accelerating,kamilov2023plug,sun2020async}. In the context of convergence to fixed points of a functional, theoretical results for PnP based on accelerated classical methods such as FISTA have not arisen in the literature. This work proposes to extend the work on provable PnP methods by introducing a quasi-Newton step to accelerate convergence, while retaining a corresponding closed-form minimization problem with relatively weak constraints.}

\subsection{Definitions and Notations}
We begin with some definitions and notation. Let $\overline{\R} = \R \cup \{+\infty\}$ be the extended real line. Recall that a function $g:\R^n \rightarrow \overline{\R}$ is \emph{proper} if the effective domain $\dom g = \{x \in \R^n \mid g(x) < +\infty\}$ is nonempty, and \emph{closed} (or \emph{lower-semicontinuous}) if for every sequence $x^k \rightarrow x$ in $\R^n$, we have $g(x) \le \liminf_k g(x^k)$.

\begin{definition}
For a scalar $\gamma> 0$ and a proper closed convex function $g:\R^n \rightarrow \overline{\R}$, the \emph{proximal map} is 
\begin{equation} \label{eq:proxDefi}
    \prox_{\gamma g}(x) = \argmin_{u \in \R^n} \left\{g(u) + \frac{1}{2\gamma} \|u-x\|^2\right\}.
\end{equation}
The \emph{Moreau envelope} is the value function of the proximal map, defined as
\begin{equation}
    g^\gamma(x) = \min_{u \in \R^n }\left\{g(u) + \frac{1}{2\gamma} \|u-x\|^2\right\}.
\end{equation}

\end{definition}
Properties of the Moreau envelope and proximal operators are well documented in classical literature \cite{Rockafellar1972CA, bauschke2011convex,moreau1965proximite,gribonval2020characterization}. In particular, \RRev{for proper closed convex $g$, the proximal operator is single-valued and nonexpansive, and the envelope function $g^\gamma$ is convex and $\mathcal{C}^1$ with derivative}
\[\nabla g^\gamma(x) = \gamma^{-1} (x - \prox_{\gamma g}(x)).\]

\subsection{Plug-and-Play Methods}\label{ssec:pnp}

The Plug-and-Play (PnP) framework was first introduced by Venkatakrishnan et al. in 2013 for model-based image reconstruction \cite{venkatakrishnan2013plug}. \Rev{PnP methods arise from composite convex optimization algorithms, wherein a prior regularization step is associated with a denoising step.} The first composite optimization algorithm considered was Alternating Directions Method of Multipliers (ADMM), a classical proximal splitting algorithm used for minimizing composite functions. In the case of image reconstruction, a maximum likelihood estimation model can be decomposed into a composite problem. For a noisy measurement $y$ and unknown data $x$, let $p(y|x)$ be the conditional likelihood, and $p(x)$ the prior of the unknown $x$. The maximum a-posteriori (MAP) estimate $\hat{x}$ is given as follows:
\begin{align*}
    \hat{x} &= \argmax_x \left\{ p(y|x) + p(x) \right\} \\
    &= \argmin_x \left\{f(x;y) + g(x)\right\},
\end{align*}
where $f$ is the likelihood/fidelity term, and $g$ is the prior/regularization term. A classical example would be TV regularization for additive Gaussian noise, where the fidelity term is $f(x;y) = \|Ax - y\|_2^2/2\sigma^2$, and the prior term is $g(x) = \lambda \|\nabla x\|_1$ \cite{rudin1992nonlinear}. To solve the minimization problem for general convex $f,\,g$, \RRev{proximal splitting algorithms such as ADMM consider alternating applications of the individual proximal operators $\prox_f, \prox_g$ or subgradients $\partial f, \partial g$. The key observation of PnP is that the prior regularization step can also be interpreted as a denoising operation \cite{ryu2019plug}.}

More generally, the PnP framework can be applied to monotone operator splitting methods. Under light conditions, the composite convex optimization problem of minimizing $f+g$ can be reformulated as the monotone inclusion problem $0 \in \partial f(x) + \partial g(x)$ \cite{Rockafellar1972CA,bauschke2011convex}. For convex $f$ and $g$, the operators $\partial f$ and $\partial g$ are monotone operators. Monotone operator splitting methods aim to solve the inclusion $0 \in \partial f(x) + \partial g(x)$, using only \RRev{the resolvents of the individual operators $\partial f, \partial g$, and/or the individual operators $\partial f, \partial g$ themselves} \cite{bauschke2011convex}. In convex analysis terms, this corresponds to splitting the proximal operator $\prox_{f+g}$ in terms of the simpler proximals $\prox_f$ and $\prox_g$ or gradients $\nabla f$ and $\nabla g$. Two common splitting algorithms are the forward-backward splitting (FBS) and the Douglas-Rachford splitting (DRS), given as follows \cite{bauschke2011convex,douglas1956numerical}: %
\begin{equation}\tag{FBS}
    x^{k+1} = \prox_g(I - \nabla f)(x^k); 
\end{equation}%
\begin{equation} \tag{DRS}
\left\{
    \begin{aligned}
        x^{k+1} &= \prox_f(y^k),\\
        y^{k+1} &= y^k + \prox_g(2x^{k+1} - y^k) - x^{k+1}.
    \end{aligned}\right.
\end{equation}

One classical application of a splitting algorithm is the iterative thresholding and shrinkage algorithm (ISTA) for LASSO problems, where the fidelity $f$ is quadratic, and the prior term is the $\ell_1$ norm $g(x) = \|x\|_1$ \cite{daubechies2004iterative,beck2009fast}. Applying the PnP framework to FBS and DRS, by replacing the prior proximal terms $\prox_g$ with a denoiser $D_\sigma$, gives the PnP-FBS and PnP-DRS methods. 

\begin{equation}\tag{PnP-FBS}
    x^{k+1} = D_\sigma(I - \nabla f)(x^k); 
\end{equation}%
\begin{equation} \tag{PnP-DRS}
\left\{
    \begin{aligned}
        x^{k+1} &= \prox_f(y^k),\\
        y^{k+1} &= y^k + D_\sigma(2x^{k+1} - y^k) - x^{k+1}.
    \end{aligned}\right.
\end{equation}

Provable PnP results first arose by Chan et al. for the PnP-ADMM scheme, demonstrating fixed-point convergence under a bounded denoiser assumption $\|D_\sigma(x) - x\| \le C \sigma^2$ \cite{chan2016plug}. Ryu et al. demonstrate convergence of the PnP-FBS algorithm when $f$ is strongly convex and the denoiser residual $D_\sigma - I$ is Lipschitz with \RRev{sufficiently small Lipschitz constant, as well as for PnP-DRS and PnP-ADMM in the case where $D_\sigma - I$ is Lipschitz with Lipschitz constant less than 1} \cite{ryu2019plug}. Various works show fixed-point convergence of PnP-ADMM and PnP-FBS when $f$ has Lipschitz gradient under an ``averaged denoiser" assumption, where $(1-\theta)I + \theta D_\sigma$ is nonexpansive for some $\theta \in (0,1)$, mainly using monotone operator theory \cite{sun2019online,sun2021scalable,hertrich2021convolutional}. Cohen et al. show fixed-point convergence of a relaxed PnP-FBS scheme when $f$ has Lipschitz gradient under a demicontractive denoiser assumption, which is a strictly weaker condition than nonexpansiveness \cite{cohen2021regularization}. Sreehari et al. show convergence of PnP-ADMM to an implicitly defined convex function when the denoiser is nonexpansive and has symmetric gradient, by utilizing Moreau's theorem to characterize the denoiser as a proximal map of a convex function \cite{sreehari2016plug,moreau1965proximite}. In the case of nonexpansive linear denoisers, PnP-FBS and PnP-ADMM converge to fixed points of a closed-form convex optimization problem \cite{nair2021fixed}. 


\Rev{While plentiful, many of these convergence results impose restrictive or difficult-to-verify conditions on the denoisers $D_\sigma$.} Instead of replacing the regularizing proximal operator $\prox_g$ with a denoiser, Hurault et al. and Cohen et al. instead consider applying FBS with the proximal operator on the fidelity term and a gradient step on the regularization, $x^{k+1} = \prox_f(I - \nabla g)(x^k)$ \cite{hurault2021gradient,cohen2021has}. Replacing the regularization step with a denoiser $D_\sigma = I - \nabla g_\sigma$ results in the Gradient Step PnP (GS-PnP) algorithm $x^{k+1} = (\prox_f \circ D_\sigma) (x^k)$. Using this parameterization, they show further that the fixed points of GS-PnP are stationary points of a particular (non-convex) function. Moreover, a follow-up work shows that a gradient-step denoiser of the form $D_\sigma = I - \nabla g_\sigma$ can be interpreted as a proximal step $D_\sigma = \prox_{\phi_\sigma}$ \cite{hurault2022proximal}. Using this, they are able to achieve iterate convergence under KL-type conditions to a stationary point of a (non-convex) closed-form functional of the form \eqref{eq:regProb}.

\Rev{The GS-PnP style schemes require that the gradient of the potential $\nabla g_\sigma$ is Lipschitz with Lipschitz constant $L<1$. Methods of training neural networks with Lipschitz constraints include spectral regularization, adversarial training against Lipschitz bounds during training, or spline based architectures \cite{ryu2019plug,miyato2018spectral,fazlyab2019efficient,neumayer2023approximation}. Hurault et al. consider fine-tuning the DRUNet denoiser by using spectral regularization to enforce the Lipschitz gradient condition \cite{hurault2022proximal}. While it can be shown empirically that the Lipschitz constant is less than one locally, there is no theoretical guarantee, which can lead to occasional divergence. One possible way of remedying this is by averaging the denoiser with the identity operator, as remarked in \cite{hurault2022proximal}. This consists of replacing the denoiser $D_\sigma = I - \nabla g_\sigma$ with the relaxed $D_\sigma^\alpha \coloneqq (1-\alpha) I + \alpha D_\sigma = I - \alpha \nabla g_\sigma$ for some $\alpha \in (0,1)$. We can rewrite the relaxed denoiser as $D_\sigma^\alpha = I - \nabla g_\sigma^\alpha$, where $g_\sigma^\alpha = \alpha g_\sigma$ has $\alpha L$-Lipschitz gradient. Taking $\alpha<1/L$ gives the appropriate contraction condition on $g_\sigma^\alpha$ and thus convergence of the associated PnP schemes \cite{hurault2022proximal,hurault2023RelaxedProxDenoiser}.}

\subsection{Quasi-Newton Methods}
For minimizing a twice continuously differentiable function $f:\R^n \rightarrow \R$, a classical second-order method is Newton's method \cite{NoceWrig06}:
\begin{equation}
    x^{k+1} = x^{k} -  (\nabla^2 f)^{-1} \nabla f (x^{k}),
\end{equation}
where $\nabla^2 f$ is the Hessian of $f$. This can be interpreted as minimizing a local quadratic approximation
\begin{subequations}
\begin{gather}
    \hat{f}_k(y) = f(x^{k})+ \nabla f(x^{k})^\top (y - x^{k}) + \frac{1}{2} (y-x^{k})^\top \nabla^2 f(x^k) (y - x^{k}),\\
    x^{k+1} = \argmin_y \hat{f}_k(y).
\end{gather}
\end{subequations}
Newton's method is able to achieve quadratic convergence rates with appropriate initialization and step-sizes \cite{NoceWrig06}. However, the inverse of the Hessian may be computationally demanding, especially in high-dimensional applications such as image processing. Quasi-Newton (qN) methods propose to replace the inverse Hessian $(\nabla^2 f)^{-1}$ with (low-rank) approximations to the inverse Hessian, \Rev{with notable examples including the Broyden-Goldfarb-Fletcher-Shanno (BFGS) algorithm, the David-Fletcher-Powell (DFP) formula, and the symmetric rank one method (SR1) \cite{NoceWrig06}}. 

\Rev{Like Newton's method, quasi-Newton methods utilize the curvature information from the Hessian approximation to accelerate convergence, with applications in non-convex stochastic optimization, neural network training, and Riemannian optimization \cite{byrd2016stochastic,huang2015broyden,wang2017stochastic}. Classical theory gives asymptotic superlinear convergence under the Dennis-Mor\'e condition, which states that the Hessian approximation converges to the Hessian at the minimum \cite{dennis1974characterization}. Non-asymptotic convergence of quasi-Newton methods is still an active area of research. BFGS and DFP have only recently been shown to have non-asymptotic superlinear convergence rates of $\mathcal{O}((1/k)^{k/2})$ when the objective function is strongly convex with Lipschitz continuous gradient, has Lipschitz continuous Hessian at the minimum, and satisfies a concordance condition \cite{jin2023non,rodomanov2021rates}. However, BFGS sees empirical success even when these conditions are not explicitly verified, including in the non-convex setting \cite{li2001modified,li2001global}. Interestingly, certain accelerated proximal gradient methods can be interpreted as a proximal quasi-Newton method \cite{ochs2019adaptive}.}

\Rev{Variants of BFGS include limited memory BFGS (L-BFGS), stochastic BFGS, greedy BFGS, and sharpened BFGS \cite{liu1989limited,jin2022sharpened,mokhtari2015global,schraudolph2007stochastic,rodomanov2021greedy}. Of these variants, the limited memory version is most suited to repeated iteration. Standard quasi-Newton methods continually update the Hessian approximation using all the previous iterates, leading to a linear per-iteration computational cost increase. L-BFGS instead utilizes only the last $m$ iterates, where $m>1$ is a user-specified parameter, typically chosen to be less than 50. }\RRev{Moreover, the Hessian need not be stored and/or computed at each iteration, as the method only relies on Hessian-vector products, which can be computed efficiently with two loop recursions \cite{NoceWrig06}.}  

To relate \Rev{quasi-Newton methods} to the PnP framework described previously, we would like to consider applying Newton-type methods for convex composite optimization, by replacing a proximal operator with a denoiser. Lee et al. consider the problem of minimizing 
\begin{equation}
    \varphi(x) = f(x) + g(x),
\end{equation}
where $f(x)$ is a convex $\mathcal{C}^1$ function, and $g$ is a possibly non-smooth convex regularizer \cite{lee2014proximal}. For a symmetric positive definite matrix $B_k \approx \nabla^2 f(x^k)$, the proximal Newton-type search direction $\Delta x^k$, satisfying $x^{k+1} = x^{k} + t_k \Delta x^k$, is given as the minimizer of a local quadratic approximation on the smooth component $\hat f_k(y)$:
\begin{subequations}
\begin{gather}
    \hat f_k(y) = f(x^k)+ \nabla f(x^k)^\top (y - x^k) + \frac{1}{2} (y-x^k)^\top B_k (y - x^k), \\
    \Delta x^k = \argmin_d \hat \varphi_k (x^k + d) \coloneqq \hat f_k(x^k+d) + g(x^k+d). \label{eq:proxNewtonSearchDir}
\end{gather}
\end{subequations}
Define the \emph{scaled proximal map} for a positive definite matrix $B$ as in \cite{lee2014proximal}:
\begin{equation}
    \prox_g^B(x) \coloneqq \argmin_{y \in \R^n} g(y) + \frac{1}{2} \|y - x\|_B^2,
\end{equation}
where the $B$-norm is defined as $\|z\|_B^2 = z^\top B z$. \RRev{For example, taking} $B$ to be the identity matrix results in the standard proximal map as defined in \eqref{eq:proxDefi}. The search direction \eqref{eq:proxNewtonSearchDir} has a closed form in terms of the scaled proximal map:
\begin{equation}
    \Delta x^k = \prox_g^{B_k}(x - B_k^{-1} \nabla f(x^k)) - x^k.
\end{equation}
With this search direction, appropriate step sizes and $B_k$, the proximal Newton-type methods are able to achieve similar convergence rates to Newton-type methods, achieving global convergence and local superlinear convergence. While the scaled proximal map allows for such analysis, it is not amenable to the PnP framework. For example, if we compute the Hessian approximation $B_k$ using a BFGS-type approach, a naive approach of replacing $\prox^{B_k}_g$ with a denoiser would require a careful analysis of the interaction of $B_k$ on the resulting regularization, and possibly require the denoiser to depend on $B_k$. Instead, we seek a proximal Newton-type method that utilizes only the unscaled proximal map, with possibly a scalar constant which can be easily interpreted as a regularization parameter controlling the strength of regularization. 

In \Cref{sec:sec2}, we will detail a classical composite minimization algorithm that uses only the unscaled proximal map $\prox_g$, as well as arbitrary descent steps that allow for Newton-type steps. We further extend the classical analysis from convex to weakly convex functions\Rev{, inspired by the GS-PnP characterization of denoisers as proximal maps of weakly convex functions}. In \Cref{sec:pnpqn}, we use this extension to propose the PnP-quasi-Newton (PnP-qN) method, further convergence and characterizing cluster points of the algorithm. In \Cref{sec:experiments}, we evaluate the proposed PnP-qN method \Rev{with the quasi-Newton method given by L-BFGS,} and compare it with other \RRev{provable and non-provable PnP methods with comparable reconstruction quality}. 

\section{Proximal Quasi-Newton}\label{sec:sec2}
In this section, we will first describe a classical algorithm for optimizing composite sums of a (possibly non-convex) smooth function and a (possibly non-smooth) convex function. We will then extend the analysis to allow for \textit{weak convexity} instead of \textit{convexity}. By replacing proximal terms with deep denoisers corresponding to proximal operators of weakly convex maps, we construct a Plug-and-Play scheme with convergence properties of the classical algorithm. 

Let us work on the Euclidean domain $\R^n$. Let \Rev{$\mathcal{C}^{1,1}_{L_f}$} denote the class of \Rev{$\mathcal{C}^1$} functions $f:\R^n \rightarrow \R$ with $L_f$-Lipschitz gradient, and $\Gamma_0$ the class of proper, closed, and convex functions $g:\R^n \rightarrow \overline{\R}$. Consider a variational objective having the following form:
\begin{equation}\label{eq:varphiStatement}
    \varphi = f + g,\quad f \in \Rev{\mathcal{C}^{1,1}_{L_f}},\ g \in \Gamma_0.
\end{equation}
We can consider $f$ as the fidelity term and $g$ as a regularization term. A prominent example \Rev{from} inverse problems \Rev{is} the quadratic fidelity loss $f(x;y) = \frac{1}{2}\|Ax - y\|^2$ for some linear forward operator $A:\R^n \rightarrow \R^m$ and observation $y \in \R^m$, where the norm is taken as the Euclidean norm. 

\subsection{MINFBE: Minimizing Forward-Backward Envelope}
We first detail a classical composite optimization algorithm for minimizing \eqref{eq:varphiStatement}, which will serve as the base of our proposed PnP scheme. Moreover, we describe some of its convergence properties that transfer to the PnP framework. By constructing a smooth convex envelope function around the original objective $\varphi$, this envelope can be shown to have desirable properties such as sharing minimizers, smoothness, and being minorized and majorized by convex functions. By applying descent steps and proximal mappings in a particular fashion, the classical algorithm is able to obtain global objective convergence to critical points at a rate of $\mathcal{O}(1/k)$, local linear convergence if the function is locally strongly convex, and superlinear convergence when the descent steps are taken to be quasi-Newton with suitable assumptions \cite{Stella2017}.

For the problem \eqref{eq:varphiStatement}, define the following expressions \cite{Stella2017}:
\begin{subequations}\begin{gather}
    l_\varphi(u,x) = f(x) + \langle \nabla f(x),\, u-x \rangle + g(u),\\
    T_\gamma(x) = \argmin_u  \left\{l_\varphi(u,x) + \frac{1}{2\gamma} \|u-x\|^2\right\} = \prox_{\gamma g}(x - \gamma \nabla f(x)), \label{eq:TgammaDef}\\ 
    R_\gamma(x) = \gamma^{-1} (x - T_\gamma(x)), \\
    \varphi_\gamma(x) = \min_u  \left\{l_\varphi(u,x) + \frac{1}{2\gamma} \|u-x\|^2\right\}.
\end{gather}\end{subequations}
Here, $l_\varphi$ is a local linearized decoupling of $\varphi$, $T_\gamma$ can be interpreted as an FBS step (with step-size $\gamma$ for $f+g$) and $R_\gamma$ is a scaled residual or ``gradient direction". Note that $x = T_\gamma(x) \Leftrightarrow x \in \zer \partial \varphi$, i.e. fixed points of $T_\gamma$ correspond to critical points of $\varphi$. $\varphi_\gamma$ is defined as the \emph{forward-backward envelope} of $\varphi$. We further explicitly write the Moreau envelope for $g$: 
\begin{subequations}\begin{align}
    g^\gamma (x) &= \min_u \left\{ g(u) + \frac{1}{2\gamma} \|u-x\|^2 \right\} \label{eq:gGammaDef1}\\
    &= g\left(\prox_{\gamma g}(x)\right) + \frac{1}{2\gamma} \|\prox_{\gamma g}(x) - x\|^2. \label{eq:gGammaDef2}
\end{align}\end{subequations}
With the above definitions, we have the following closed-form expressions for the forward-backward envelope:
\begin{subequations}\begin{align}
    \varphi_\gamma &= f(x) + g(T_\gamma(x)) - \gamma \langle \nabla f(x), R_\gamma (x) \rangle + \frac{\gamma}{2} \|R_\gamma(x)\|^2 \label{eq:phigamma1}\\
    &= f(x) - \frac{\gamma}{2} \|\nabla f(x)\|^2 + g^\gamma (x - \gamma \nabla f(x)). \label{eq:phigamma2}
\end{align}\end{subequations}
In fact, $\varphi_\gamma$ has many desirable properties, such as sharing minimizers with $\varphi$, and having an easily computable derivative in terms of the Hessian of $f$.
\begin{prop}[{\cite[Sec 2]{Stella2017}}]\label{prop:sharedMinimizer}
    The following holds:
    \begin{enumerate}[i.]
        \item $\varphi(z) = \varphi_\gamma(z)$ for all $\gamma>0,\, z \in \zer \partial \varphi$; \label{prop:sharedMinimizerProp1}
        \item $\inf \varphi = \inf \varphi_\gamma$ and $\argmin \varphi \subseteq \argmin \varphi_\gamma$ for $\gamma \in (0,1/L_f]$;
        \item $\argmin \varphi = \argmin \varphi_\gamma$ for all $\gamma \in (0, 1/L_f)$.
    \end{enumerate}
    Suppose additionally that $f$ is $\mathcal{C}^2$. Then $\varphi_\gamma$ is $\mathcal{C}^1$ and the gradient of $\varphi_\gamma$ can be written as 
    \begin{equation}\label{eq:gradvarphigamma}
        \nabla \varphi_\gamma(x) = \left(I - \gamma \nabla^2 f(x)\right)R_\gamma(x).
    \end{equation}
    Moreover, if $\gamma \in (0,1/L_f)$, the set of stationary points of $\varphi_\gamma$ equals $\zer \partial \varphi$.
\end{prop}

Assuming that we are able to compute both $\varphi_\gamma$ and $\varphi$, \Cref{prop:sharedMinimizer}(i) allows us to check whether we have converged to a stationary point of $\varphi$. \Cref{alg:minfbe} is a classical forward-backward algorithm for optimizing the nonsmooth composite objective \eqref{eq:varphiStatement}. 
\begin{algorithm}
\caption{MINFBE \cite{Stella2017}}\label{alg:minfbe}
\begin{algorithmic}[1]
\Require $x^0, \gamma_0 > 0, \xi \in (0,1), \beta \in [0,1), k \gets 0$
\If{$R_{\gamma_k}(x^k) = 0$}
    \State stop
\EndIf
\State Choose $d^k$ s.t. $\langle d^k,\ \nabla \varphi_{\gamma_k}(x^k) \rangle \le 0$
\State Choose $\tau_k \ge 0$ and $w^k = x^k +  \tau_k d^k$ s.t. $\varphi_{\gamma_k}(w^k) \le \varphi_{\gamma_k}(x^k)$
\If{$f(T_{\gamma_k}(w^k)) > f(w^k) - \gamma_k \langle \nabla f(w^k),\ R_{\gamma_k} (w^k)\rangle + \frac{(1 - \beta) \gamma_k}{2} \|R_{\gamma_k}(w^k)\|^2$ }
    \State{$\gamma_k \gets \xi \gamma_k$, goto 1}
\EndIf
\State $x^{k+1} \gets T_{\gamma_k}(w^k)$
\State $\gamma_{k+1} \gets \gamma_k$
\State $k \gets k+1$, goto 1

\end{algorithmic}
\end{algorithm}

\Rev{In \Cref{alg:minfbe}, $\xi$ is an Armijo backtracking parameter, while $\beta$ is used to control the strictness of the descent condition in Step 6.} For appropriately chosen $\gamma$, the condition in Step 6 never holds, as stated in the next lemma. Moreover, the step-sizes $\gamma_k$ are bounded below by a constant in terms of $\sigma,\, \beta$ and $L_f$. This guarantees that a step is always possible.
 
\begin{lemma}[{\cite[Lem 3.1]{Stella2017}}]\label{lem:lem3_1}
Let $(\gamma_k)_{k \in \mathbb{N}}$ be the sequence of step-size parameters in \Cref{alg:minfbe}, and let $\gamma_\infty = \min_{i \in \mathbb{N}} \gamma_i$. Then for all $k \ge 0$,
\begin{equation*}
    \gamma_k \ge \gamma_\infty \ge \min \{\gamma_0,\, \xi (1-\beta)/L_f\}.
\end{equation*}
\end{lemma}

The MINFBE algorithm can be interpreted as a descent step (Step 5) followed by a FBS step (Step 9). In particular, note that the descent direction $d^k$ does not have to be the direction of steepest descent, which allows for more flexibility in the algorithm. By combining the two of these steps together, the algorithm achieves global convergence as well as local linear convergence. This algorithm enjoys the following convergence guarantees.
\begin{definition}[Linear and Superlinear Convergence]
    We say a sequence $(x^k)_{k \in \mathbb{N}}$ converges to $x_*$;
    \begin{enumerate}[i.]
        \item \emph{Q-linearly} with factor $\omega \in [0,1)$ if $\|x^{k+1} - x_*\| \le \omega \|x^k - x_*\|$ for all $k \ge 0$;
        \item \emph{Q-superlinearly} if $\|x^{k+1} - x_*\|/ \|x^k - x_*\| \rightarrow 0$. 
    \end{enumerate}
    The convergence is R-linear (R-superlinear) if $\|x^k - x_*\|\le a_k$ for some sequence $(a_k)_{k \in \mathbb{N}}$ s.t. $a_k \rightarrow 0$ Q-linearly (Q-superlinearly).
\end{definition}
\begin{theorem}[{\cite[Thm 3.6, 3.7]{Stella2017}}]\label{thm:convexIterConv}
Suppose that $f$ is convex and that $\varphi$ is coercive. In particular, suppose that the level set $\{x \in \R^n \mid \varphi(x) \le \varphi(x^0)\}$ has diameter $R$, $0<R<\infty$. Then for the sequences generated by \Cref{alg:minfbe}, either $\varphi(x^0) - \inf \varphi \ge R^2 / \gamma_0$ and 
\begin{equation}
    \varphi(x^1) - \inf \varphi \le \frac{R^2}{2 \gamma_0},
\end{equation}
or for any $k \in \mathbb{N}$, it holds that
\begin{equation}
    \varphi(x^k) - \inf \varphi \le \frac{2 R^2}{k \min \{\gamma_0, \xi (1-\beta)/L_f\}}.
\end{equation}

Suppose in addition that $x_*$ is a strong minimizer of $\varphi$, i.e. there exists a neighborhood $N$ of $x_*$ and $c>0$ such that for any $x\in N$, 
\begin{equation*}
\varphi(x) - \varphi(x_*) \ge \frac{c}{2} \|x - x_*\|^2.
\end{equation*}
Then for sufficiently large $k$, $(\varphi(x^k))_{k \in \mathbb{N}}$ and $(\varphi_{\gamma_k}(w^k))_{k \in \mathbb{N}}$ converge Q-linearly to $\varphi(x_*)$ with factor $\omega$, where
\begin{equation*}
    \omega \le \max \left\{ \frac{1}{2}, 1-\frac{c}{4} \min \{\gamma_0, \xi(1-\beta)/L_f\}\right\} \in \left[1/2,1\right),
\end{equation*}
and $(x^k)_{k \in \mathbb{N}}$ converges R-linearly to $x_*$. If $x_*$ is also a strong minimizer of $\varphi_{\gamma_\infty}$ where $\gamma_\infty$ is defined as in \Cref{lem:lem3_1}, then $(\varphi(w^k))_{k \in \mathbb{N}}$ also converges R-linearly to $x_*$.
\end{theorem}
In MINFBE, the initial descent step $w^k$ can be chosen arbitrarily as long as the objective function decreases. Suppose now that the descent direction is chosen using a quasi-Newton method:
\begin{equation*}
    d^k = - B_k^{-1} \nabla \varphi_\gamma(x^k).
\end{equation*}
If $B_k$ are positive definite, then $d^k$ are valid search directions. Assuming that $B_k$ satisfy the Dennis-Moré condition \Rev{\cite{NoceWrig06,dennis1974characterization}}, we can get superlinear convergence of the iterates.

\begin{theorem}[{\cite[Thm 4.1]{Stella2017}}]
    Fix $\gamma>0$. Suppose that $\nabla \varphi_\gamma$ is strictly differentiable at a stationary point $x_* \in \zer \partial \varphi$, and that $\nabla^2 \varphi_\gamma(x_*)$ is nonsingular. Let $(B_k)_{k \in \mathbb{N}}$ be a sequence of nonsingular $\R^{n \times n}$ matrices, and suppose the sequences
    \begin{equation}\label{eq:DennisMore}
        w^k = x^k - B_k^{-1} \nabla \varphi_\gamma(x^k) ,\quad x^{k+1} = T_\gamma(w^k)
    \end{equation}
    converge to $x_*$. If $x^k, w^k \notin \zer \partial \varphi$ for all $k \ge 0$ and the Dennis-Moré condition
    \begin{equation}\label{eq:DennisMore1}
        \lim_{k \rightarrow \infty} \frac{\|(B_k - \nabla^2 \varphi_\gamma(x^k))(w^k - x^k)\|}{\|w^k - x^k\|} = 0
    \end{equation}
    holds, then $(x^k)_{k \in \mathbb{N}}$ and $(w^k)_{k \in \mathbb{N}}$ converge Q-superlinearly to $x_*$.
\end{theorem}

\RRev{If $B_k$ are updated accordingly to the BFGS update step, then the updates as given in the previous theorem converge superlinearly to the minimum, under some additional assumptions on $\varphi$ such as being convex with strong local minimum $x_*$, or satisfying a stronger Kurdyka-Łojasiewicz property at cluster points $\omega(x^0)$ \cite[Thm 4.3]{Stella2017}.} 
Moreover, it can be shown that $\tau_k=1$ is a valid step-size for sufficiently large $k$. For completeness, the BFGS update steps are given as below. Note that it is usually more practical to update the inverse Hessian approximation $H_k = B_k^{-1}$ \cite{NoceWrig06}.
\begin{subequations}
\begin{gather}
    s^k = w^k - x^k, \quad y^k = \nabla \varphi_\gamma(w^k) - \nabla \varphi_\gamma(x^k), \\
    B_{k+1} = \left\{ \begin{matrix*}[l]
        B_k + \frac{y^k y^{k\top}}{y^{k\top}s^k} - \frac{B_k s^k (B_k s^k)^\top}{s^{k\top}B^ks^k} & \text{if } \langle s^k, y^k \rangle > 0, \\
        B_k & \text{otherwise}\\
    \end{matrix*}\right..\\
    H_{k+1} = \left\{ \begin{matrix*}[l]
        \left(I - \frac{s^k y^{k\top}}{y^{k\top} s^k}\right)H_k\left(I - \frac{y^k s^{k\top}}{y^{k\top} s^k}\right) + \frac{s^k s^{k\top}}{y^{k\top} s^k} & \text{if } \langle s^k, y^k \rangle > 0, \\
        H_k & \text{otherwise}\\
    \end{matrix*}\right..
\end{gather}
\end{subequations}

\subsection{Weakly-Convex Extension}
Suppose now that $g$ is not convex, but instead is $M$-weakly convex. \Rev{Recall that a function $g(x)$ is $M$-weakly convex if $g+M\|x\|^2/2$ is convex. For a $M$-weakly convex function $g$, we have for all $x, y$ and $z \in \partial g(y)$ (where $\partial g$ denotes the Clarke subdifferential \RRev{of $g$}),}
\begin{subequations}
\begin{align}
g(x) &\ge g(y) + \langle z, x-y \rangle - \frac{M}{2} \|x-y\|^2,\label{eq:WeakConvIneq} \\
g(tx + (1-t)y) &\le t g(x) + (1-t) g(y) + \frac{M}{2} t(1-t)\|x-y\|^2.
\end{align}
\end{subequations}

\Rev{
In the following \Cref{sec:pnpqn}, we will model the proposed denoiser $D_\sigma = \prox_g$ as the \RRev{proximal operator} of a weakly convex function. In particular, a gradient step denoiser $D_\sigma = I - \nabla g_\sigma$ with contractive $\nabla g_\sigma$ is the \RRev{proximal operator} of a weakly convex function \cite{hurault2023RelaxedProxDenoiser}. We can extend the classical convex analysis to this case as well, albeit with a smaller allowed $\gamma$.}

To transfer the results from the previous section to the case where $g$ is weakly convex, we are required to check that the function values at the MINFBE iterates are non-increasing. As we will show in the following proposition, this is still the case for sufficiently small $\gamma$. Many properties of the forward-backward envelope still hold, and we are still able to attain global convergence and superlinear local convergence, subject to the Dennis-Mor\'e condition \eqref{eq:DennisMore1}. 
\begin{prop} \label{prop:varphiIneqs}
    For all $x \in \R^n$, $\gamma>0$,
    \begin{enumerate}[i.]
        \item $\varphi_\gamma(x) \le \varphi(x)- \frac{\gamma - M \gamma^2}{2}\|R_\gamma(x)\|^2$;
        \item $\varphi(T_\gamma(x)) \le \varphi_\gamma(x) - \frac{\gamma}{2}(1-\gamma L_f) \|R_\gamma(x)\|^2$ for all $\gamma > 0$;
        \item $\varphi(T_\gamma(x)) \le \varphi_\gamma(x)$ for all $\gamma \in (0, 1/L_f]$.
    \end{enumerate}
\end{prop}
\begin{proof}
\textbf{(i). } By the optimality condition in \eqref{eq:TgammaDef}, we have 
\begin{equation*}
    R_\gamma(x) - \nabla f(x) \in \partial g(T_\gamma(x)).
\end{equation*}
By \eqref{eq:WeakConvIneq}, we have
\begin{align*}
    g(x) &\ge g(T_\gamma(x)) + \langle R_\gamma(x) - \nabla f(x), x - T_\gamma(x) \rangle - \frac{M}{2}\|x-T_\gamma(x)\|^2\\
    &= g(T_\gamma(x))- \gamma \langle \nabla f(x), R_\gamma(x) \rangle + \gamma \|R_\gamma(x)\|^2 - \frac{M \gamma^2}{2} \|R_\gamma(x)\|^2.
\end{align*}
Adding $f(x)$ to both sides and applying \eqref{eq:phigamma1} gives the result.

\textbf{(ii), (iii). } The proof is identical to that in \cite[Prop 2.2]{Stella2017}, requiring only the Lipschitz convexity of $\nabla f$. 
\end{proof}

\begin{prop}\label{prop:PhiGammaEquality}
    Suppose $\gamma - M \gamma^2 \ge 0$, or equivalently $\gamma \in [0,1/M]$. Then the following hold:
    \begin{enumerate}[i.]
        \item $\varphi_\gamma(z) = \varphi(z)$ for all $z \in \zer \partial \varphi$;
        \item $\inf \varphi = \inf \varphi_\gamma$ and $\argmin \varphi \subseteq \argmin \varphi_\gamma$ for $\gamma \in (0, 1/L_f]$;
        \item $\argmin \varphi = \argmin \varphi_\gamma$ for $\gamma \in (0, 1/L_f)$.
    \end{enumerate}
\end{prop}
\begin{proof}
    \textbf{(i). } \Cref{prop:varphiIneqs}(i) combined with the condition $\gamma-M\gamma^2 \ge 0$ shows $\varphi_\gamma(x) \le \varphi(x)$. If $z \in \zer \partial \varphi$, then $z = T_\gamma(z)$, and \Cref{prop:varphiIneqs}(ii) reads $\varphi(z) \le \varphi_\gamma(z)$.

    \textbf{(ii), (iii). } Identical to \cite[Prop 2.3]{Stella2017}. 
\end{proof}

With weakly convex functions, we are still able to provide a lower bound on the $\gamma$ such that the condition in Step 6 of \Cref{alg:minfbe} does not hold, \RRev{removing the need to reduce step-sizes}. The proof relies only on the Lipschitz constant of $\nabla f$ and does not require convexity of $g$. However, we require that $\gamma - M\gamma^2 \ge 0$. In practice, the denoisers we use have $M<1/2$, which allows for any $\gamma \in (0,1)$. 

\begin{lemma}\label{lem:gammaBddBelowWeakConv}
    Suppose $g$ is weakly convex. If $0<\gamma<\min \{(1-\beta)/L_f, 1/M\}$, then the condition in Step 6 in \Cref{alg:minfbe} never holds. Moreover, this implies MINFBE iterations satisfy $\gamma_k \ge \gamma_\infty \ge \min\{\gamma_0, \xi(1-\beta)/L_f, 1/M\} > 0$ for all $k$.
\end{lemma}
\begin{proof}
    Suppose $0<\gamma < \min \{(1-\beta)/L_f, 1/M\}$, and for contradiction that the condition in Step 6 holds. Then there exists some $w$ such that 
    \[f(T_\gamma(w)) > f(w) - \gamma \langle \nabla f(w), R_\gamma(w)\rangle + \frac{(1-\beta) \gamma}{2}\|R_\gamma(w^k)\|^2.\]
    Adding $g(T_\gamma(w))$ to both sides and considering \eqref{eq:phigamma1}, this becomes
    \[\varphi(T_\gamma(w)) > \varphi_\gamma(w) - \frac{\beta \gamma}{2}\|R_\gamma(w)\|^2.\]
    But from \Cref{prop:varphiIneqs}(ii), we also have
    \begin{align*}
        \varphi(T_\gamma(w)) &\le \varphi_\gamma(w) - \frac{\gamma}{2}(1-\gamma L_f) \|R_\gamma(w)\|^2 \\
        &\le \varphi_\gamma(w) - \frac{\beta \gamma}{2}\|R_\gamma(w)\|^2,
    \end{align*}
    where the second inequality follows from $\gamma < (1-\beta)/L_f$, giving a contradiction. The second part holds since $(\gamma_k)_{k \in \mathbb{N}}$ is a non-increasing sequence.
\end{proof}
\begin{remark}
While $\gamma<1/M$ is not strictly needed for the proof of the above lemma, this requirement is needed for convergence in future results.
\end{remark}

The following theorem characterizes the convergence of the functional $\varphi$, which relies on the non-increasing condition of Step 5 in \Cref{alg:minfbe}. This is an analogue of \cite[Prop 3.4]{Stella2017}.
\begin{theorem}\label{thm:phiDecreases}
    Suppose $0<\gamma_0<1/M$. Then the MINFBE iterations satisfy the following:
    \begin{enumerate}[i.]
        \item $\varphi(x^{k+1}) \le \varphi(x^k) - \frac{\beta \gamma_k}{2} \|R_{\gamma_k}(w^k)\|^2 - \frac{\gamma_k - M \gamma_k^2}{2} \|R_{\gamma_k}(x^k)\|^2$;
        \item Either the sequence $\|R_{\gamma_k}(x^k)\|$ is square-summable, or $\varphi(x^k) \rightarrow \inf \varphi = -\infty$ and the set $\omega (x^0)$ of cluster points of the sequence $(x^k)_{k \in \mathbb{N}}$ is empty.
        \item $\omega(x^0) \subseteq \zer \partial \varphi$;
        \item If $\beta >0$, then either the sequence $\|R_{\gamma_k}(w^k)\|$ is square-summable and every cluster point of $(w^k)_{k \in \mathbb{N}}$ is critical, or $\varphi_{\gamma_k}(w^k) \rightarrow \inf \varphi = -\infty$ and $(w^k)_{k \in \mathbb{N}}$ has no cluster points.
    \end{enumerate}
\end{theorem}
\begin{proof}
    \textbf{(i). } Recalling $x^{k+1} = T_{\gamma_k}(w^k)$,
    \begin{align}
        \varphi(x^{k+1}) &\le \varphi_{\gamma_k}(w^{k}) - \frac{\beta \gamma_k}{2} \|R_{\gamma_k}(w^k)\|^2 \notag \\
        &\le \varphi_{\gamma_k}(x^{k})- \frac{\beta \gamma_k}{2} \|R_{\gamma_k}(w^k)\|^2 \label{eq:varphiStepPhigamma}\\
        &\le \varphi(x^k) - \frac{\beta \gamma_k}{2} \|R_{\gamma_k}(w^k)\|^2 - \frac{\gamma_k - M \gamma_k^2}{2} \|R_{\gamma_k}(x^k)\|^2, \label{eq:perstepDecrease}
    \end{align}
    where the first and second inequalities come from Step 6 and 5 in \Cref{alg:minfbe} respectively, and the final inequality is \Cref{prop:varphiIneqs}(i). 

    \textbf{(ii)-(iv). } We follow \cite{Stella2017} with minor modifications. Let $\varphi_* = \lim_{k \rightarrow \infty} \varphi(x^k)$, which exists as $(\varphi(x^k))_{k \in \mathbb{N}}$ is monotone by (i) and $\gamma_k - M\gamma_k^2\ge 0$. If $\varphi_* = -\infty$, then $\inf \varphi = -\infty$. By properness and lower semi-continuity of $\varphi$, as well as the monotonicity of $\varphi(x^k)$, no cluster points of $(x^k)_{k \in \mathbb{N}}$ exist. If instead $\varphi_* > -\infty$, by telescoping \eqref{eq:perstepDecrease},
    \begin{equation}\label{eq:telescoping}
        \frac{1}{2} \sum_{i=0}^k \gamma_i \left(\beta \|R_{\gamma_i}(w^i)\|^2 + (1-\gamma_i M)\|R_{\gamma_i}(x^i)\|^2\right) \le \varphi(x^0) - \varphi(x^{k+1}) \le \varphi(x^0) - \varphi_*.
    \end{equation}
    Since $\gamma_k$ is uniformly bounded below by \Cref{lem:gammaBddBelowWeakConv}, we have square summability of $\|R_{\gamma_k}(x^k)\|$, showing (ii). 

    By square summability, $R_{\gamma_k}(x^k) \rightarrow 0$. Moreover, the functions $R_{\gamma_k} = R_{\gamma_\infty}$ are constant for sufficiently large $k$, and $R_{\gamma_\infty}$ is continuous by continuity of the proximal operator and of $\nabla f$. Therefore, any cluster point $z \in \omega(x^k)$ has $R_{\gamma_\infty}(x^{k_j}) \rightarrow R_{\gamma_\infty}(z) = 0$ for some subsequence $x^{k_j} \rightarrow z$. Thus $z = T_{\gamma_\infty}(z) \Rightarrow z \in \zer \partial \varphi$, showing (iii).
    
    

    If $\beta>0$, for sufficiently large $k$ such that $\gamma_k = \gamma_\infty$, the following chain of inequalities holds:
    \begin{equation}
        \varphi_{\gamma_k}(w^{k+1}) \le \varphi_{\gamma_k}(x^{k+1}) = \varphi_{\gamma_k}(T_k(w^k)) \le \varphi_{\gamma_k}(w^k).
    \end{equation}
    The first inequality comes from Step 5, the equality from Step 9, and the final inequality from \Cref{prop:varphiIneqs}. The monotonicity of $\varphi_{\gamma_k}(w^k)$ for sufficiently large $k$ allows for a similar argument to hold for the $w^k$ sequence, giving (iv).
\end{proof}
Convergence results can also be extended to the weakly convex case. In particular, the following theorem shows the convergence of the residuals between each step.

\begin{theorem}[Global Residual Convergence] \label{thm:residualConvergence}
    Suppose $0<\gamma_0\le1/(2M)$, and let $c = \min\{\gamma_0, \xi(1-\beta)/L_f, 1/M\}>0$ be the lower bound for $\gamma_\infty$. The MINFBE iterations satisfy
    \begin{equation}
        \min_{i\le k} \|R_{\gamma_i}(x^i)\|^2 \le \frac{2}{k+1} \frac{\varphi(x^0) - \inf \varphi}{c - Mc^2}.
    \end{equation}
    If in addition $\beta >0$, then we also have
    \begin{equation}
        \min_{i\le k} \|R_{\gamma_i}(w^i)\|^2 \le \frac{2}{k+1} \frac{\varphi(x^0) - \inf \varphi}{\beta c}.
    \end{equation}
\end{theorem}
\begin{proof}
    As in \cite[Thm 3.5]{Stella2017}. If $\inf \varphi = -\infty$, there is nothing to prove, so suppose otherwise that $\inf \varphi > -\infty$. Considering \eqref{eq:telescoping} along with $(\gamma_k)_{k \in \mathbb{N}}$ being nonincreasing implies
    \begin{equation}
        \frac{(k+1)(\gamma_k - M\gamma_k^2)}{2} \min_{i\le k}\|R_{\gamma_i}(x^i)\|^2 + \frac{(k+1)\beta \gamma_k}{2} \min_{i\le k}\|R_{\gamma_i}(w^i)\|^2\le \varphi(x^0) - \inf \varphi.
    \end{equation}
    Now note that $\gamma - M \gamma^2$ is increasing for $\gamma < 1/(2M)$, so $\gamma_k - M \gamma_k^2$ is lower bounded by $c-Mc^2>0$. Rearranging yields both inequalities.
\end{proof}

To obtain convergence of the objective similar to \Cref{thm:convexIterConv}, it is insufficient for $g$ to be weakly convex. We can alternatively utilize the KL property, which is a useful and general property satisfied by a large class of functions, including semialgebraic functions \cite{attouch2010proximal}. Moreover, it can be used to show convergence in the absence of other regularity conditions such as convexity \cite{Attouch2013, Bolte2014, hurault2022proximal}. 

\begin{definition}[KL Property \cite{Attouch2013,Bolte2014}]
    Suppose $\varphi:\R^n \rightarrow \overline{\R}$ is proper and lower semi-continuous. $\varphi$ satisfies the \emph{Kurdyka-Łojasiewicz (KL) property} at a point $x_*$ in $\dom \partial \varphi$ if there exists $\eta \in (0,+\infty]$, a neighborhood $U$ of $x_*$ and a continuous concave function $\Psi:[0,\eta) \rightarrow [0,+\infty)$ such that:
    \begin{enumerate}
        \item $\Psi(0) = 0$;
        \item $\Psi$ is $\mathcal{C}^1$ on $(0,\eta)$;
        \item $\Psi'(s)>0$ for $s \in (0,\eta)$;
        \item For all $u \in U \cap \{\varphi(x_*) < \varphi(u) < \varphi(x_*) + \eta\}$, we have
        \[\varphi'(\varphi(u) - \varphi(x_*)) \dist (0, \partial \varphi(u)) \ge 1.\]
    \end{enumerate}
    We say that $\varphi$ is a KL function if the KL property is satisfied at every point of $\dom \partial \varphi$.
\end{definition}
Utilizing the KL property, we are able to show that the iterates generated by MINFBE are sufficiently well-behaved, and hence converge. Moreover, from \Cref{thm:phiDecreases}, we have that the iterates converge to critical points of the non-convex objective $\varphi$. Under the PnP scheme, this will correspond to convergence to critical points of some function determined by the denoiser.


\begin{theorem}\label{thm:KLConvergence}
    Suppose that $f$ satisfies the KL condition and $g$ is semialgebraic, and both $f$ and $g$ are bounded from below. Suppose further that there exist constants $\bar\tau, c>0$ such that $\tau_k<\bar\tau$ and $\|d^k\| \le c \|R_{\gamma_k}(x^k)\|$, $\beta>0$, and that $\varphi$ is coercive or has compact level sets. Then the sequence of iterates $(x^k)_{k \in \mathbb{N}}$ is either finite and ends with $R_{\gamma_k}(x^k)=0$, or converges to a critical point of $\varphi$.
\end{theorem}
\begin{proof}
Deferred to the supplementary material. The proof is very similar to that in \cite[Thm 3.9, Appendix 4]{Stella2017}.
\end{proof}

The crux of using the MINFBE method is that we are able to incorporate Newton-type steps into the iterations. Since we are able to get convergence to a critical point from the previous theorem, we are in a position to apply the next theorem to show superlinear convergence in a neighborhood of a minimizer. 

\begin{theorem}\label{thm:Superlinear}
    Suppose that $f$ is continuously differentiable with $L_f$-Lipschitz gradient and $g$ is $M$-weakly convex. Let $\gamma = \gamma_\infty$ as in \Cref{lem:gammaBddBelowWeakConv}. Suppose the search directions are chosen as 
    \[d^k = -B_k^{-1} \nabla \varphi_\gamma(x^k),\]
    the step-sizes in Step 5 are chosen with $\tau_k = 1$ tried first, and $B_k$ satisfy the Dennis-Moré condition \eqref{eq:DennisMore}. Suppose further that the iterates $(x^k)_{k \in \mathbb{N}},\, (w^k)_{k \in \mathbb{N}}$ converge to a critical point $x_*$ at which $\nabla \varphi_\gamma$ is continuously differentiable with $\nabla^2 \varphi_\gamma (x_*) \succ 0$. Then $(x^k)_{k \in \mathbb{N}}$ and $(w^k)_{k \in \mathbb{N}}$ converge Q-superlinearly to $x_*$.
\end{theorem} 
\begin{proof}
    The proof is nearly identical to \cite[Thm 4.1]{Stella2017}. If $\gamma_g$ is $M$-weakly convex, then for $\gamma < 1/M$, $u\mapsto\left(g(u) + \frac{1}{2\gamma}\|u-x\|^2 \right)$ 
    is strongly convex. Thus $\prox_{\gamma g}$ is 1-Lipschitz \cite{Rockafellar1972CA}. The rest of the proofs of Thm 4.1 and 4.2 of \cite{Stella2017} follows as usual. 
\end{proof}
This shows superlinear convergence instead of linear convergence in the case where the critical point is a strong local minimum, i.e. it is locally strongly convex. Note the differentiability condition in the second part can be dropped if $f$ and $g$ are both $\mathcal{C}^2$. Moreover, assuming either $\varphi$ is convex and $x_*$ is a strong local minimum, or $\varphi$ satisfies a stronger KL inequality, these conditions indeed hold \RRev{if $B_k$ is updated according to the BFGS scheme} \cite[Thm 4.3]{Stella2017}.

\section{PnP-qN: Deep Denoiser Extension}\label{sec:pnpqn}
To convert \Cref{alg:minfbe} to the PnP framework, we consider replacing the proximal step in \eqref{eq:TgammaDef} with a denoiser. In particular, we consider the gradient-step denoiser setup in \cite{hurault2022proximal}. Let the denoiser $D_\sigma$ be given by
\begin{subequations} \label{eq:DenoiserFormulation}
\begin{gather}
    D_\sigma = I - \nabla g_\sigma, \label{eq:dSigmaNN}\\
    g_\sigma = \frac{1}{2} \|x - N_\sigma(x) \|^2, \label{eq:gSigmaNN}
\end{gather}
\end{subequations}
where $g_\sigma$ is a $\mathcal{C}^2$ function with $L$-Lipschitz gradient with $L<1$. Note the subscript in $g_\sigma$ represents a denoising strength, as opposed to the forward-backward envelope of $g$ as we will define for our problem later. The mapping $N_\sigma(x)$ takes the form of a $\mathcal{C}^2$ neural network, allowing for the computation of $g_\sigma$ explicitly. Under these assumptions, the denoiser $D_\sigma$ takes the form of a proximal mapping of a weakly convex function, as stated in the next proposition. 


\begin{prop}[{\cite[Prop 1]{hurault2023RelaxedProxDenoiser}}]\label{prop:defPhiSigma}
$D_\sigma(x) = \prox_{\phi_\sigma}(x)$, where $\phi_\sigma$ is defined by 
\begin{equation}\label{eq:phiSigmaClosedForm}
    \phi_\sigma (x) = g_\sigma (D_\sigma^{-1}(x)) - \frac{1}{2} \|D_\sigma^{-1}(x) - x\|^2
\end{equation}
if $x \in \mathrm{Im}(D_\sigma)$, and $\phi_\sigma(x) = +\infty$ otherwise. Moreover, $\phi_\sigma$ is $\frac{L}{L+1}$-weakly convex.
\end{prop}
\RRev{This proposition allows us to take the weak convexity constant required in the previous section as $M=L/(L+1)$}. Since $L<1$, we have $M<1/2$. \RRev{This result can be thought of a slight extension of the fact that a function $f$ is a proximal operator of some proper convex l.s.c. function $\varphi$, if and only if it is a subgradient of a convex l.s.c. function $\psi$ and $f$ is nonexpansive \cite{gribonval2020characterization,moreau1965proximite}.}

Suppose that $\gamma_k = \gamma>0$ is fixed in the MINFBE iterations, satisfying the conditions in \Cref{lem:gammaBddBelowWeakConv}. Consider making the substitution with $\phi_\sigma$ defined as in \Cref{prop:defPhiSigma}, targeting $\varphi = f+g$:
\begin{equation}\label{eq:gPhiSub}
    \gamma g = \phi_\sigma.
\end{equation}
The FBS step $T_\gamma (x)= \prox_{\gamma g} (x - \gamma \nabla f(x))$ thus becomes, \RRev{using $D_\sigma = \prox_{\phi_\sigma}$,}
\begin{equation}
    T_\gamma(x) = D_\sigma (x - \gamma \nabla f(x)).
\end{equation}
This will target the objective function $\varphi(x) = f(x) + g(x) = f(x) + \phi_\sigma(x)/\gamma$. To iterate \Cref{alg:minfbe} with this substitution, we need to evaluate $\varphi_\gamma$. Recalling \eqref{eq:phigamma2}, we can instead evaluate the Moreau envelope $g^\gamma$. By definition \eqref{eq:gGammaDef2} and the substitution \eqref{eq:gPhiSub}, we have:
\begin{align*}
        g^\gamma(y) \overset{\eqref{eq:gGammaDef2}}&{=} g( \prox_{\gamma g}(y)) + \frac{1}{2\gamma}\|\prox_{\gamma g}(y) - y\|^2 \\
        \overset{\eqref{eq:gPhiSub}}&{=} \frac{1}{\gamma} \phi_\sigma(D_\sigma (y)) + \frac{1}{2\gamma}\|D_\sigma(y) - y\|^2 \\
        \overset{\eqref{eq:phiSigmaClosedForm}}&{=} \frac{1}{\gamma} g_\sigma (D^{-1}_\sigma(D_\sigma(y))) - \frac{1}{2\gamma}\|D^{-1}_\sigma(D_\sigma(y)) - D_\sigma(y)\|^2  + \frac{1}{2\gamma} \|D_\sigma(y) - y\|^2 \\
        &= \frac{1}{\gamma}g_\sigma(y).
\end{align*}
Using this substitution, we obtain the Plug-and-Play scheme PnP-MINFBE, detailed in \Cref{alg:pnpminfbe}. We have a closed form for the forward-backward envelope of $\varphi$, as well as some other expressions essential for iterating MINFBE, given by:
\begin{subequations}\label{eqs:pnpMINFBEclosedform}
\begin{gather}
    \varphi(x) = f(x) + \frac{1}{\gamma} \phi_\sigma(x), \label{eq:TrueVarphi}\\
    \varphi_\gamma(x) = f(x) - \frac{\gamma}{2} \|\nabla f(x)\|^2 + \frac{1}{\gamma} g_\sigma(x - \gamma \nabla f(x)), \label{eq:TrueVarphiGamma} \\
    \nabla \varphi_\gamma (x) = (I - \gamma \nabla^2 f) R_\gamma(x), \label{eq:gradVarphiGamma2} \\
    \varphi(x^{k+1}) = f(x^{k+1}) + \frac{1}{\gamma} \left(g_\sigma(w^k - \gamma \nabla f(w^k)) - \|w^k - \gamma \nabla f(w^k) - T_\gamma(w^k)\|^2/2\right). \label{eq:lastGather}
\end{gather}
\end{subequations}

\begin{algorithm}
\caption{PnP-MINFBE}\label{alg:pnpminfbe}
\begin{algorithmic}[1]
\Require $x^0, \gamma <\min\{\gamma_0, (1-\beta)/L_f, 1/M\}, \beta \in [0,1), k \gets 0$
\If{$R_{\gamma_k}(x^k) = 0$}
    \State stop
\EndIf
\State Choose $d^k$ s.t. $\langle d^k,\ \nabla \varphi_{\gamma}(x^k) \rangle \le 0$
\State Choose $\tau_k \ge 0$ and $w^k = x^k +  \tau_k d^k$ s.t. $\varphi_{\gamma}(w^k) \le \varphi_{\gamma}(x^k)$
\State $x^{k+1} \gets D_\sigma(w^k - \gamma \nabla f(w^k))$
\State $k \gets k+1$, goto 1
\end{algorithmic}
\end{algorithm}
To compute the search direction $d^k$ at each step, we can use a quasi-Newton method to approximate the inverse Hessian of $\varphi_\gamma$. While a closed form exists for $\nabla^2 \varphi_\gamma$, \RRev{such as in \cite[Thm 2.10]{Stella2017}}, it requires the Jacobian of the denoiser $D_\sigma$, rendering methods requiring the Hessian computationally intractable due to the dimensionality of our problems. Therefore, we resort to a BFGS-like algorithm using the differences and secants
\[s^k = w^k - x^k,\ y^k = \nabla \varphi_\gamma (w^k) - \nabla \varphi_\gamma(x^k).\]
In particular, we will use the L-BFGS method due to the memory restrictions imposed by using images for our experiments. This can be implemented using a two-loop recursion, using only the last $m$ secants computed \cite{NoceWrig06}. We additionally impose a safeguard to reject updating the Hessian approximation if the secant condition $\langle s^k, y^k\rangle > 0$ is not satisfied. For completeness, we write the two-loop recursion for L-BFGS in \Cref{alg:L-BFGS}. The initial (inverse) Hessian approximations are chosen as $H^k_0 = c_k I$ as in \cite{NoceWrig06}, given by
\[c_k = \frac{\langle s^{k-1}, y^{k-1} \rangle}{\langle y^{k-1}, y^{k-1} \rangle}.\]

\begin{algorithm}
    \caption{L-BFGS \cite{NoceWrig06}}\label{alg:L-BFGS}
    \begin{algorithmic}[1]
        \Require $m>0$, secants $(s^i)^{k-1}_{i=k-m}$, differences $(y^i)^{k-1}_{i=k-m}$, initial Hessian guesses $(H^k_0)_{k \in \mathbb{N}}$
        \State $q \gets \nabla \varphi_\gamma (x^k)$
        \State $\rho_i \gets 1/\langle y^i, s^i \rangle$ for $i = k-1, k-2, ..., k-m$
        \For {$i = k-1, k-2, ..., k-m$}
            \State $\alpha_i \gets \rho_i \langle s^i, q\rangle$
            \State $q \gets q - \alpha_i y^i$
        \EndFor
        \State $r \gets H^k_0 q$
        \For {$i = k-m, k-m+1, ...., k-1$}
            \State $\beta \gets \rho_i \langle y^i, r\rangle$
            \State $r \gets r + (\alpha_i - \beta)s^i$
        \EndFor
        \State \textbf{stop with } $B_k^{-1} \nabla \varphi_\gamma(x^k) = H^k \nabla \varphi_\gamma(x^k) = r$
    \end{algorithmic}
\end{algorithm}

\begin{algorithm}
\caption{PnP-LBFGS}\label{alg:pnpLBFGS}
\begin{algorithmic}[1]
\Require $x^0, \gamma <\min\{(1-\beta)/L_f, 1/M\}, \beta \in [0,1), k \gets 0$
\If{$R_{\gamma_k}(x^k) = 0$}
    \State stop
\EndIf
\State Compute $d^k \gets -B_k^{-1} \nabla \varphi_\gamma(x^k)$ using L-BFGS (c.f. Algorithm \ref{alg:L-BFGS}) with differences and secants $(s^i,y^i)^{k-1}_{i=k-m}$.
\State Choose $\tau_k \in [0,1]$ and $w^k = x^k +  \tau_k d^k$ s.t. $\varphi_{\gamma}(w^k) \le \varphi_{\gamma}(x^k)$
\State $x^{k+1} \gets D_\sigma(w^k - \gamma \nabla f(w^k))$
\State $s^k \gets w^k - x^k,\ y^k \gets \nabla \varphi_\gamma(w^k) - \nabla \varphi_\gamma(x^k)$
\State $k \gets k+1$, goto 1
\end{algorithmic}
\end{algorithm}

Utilizing the results from the previous section, we can show the following convergence results for \RRev{PnP-MINFBE (\Cref{alg:pnpminfbe}) and PnP-LBFGS (\Cref{alg:pnpLBFGS})}.

\begin{corollary}\label{cor:PnPMINFBEConvergence}
    Suppose that $f$ is $\mathcal{C}^1$ and KL with $L_f$-Lipschitz gradient, $g_\sigma$ is $\mathcal{C}^2$ and semi-algebraic with $L_g$-Lipschitz gradient with $L_g<1$. Assume further that $\gamma<1/(2M)$ is chosen as in \Cref{lem:gammaBddBelowWeakConv} such that $\gamma = \gamma_\infty$, and there exist $\bar \tau, c>0$ such that $\tau_k \le \bar \tau$ and $\|d^k\| \le c \|R_\gamma(x^k)\|$. Then the PnP-MINFBE iterations of \Cref{alg:pnpminfbe} satisfy the following:

    \begin{enumerate}[i.]
        \item $\varphi(x^k)$ decreases monotonically;
        \item The residuals $R_\gamma(x^k)$ converge to zero at a rate $\mathcal{O}(1/\sqrt{k})$;
        \item If the iterates are bounded, then the iterates are either finite or converge to a critical point of $\varphi = f + \frac{1}{\gamma} \phi_\sigma$. Moreover, $\varphi = \varphi_\gamma$ at these critical points.
        \item If furthermore $d^k = -B_k^{-1} \nabla \varphi_\gamma(x^k)$ and the $B_k$ satisfy the Dennis-Moré condition \eqref{eq:DennisMore}, then the $x^k$ and $w^k$ converge superlinearly to $x_*$.
    \end{enumerate}
\end{corollary}
\begin{proof}
    \textbf{(i), (ii). } 
    Follows from \Cref{thm:phiDecreases,thm:residualConvergence}.
    \textbf{(iii). } By the Tarski-Siedenberg theorem \cite{Attouch2013}, compositions and inverses of semi-algebraic mappings are semi-algebraic. Therefore $D_\sigma$ and $D_\sigma^{-1}$ are semi-algebraic (on their domain), and hence so is $\phi_\sigma$. Therefore, 
    \begin{equation*}
        \varphi = f + \frac{1}{\gamma}\phi_\sigma
    \end{equation*}
    is a KL function. Moreover, $\varphi_\gamma$ is also a KL function. So we have convergence by \Cref{thm:KLConvergence}. The final part follows from \Cref{prop:PhiGammaEquality}.
    \textbf{(iv). } Follows from \Cref{thm:Superlinear}.
\end{proof}
\begin{remark}\label{rmk:constantTau}
    An essential part of the classical proof relies on the fact that $\tau=1$ will eventually always be accepted in MINFBE, under a Newton-type descent direction choice. \Rev{During numerical testing, we observed that the Armijo search for $\tau$ was only occasionally necessary when the image is being optimized, with at most 10 line searches required before converging.} 
\end{remark}
In our case, $f$ will be a quadratic fidelity term of the form $f(x) = \|Ax - y\|^2/2$ for some linear operator $A$ and measurement $y$. This is semi-algebraic and hence KL, and moreover trivially bounded below. From \eqref{eq:gSigmaNN}, we additionally have that $g_\sigma$ is bounded below. Since $N_\sigma$ will take the form of a neural network which is a composition of semi-algebraic operations and arithmetic operations, $g_\sigma$ will also be semi-algebraic. Therefore, we can apply \Cref{cor:PnPMINFBEConvergence} and get convergence to critical points \RRev{of the associated function $\varphi = f + \frac{1}{\gamma} \phi_\sigma$}.
\begin{table}[]
\centering
\caption{Hyperparameters for PnP-LBFGS.}
\label{tab:hparams}
\begin{tabular}{@{}ccccccc@{}}
\toprule
\multicolumn{1}{l}{} & \multicolumn{3}{c}{Deblur} & \multicolumn{3}{c}{SR} \\ \cmidrule(l){2-7} 
$\sigma$         & 2.55    & 7.65   & 12.75   & 2.55  & 7.65  & 12.75  \\ \midrule
$\alpha$             & 0.5 & 0.5 & 0.7 & 0.5 & 0.5 & 0.5                               \\
$\gamma$             & \multicolumn{6}{c}{1}                               \\
$\beta$              & \multicolumn{6}{c}{0.01}                            \\
$\lambda$            & 1       & 1      & 1       & 4     & 1.5   & 1      \\
$\sigma_d / \sigma$  & 1       & 0.75      & 0.75    & 2     & 1     & 0.75   \\ \bottomrule
\end{tabular}
\end{table}
\begin{table}[]
\centering
\caption{Hyperparameters for PnP-$\hat\alpha$PGD.}
\begin{tabular}{@{}ccccccc@{}}
\toprule
                  & \multicolumn{3}{c}{Deblur} & \multicolumn{3}{c}{SR} \\ \cmidrule(l){2-7} 
$\sigma$          & 2.55    & 7.65   & 12.75   & 2.55  & 7.65  & 12.75  \\ \midrule
$\alpha$ & 0.6 & 0.8 & 0.85 & 1 & 1 & 1 \\
$L_f$ & \multicolumn{3}{c}{1} & \multicolumn{3}{c}{0.25} \\
$\lambda$ & \multicolumn{6}{c}{$(\alpha+1)/(\alpha L_f)$} \\
$\hat\alpha$ &\multicolumn{6}{c}{1/($\lambda L_f$)} \\ 
$\sigma_d/\sigma$ & 1.5     & 1      & 1       & 2     & 2     & 2      \\ \bottomrule
\end{tabular}
\label{tab:R1_aPGD_hparams}
\end{table}

\section{Experiments}\label{sec:experiments}
In this section, we consider the application of the proposed PnP-LBFGS method, \RRev{given by \Cref{alg:pnpLBFGS}}, with a pre-trained denoiser to image deblurring and super-resolution. We use the pretrained Lipschitz-constrained proximal denoiser given in \cite{hurault2022proximal}. The (gradient-step) denoiser takes the form \eqref{eq:DenoiserFormulation}, where $N_\sigma$ is a neural network based on the DRUNet architecture \cite{zhang2021plug}. The Lipschitz constraint on $\nabla g_\sigma$ is enforced by applying a penalty on the spectral norm of $\nabla^2 g_\sigma$ during training. \RRev{While this spectral constraint affects the performance of the end-to-end denoiser, it provides sufficient conditions for convergence in the context of PnP, in particular, convergence to a critical point of a closed-form functional.}


The datasets we consider for image reconstruction are the CBSD68, CBSD10 and set3c datasets\footnote{\url{https://www2.eecs.berkeley.edu/Research/Projects/CS/vision/bsds/}}, containing images of size $256\times256$ with three color channels and pixel intensity values in $[0,255]$ \cite{martin2001database}. The forward operators corresponding to \RRev{the considered reconstruction problems of deblurring and super-resolution} are linear, and we can write the fidelity term as $f(x) = \lambda \|Ax-y\|^2/2$, where $A$ is the degradation operator, $y$ is the degraded image, and $\lambda$ is a regularization parameter. For reconstruction, $y$ will be taken as $y = Ax_{\text{true}} + \varepsilon$, where $x_{\text{true}}$ is the ground-truth image and the noise $\varepsilon$ is pixel-wise Gaussian with standard deviations $\sigma \in \{2.55, 7.65, 12.75\}$ corresponding to 1\%, 3\%, and 5\% noise (relative to the maximum pixel intensity value), respectively. The underlying optimization problems corresponding to fixed points of PnP-MINFBE thus take the form (as in \eqref{eq:TrueVarphi}):
\begin{equation}
    \min_x \varphi(x) = \frac{\lambda}{2} \|Ax - y\|^2 + \frac{1}{\gamma} \phi_\sigma,
\end{equation}
where $\gamma \le \min\{(1-\beta)/L_f, 1/2M\}$ as in \Cref{lem:gammaBddBelowWeakConv,thm:residualConvergence}. In this case, $f$ is $\mathcal{C}^2$, and we can easily compute the derivative \RRev{of the forward-backward envelope using} \eqref{eq:gradVarphiGamma2}.

The methods we compare against are PnP methods with similar convergence guarantees, namely $\mathcal{O}(1/\sqrt{k})$ residual convergence and a KL-type iterate convergence \cite{hurault2022proximal}. \RRev{Our analysis additionally shows} superlinear convergence to minima with positive-definite Hessian using Newton's directions. Although we can not verify whether the Hessian approximation $B_k$ obtained via L-BFGS satisfies the Dennis-Moré condition for superlinear convergence, we will empirically demonstrate faster convergence in terms of both time and iteration count compared to the competing methods.

 

The PnP methods that we will compare against are the \Rev{PnP-PGD, PnP-DRS, PnP-DRSdiff and PnP-$\hat\alpha$PGD methods \cite{hurault2022proximal,hurault2023RelaxedProxDenoiser}}. Here PGD stands for proximal gradient descent, DRS for Douglas-Rachford splitting, DRSdiff for DRS with differentiable fidelity terms, and $\hat\alpha$PGD for $\hat\alpha$-relaxed PGD. The update rules corresponding to the chosen PnP methods for comparison are as follows: 

\begin{align}
    &\left\{ \begin{matrix*}[l]
        z^{k+1} = x^k - \lambda \nabla f(x^k) \\
        x^{k+1} = D_\sigma(z^{k+1})
    \end{matrix*}\right. \label{eq:PnP-PGD}\tag{PnP-PGD} \\
    &\left\{ \begin{matrix*}[l]
        y^{k+1} = \prox_{\lambda f}(x^k) \\
        z^{k+1} = D_\sigma(2y^{k+1} - x^k) \\
        x^{k+1} = x^k + (z^{k+1} - y^{k+1})
    \end{matrix*}\right. \label{eq:PnP-DRSdiff}\tag{PnP-DRSdiff} \\
    &\left\{ \begin{matrix*}[l]
        y^{k+1} = D_\sigma(x^k) \\
        z^{k+1} = \prox_{\lambda f}(2y^{k+1} - x^k) \\
        x^{k+1} = x^k + (z^{k+1} - y^{k+1})
    \end{matrix*}\right. \label{eq:PnP-DRS}\tag{PnP-DRS} \\
    &\Rev{\left\{ \begin{matrix*}[l]
        q^{k+1} =(1-\hat\alpha)y^k + \hat\alpha x^k \\
        x^{k+1} = D_\sigma (x^k - \lambda \nabla f(q^{k+1})) \\
        y^{k+1} = (1-\hat\alpha)y^k+\hat\alpha x^{k+1}
    \end{matrix*}\right. \tag{PnP-\text{$\hat\alpha$}PGD}\label{eq:PnP-aPGD}} 
\end{align}
\begin{figure}[h]
    \centering
    \subfloat[\centering Residual DPIR\textsuperscript{1}, $\sigma=7.65$]{{\includegraphics[height=3.75cm]{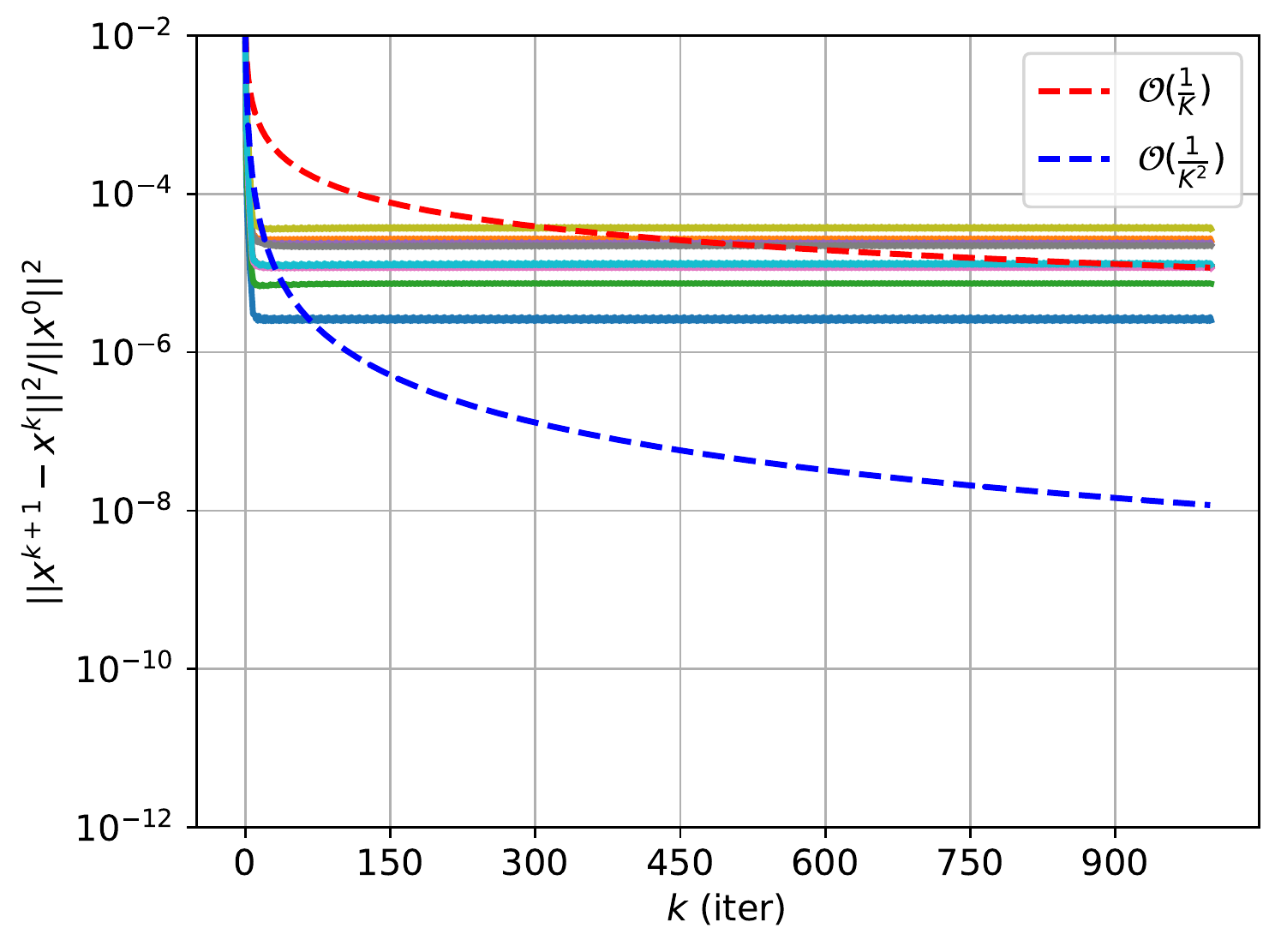}}}%
    \subfloat[\centering Residual DPIR\textsuperscript{2}, $\sigma=7.65$]{{\includegraphics[height=3.75cm]{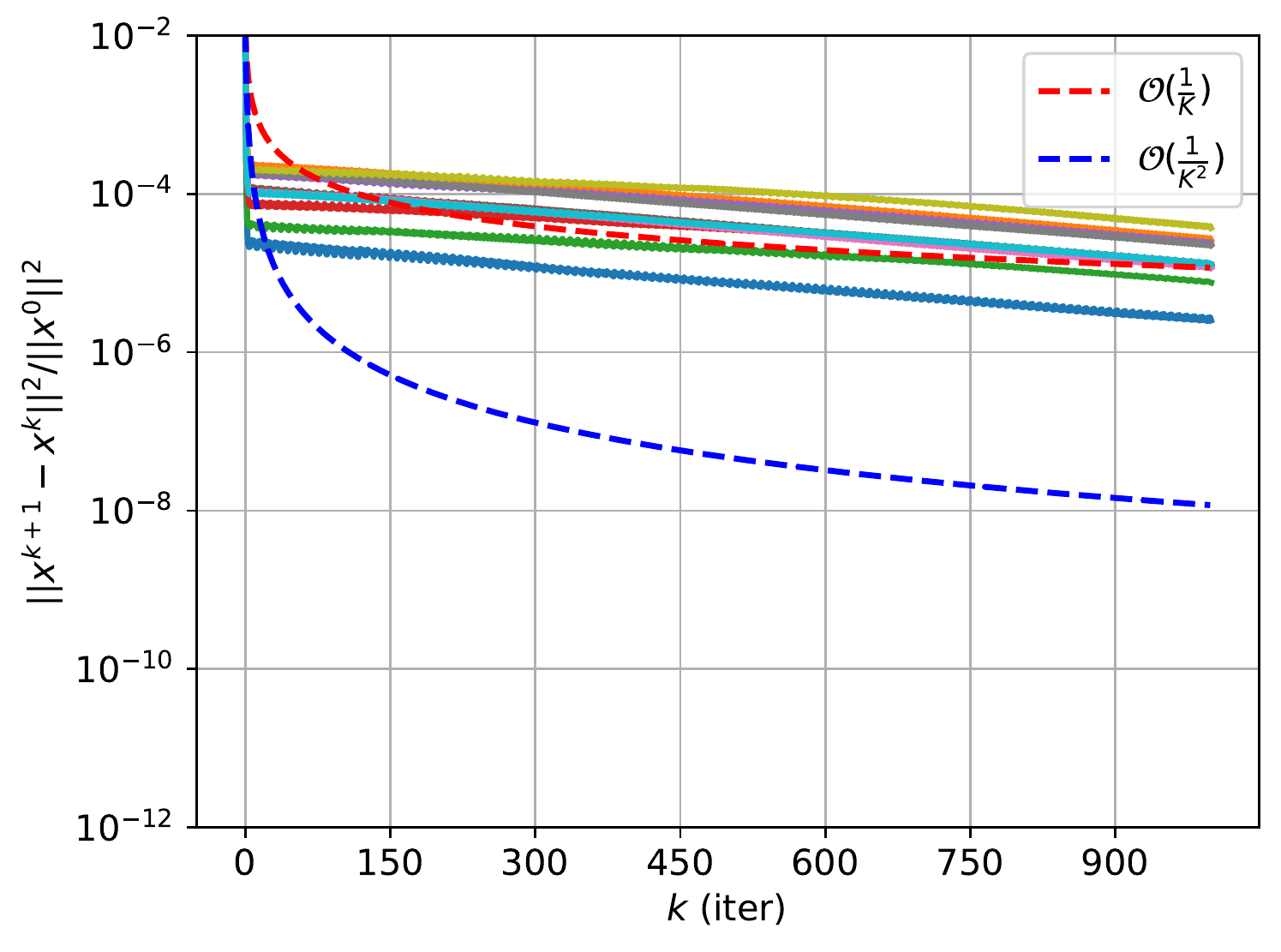}}}%
`    \subfloat[\centering Residual DPIR\textsuperscript{1}, $\sigma=2.55$]{{\includegraphics[height=3.75cm]{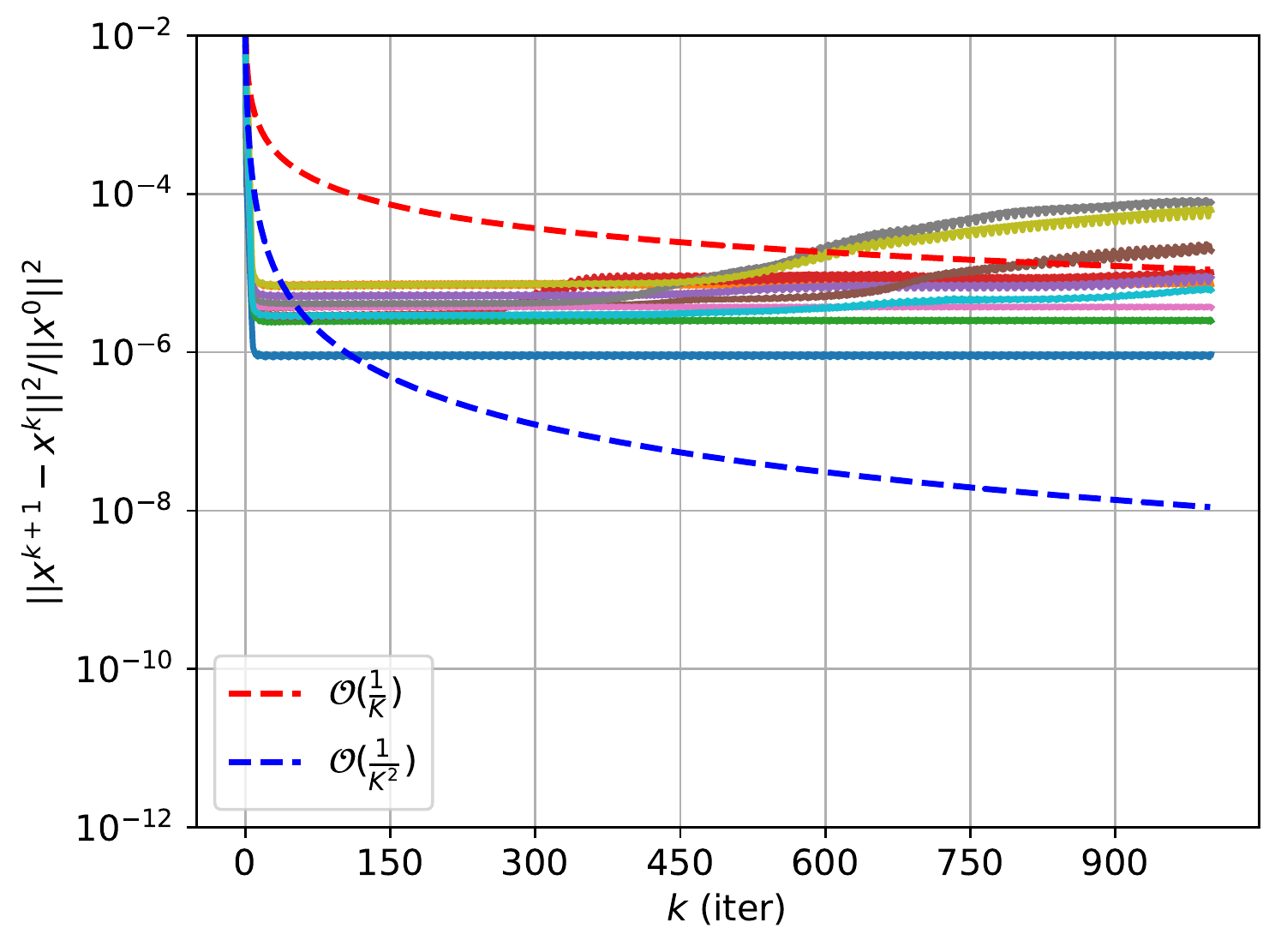}}}%

    \subfloat[\centering  PSNR DPIR\textsuperscript{1}, $\sigma=7.65$]{{\includegraphics[height=3.75cm]{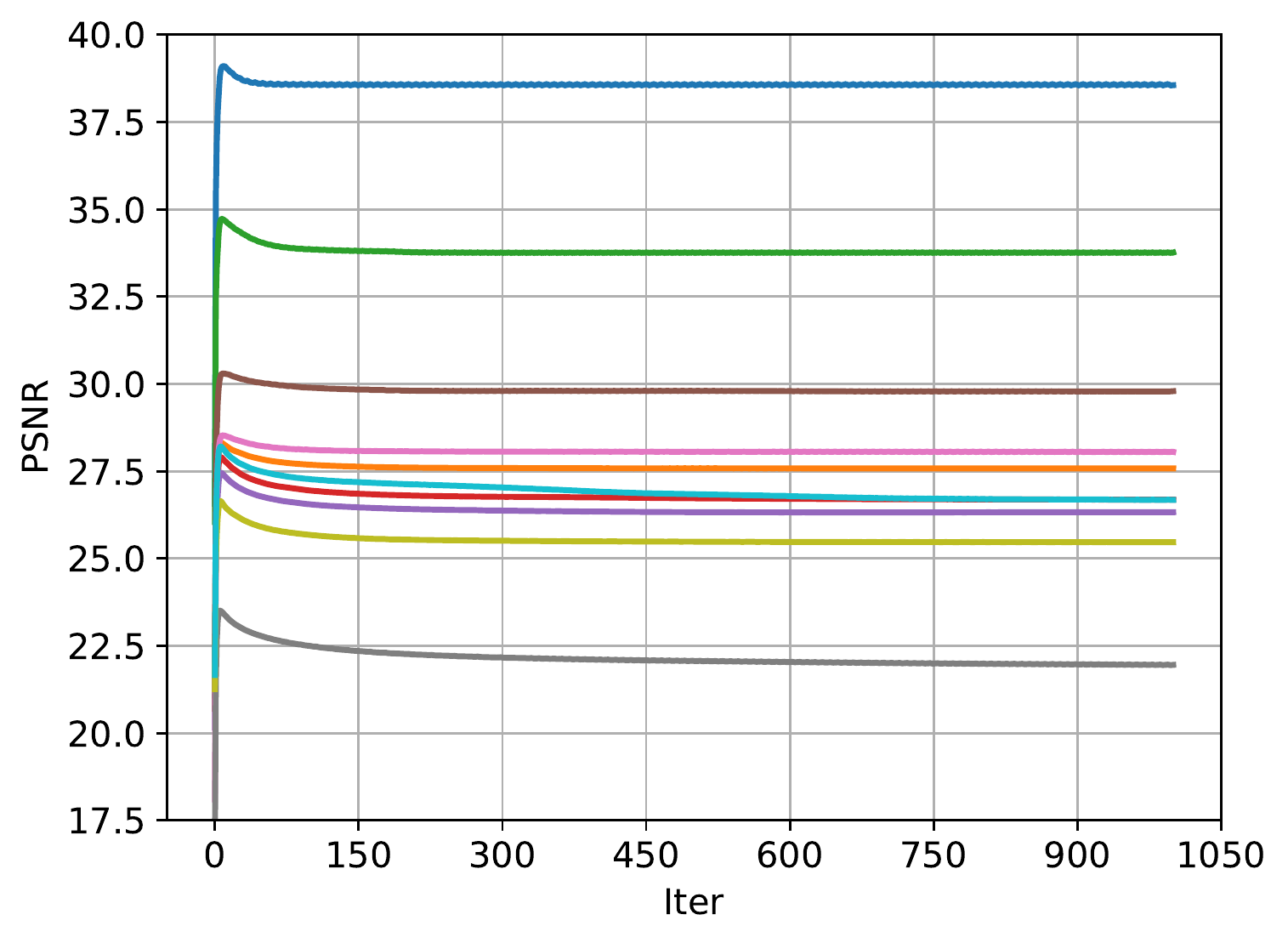}}}%
    \subfloat[\centering  PSNR DPIR\textsuperscript{2}, $\sigma=7.65$]{{\includegraphics[height=3.75cm]{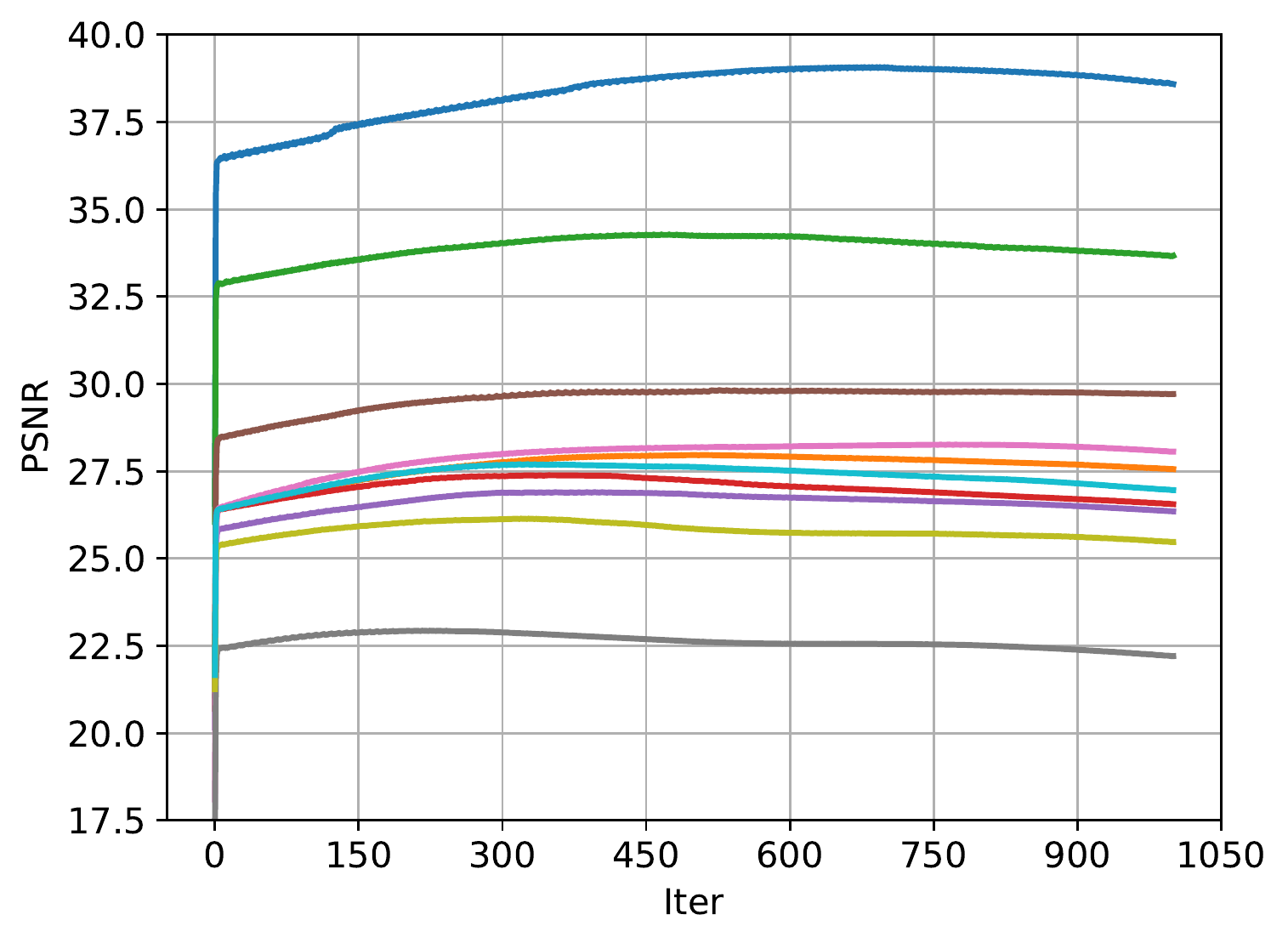}}}%
`    \subfloat[\centering PSNR DPIR\textsuperscript{1}, $\sigma=2.55$]{{\includegraphics[height=3.75cm]{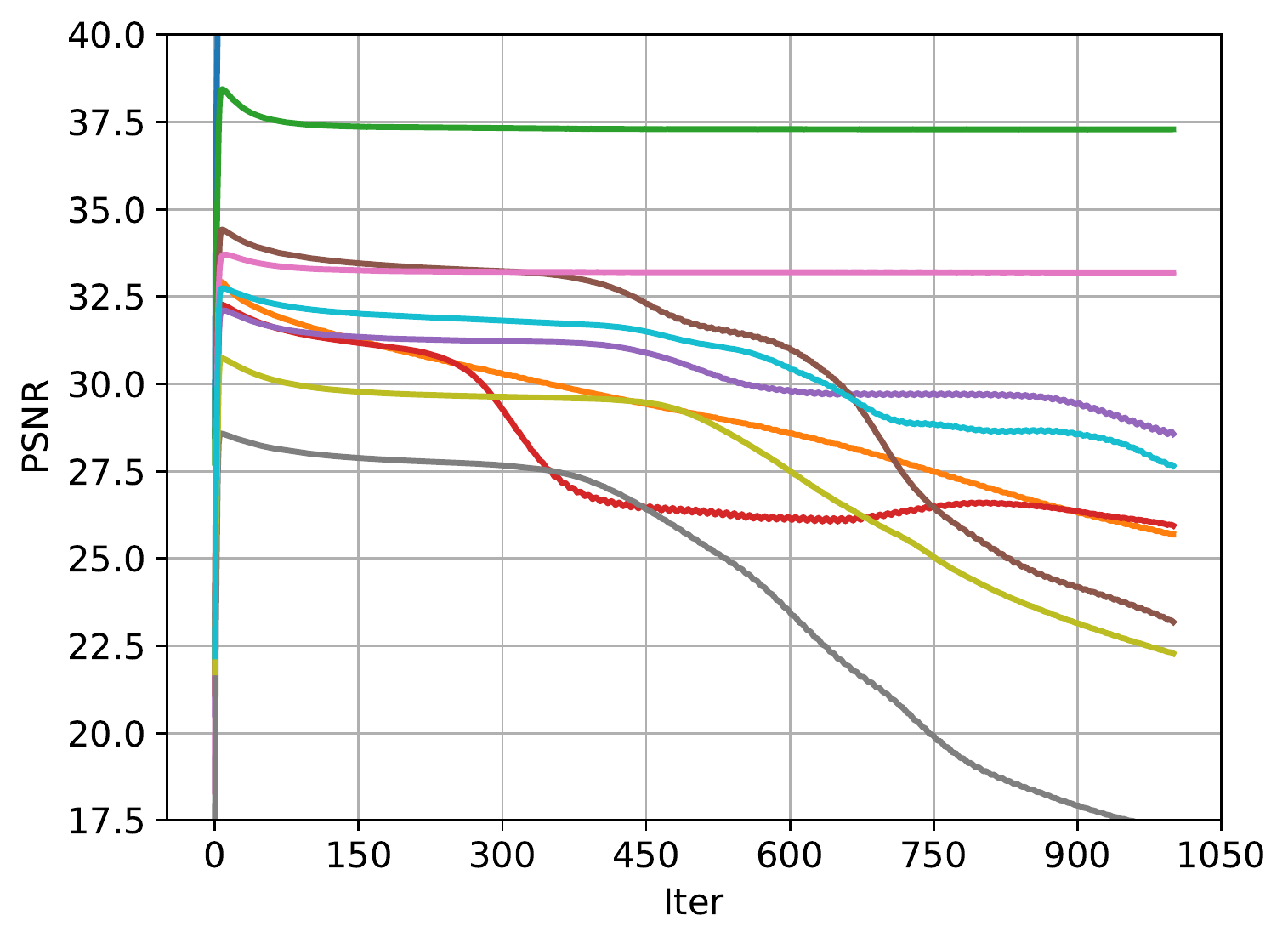}}}%
    \caption{\Rev{Performance of DPIR measured in terms of residual $\|x^{k+1}-x^k\|^2/\|x^0\|^2$ and PSNR for deblurring with noise levels $\sigma=2.55, 7.65$, applied with two different denoiser strength regimes. \RRev{Each curve corresponds to one of the 10 images from the CBSD10 dataset.} DPIR\textsuperscript{1} has denoiser strength decreased from 49 to $\sigma$ over 8 iterations for deblurring, and extended with $\sigma_d=\sigma$ for following iterations. DPIR\textsuperscript{2} has denoiser strength decreased from 49 to $\sigma$ over 1000 iterations. We observe that both methods have decreasing PSNR at later iterations and non-converging residual, and further that DPIR diverges for small noise levels. }}
    \label{fig:dpir_divergence} 
\end{figure}

\begin{figure}[h]
    \centering
    \subfloat[\centering Residual PnP-FISTA]{{\includegraphics[height=3.75cm]{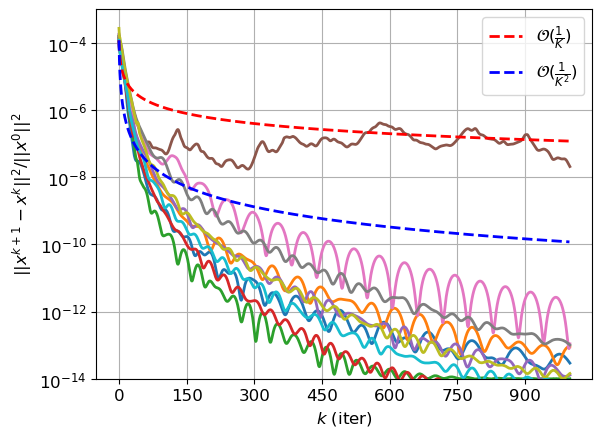}}}%
    \subfloat[\centering PSNR PnP-FISTA]{{\includegraphics[height=3.75cm]{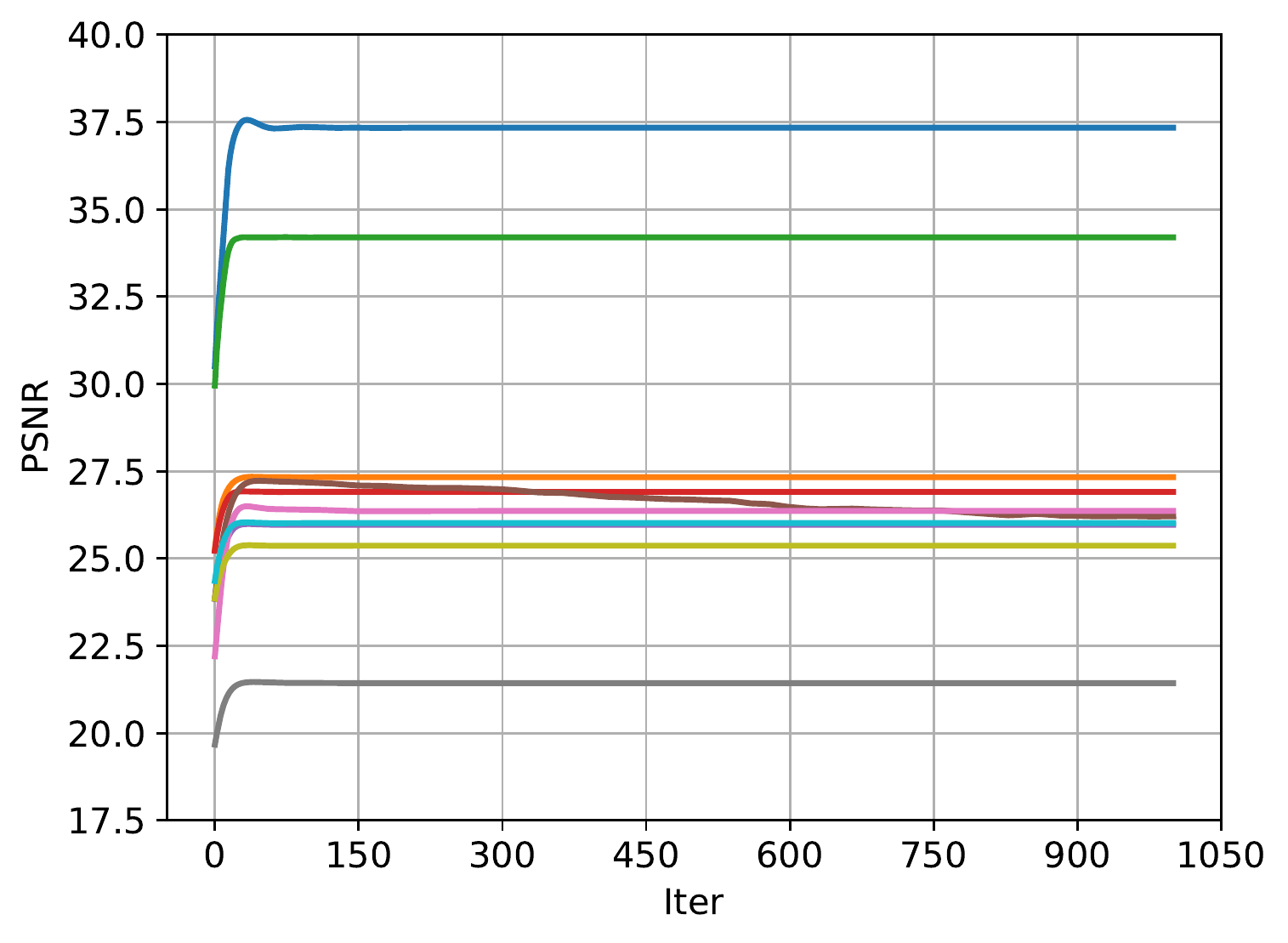}}}%
    \subfloat[\centering Failure example]{{\includegraphics[height=3.75cm]{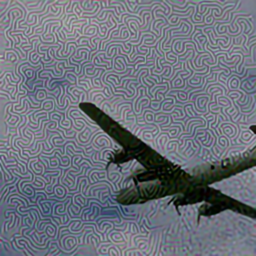}}}%
    \caption{\Rev{Residual $\|x^{k+1}-x^k\|^2/\|x^0\|^2$ and PSNR for PnP-FISTA applied to super-resolution with noise level $\sigma=7.65$. \RRev{Each curve corresponds to one of the 10 images from the CBSD10 dataset.} Using the parameters of PnP-LBFGS, which should resolve any Lipschitz constraint issues, has the same divergence issue. PnP-FISTA sometimes fails, leading to images with artifacts as seen in subfigure (c).}}
    \label{fig:fista_divergence}%
\end{figure}
\begin{figure}[]%
    \centering
    \subfloat[\centering PnP-LBFGS\textsuperscript{1} (28.75dB)]{{\includegraphics[height=3.50cm]{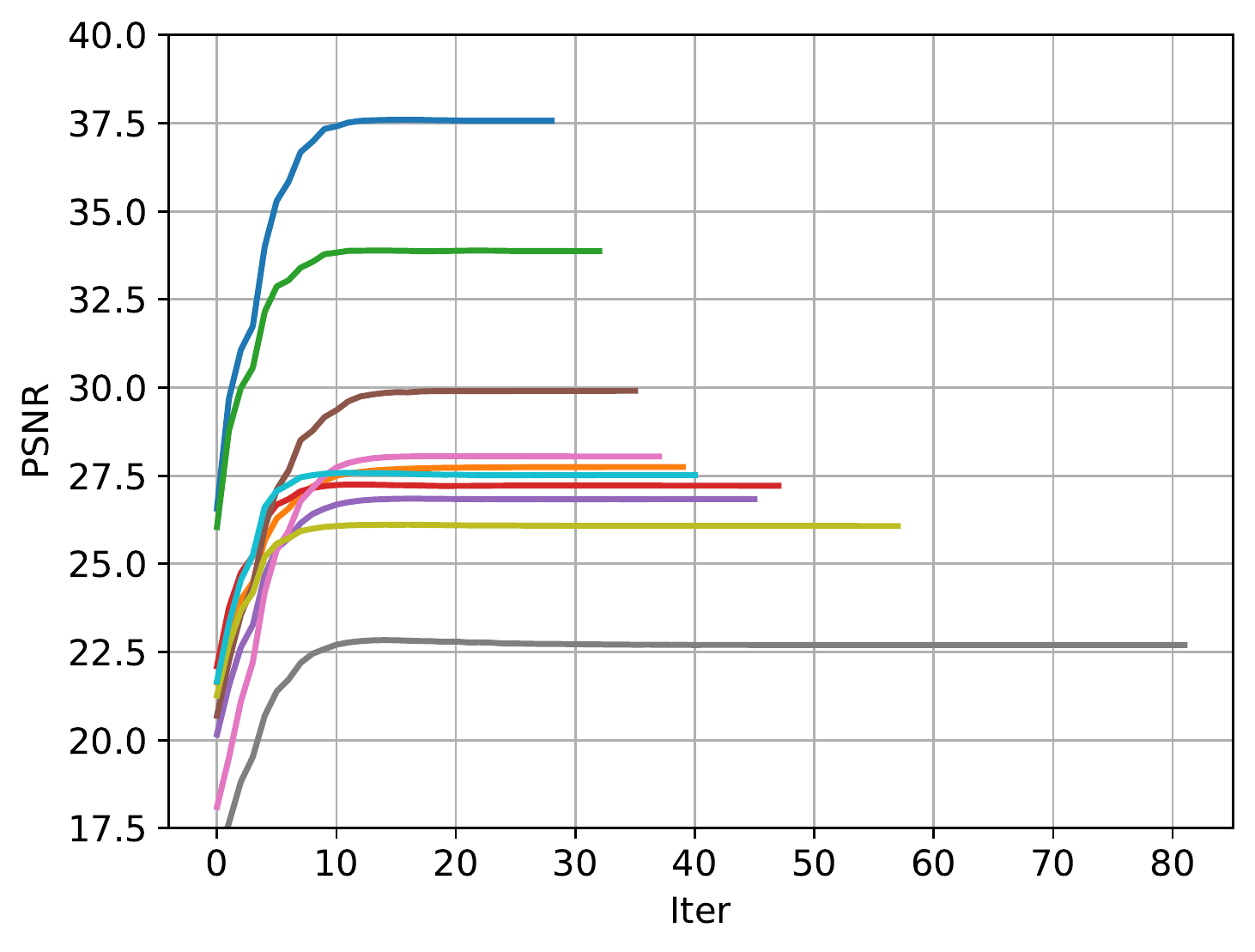}}}%
    \subfloat[\centering PnP-LBFGS\textsuperscript{2} (28.75dB)]{{\includegraphics[height=3.50cm]{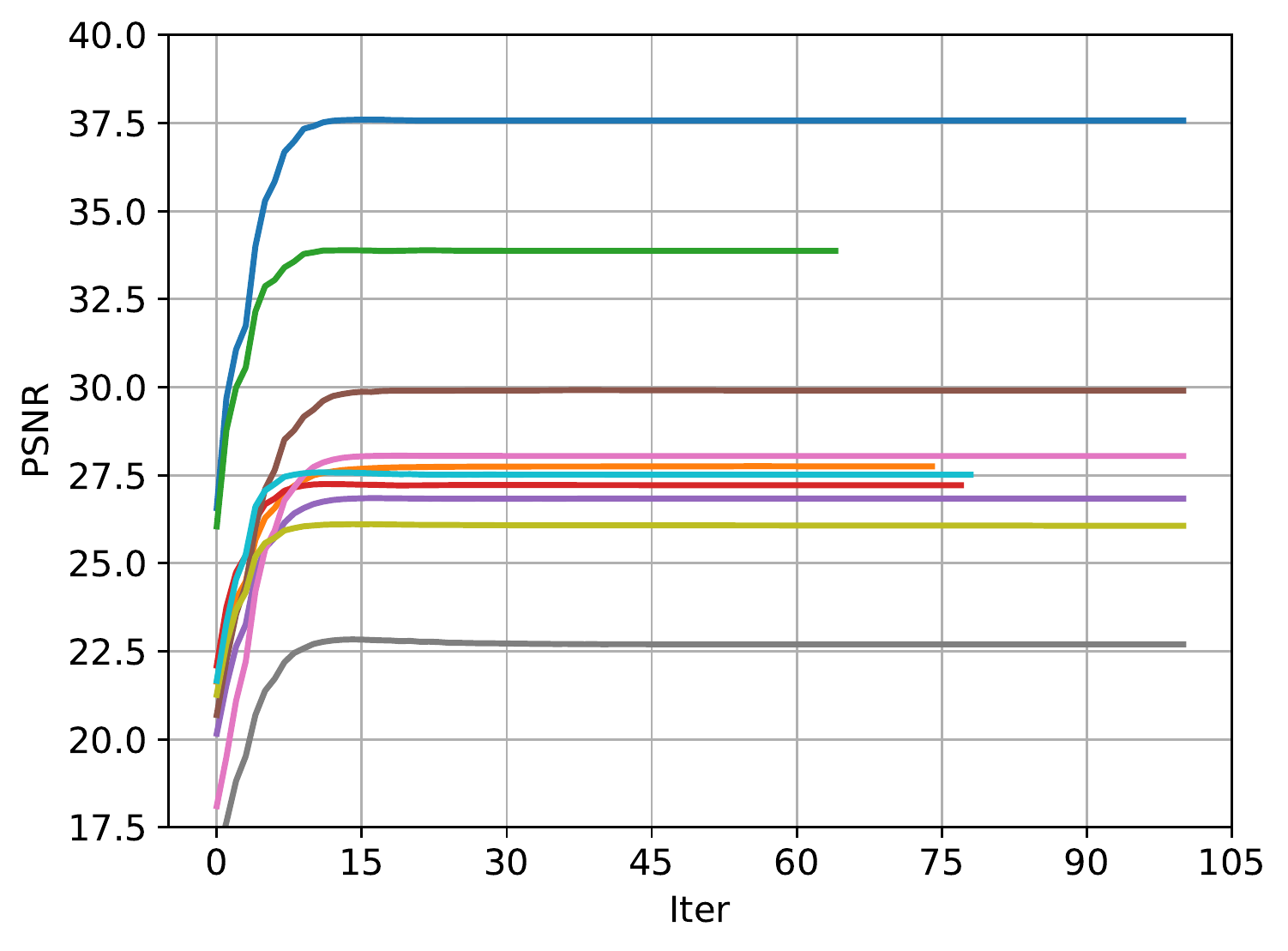}}}%
    \subfloat[\centering DPIR (28.49dB)]{{\includegraphics[height=3.50cm]{figs/deblur/dpir/PSNR.pdf}}}%
    
    \subfloat[\centering PnP-PGD (28.60dB)]{{\includegraphics[height=3.50cm]{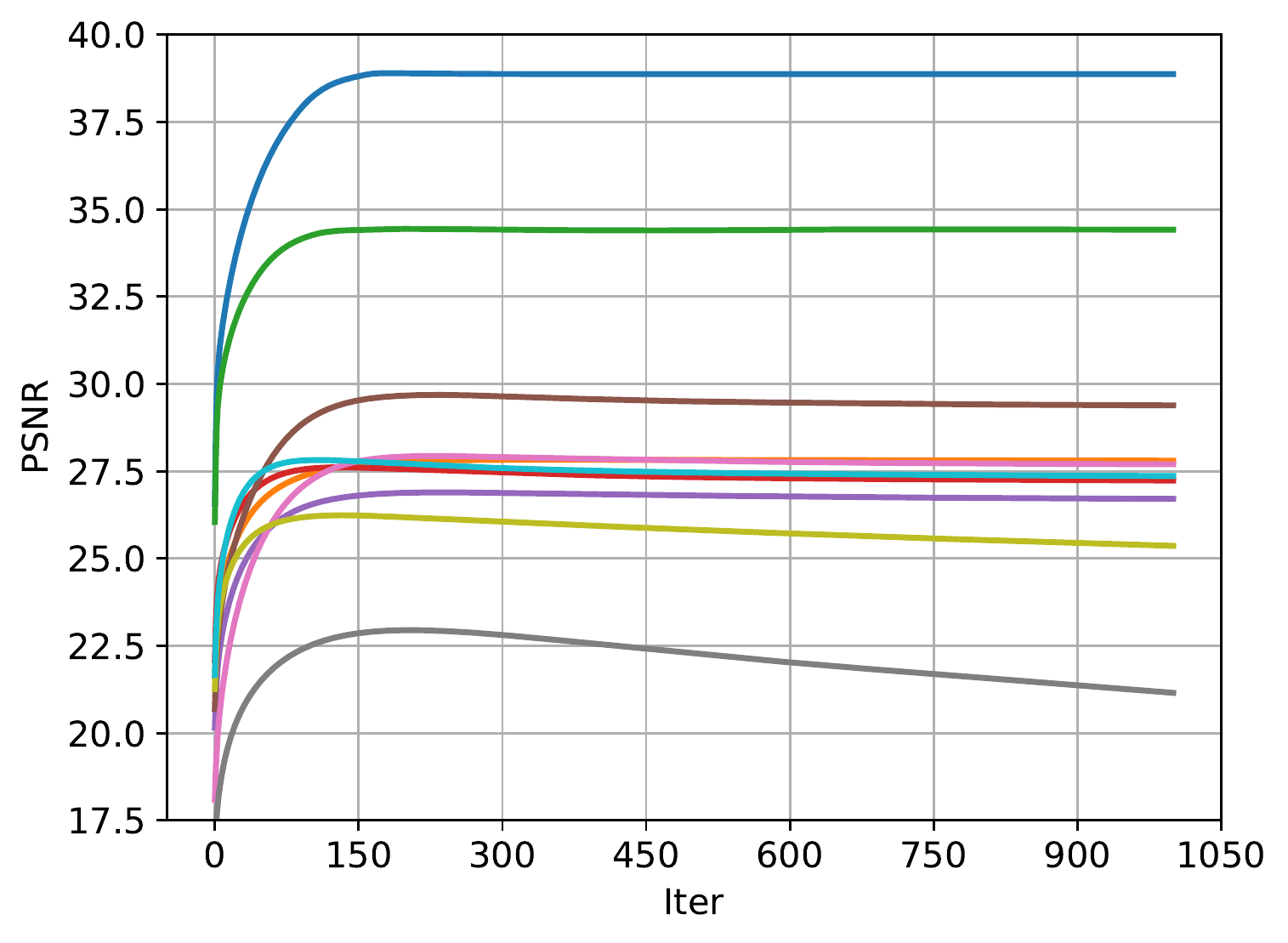}}}%
    \subfloat[\centering PnP-$\hat\alpha$PGD (29.05dB)]{{\includegraphics[height=3.50cm]{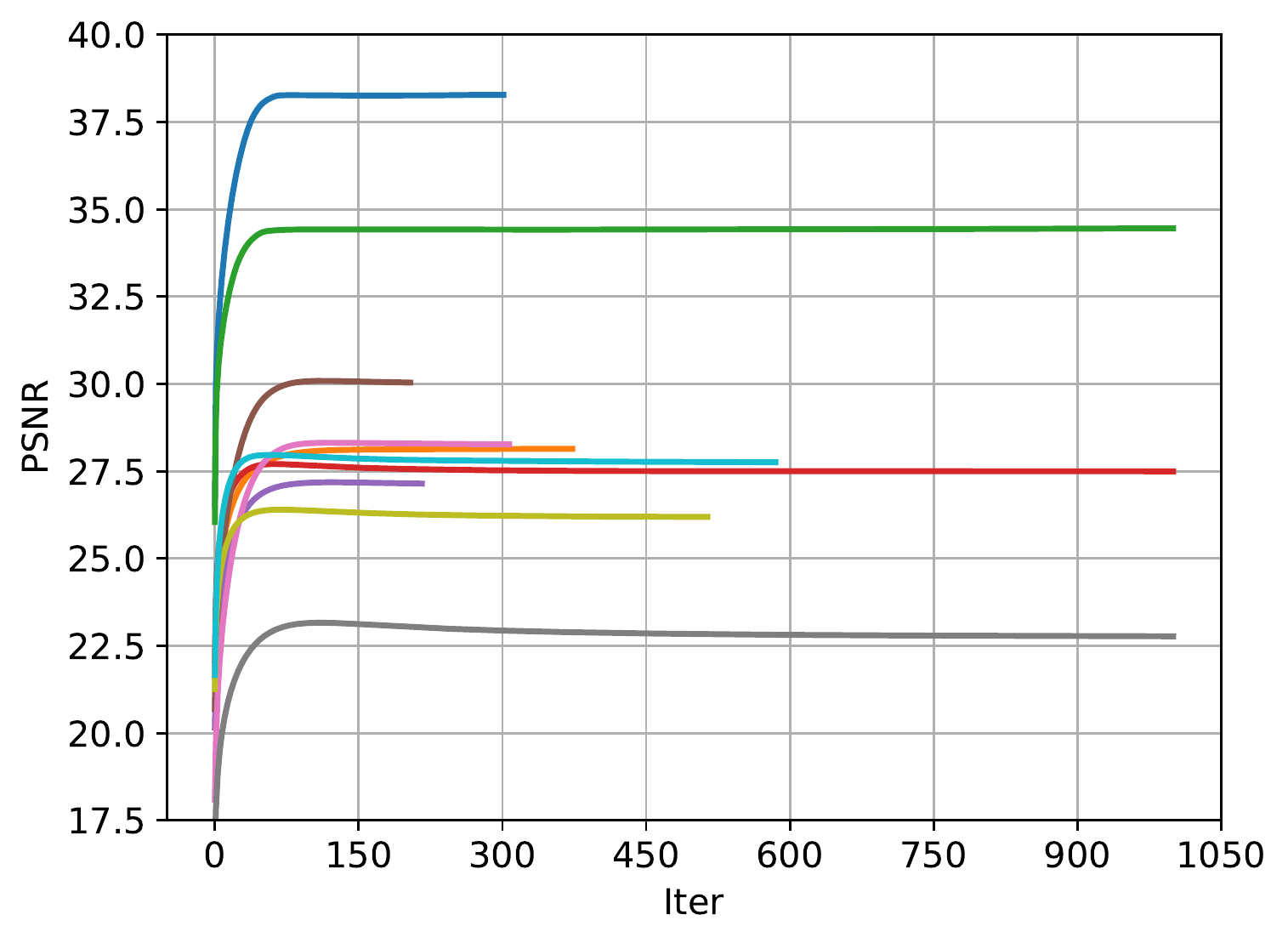}}}%
    \subfloat[\centering PnP-FISTA (28.75dB)]{{\includegraphics[height=3.50cm]{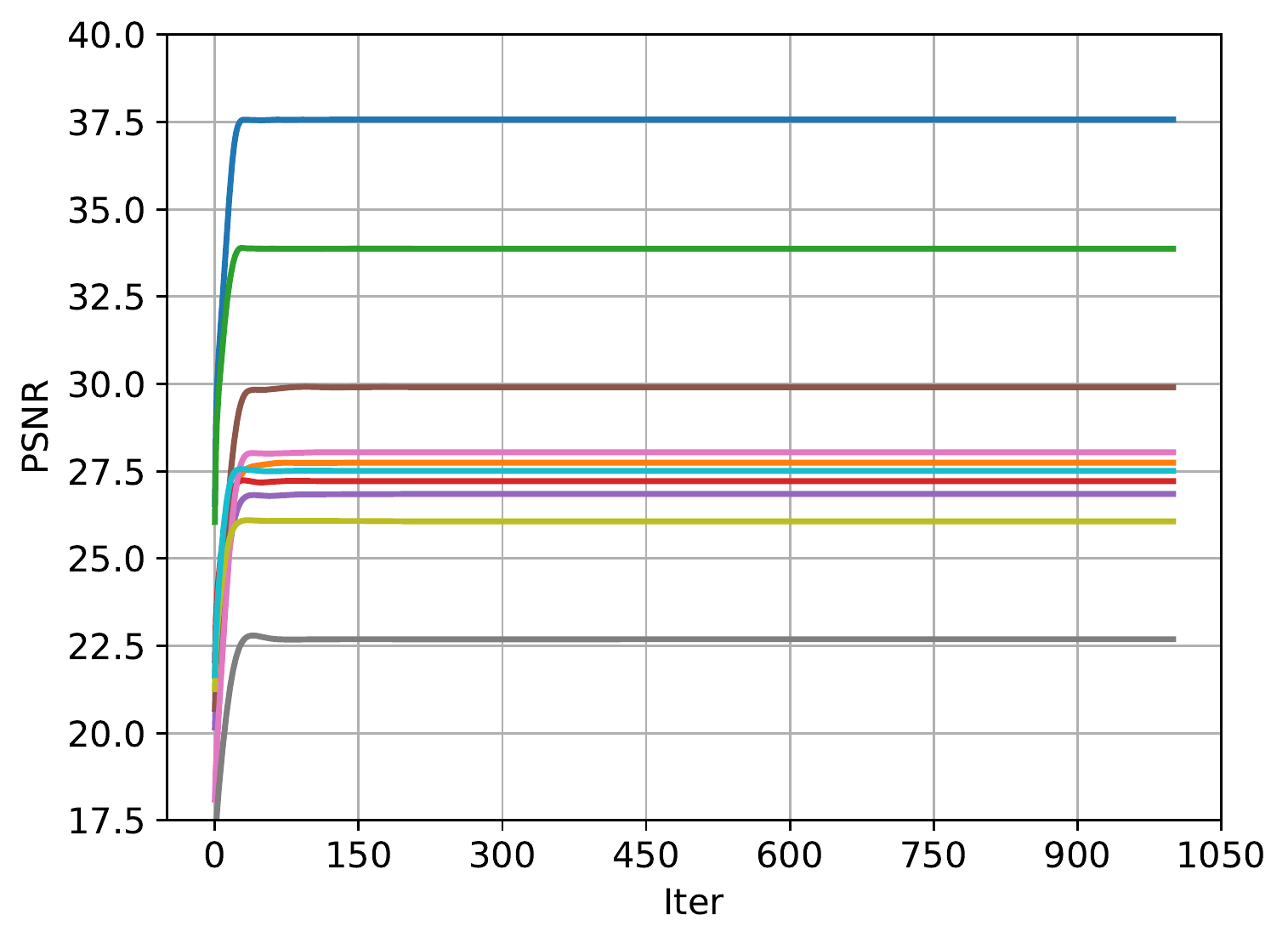}}}%
        
    \subfloat[\centering PnP-DRSdiff (28.61dB)]{{\includegraphics[height=3.50cm]{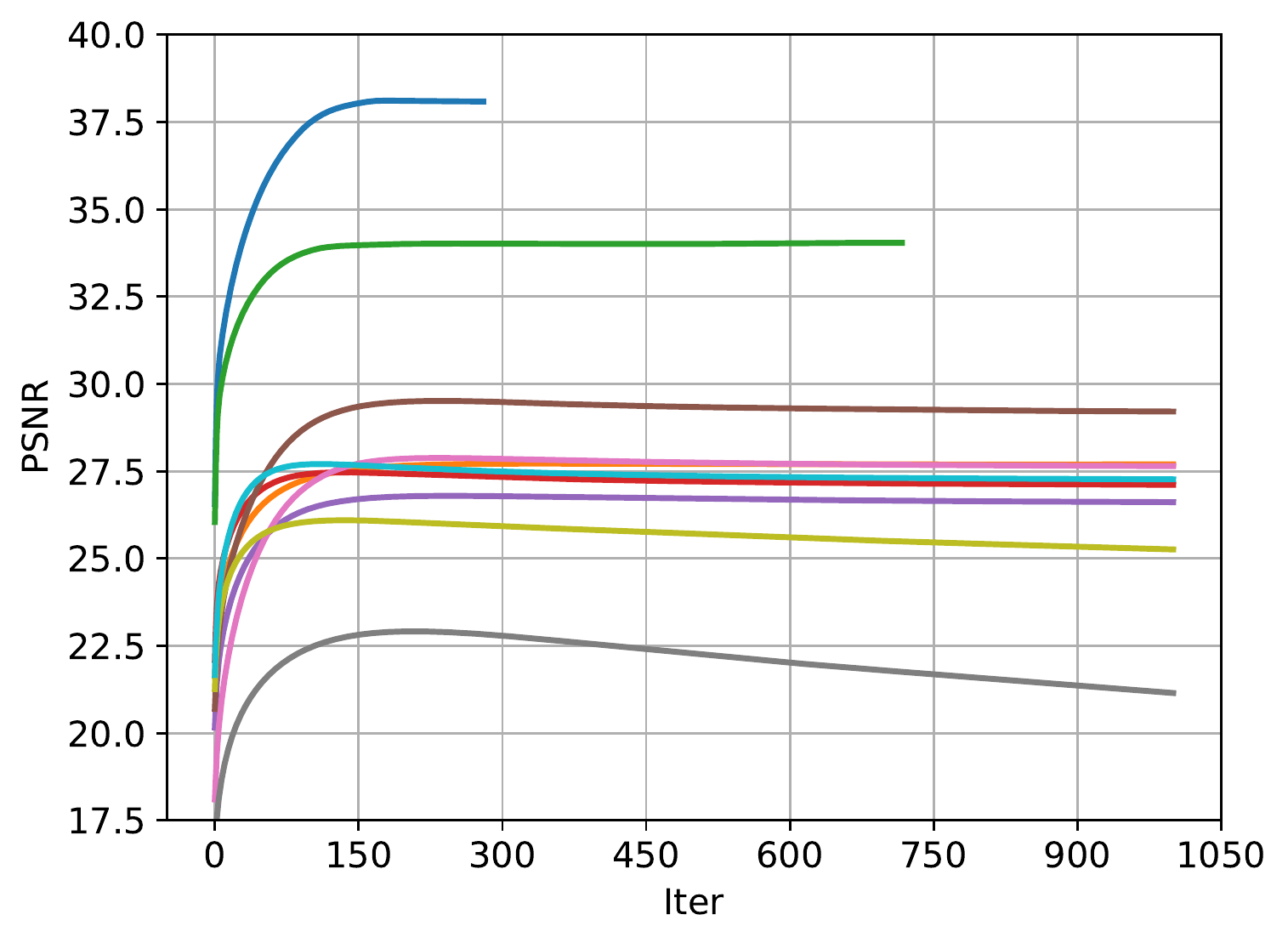}}}%
    \subfloat[\centering PnP-DRS (28.80dB)]{{\includegraphics[height=3.50cm]{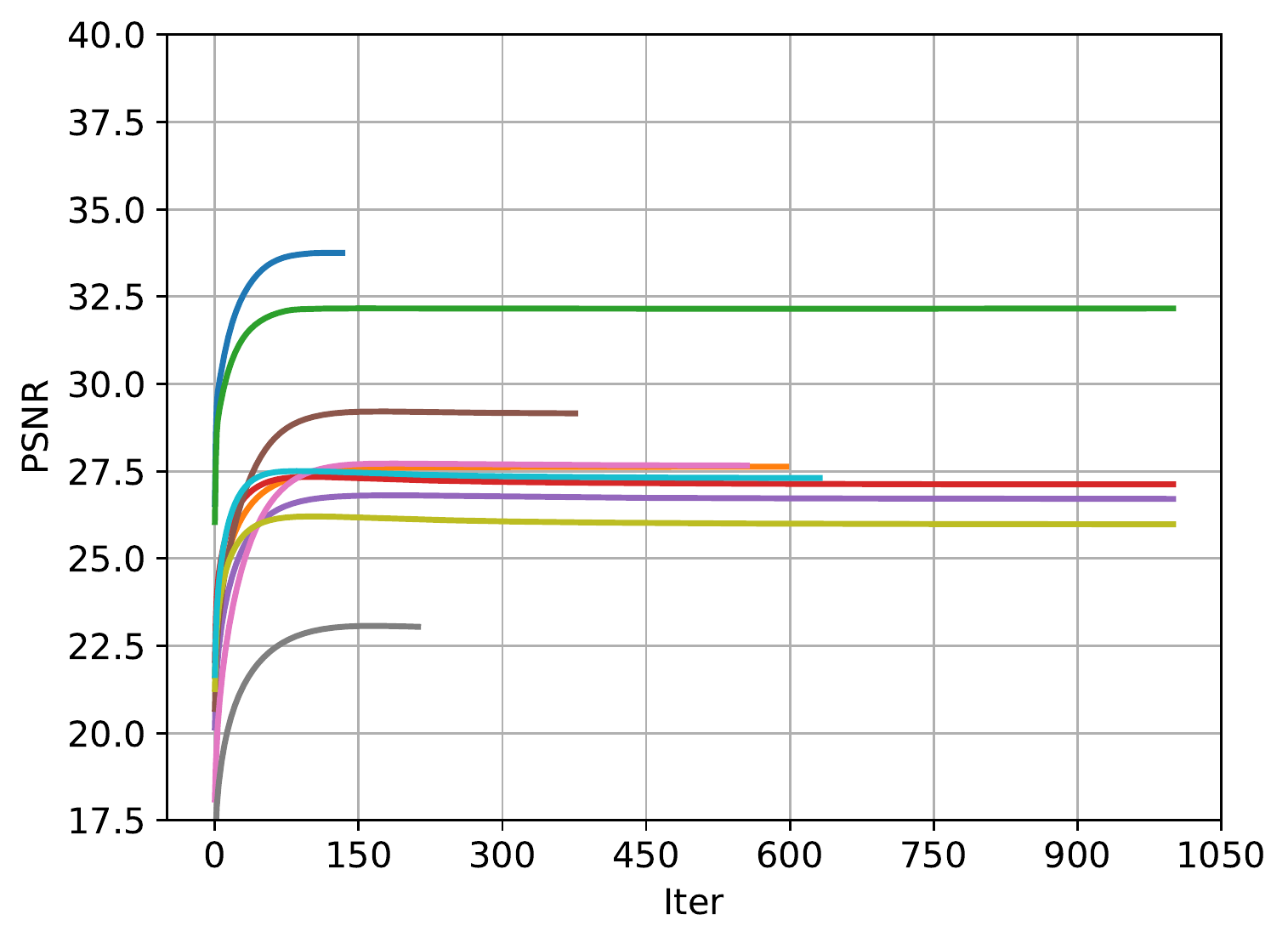}}}%
    \caption{Convergence of the PSNRs for \Rev{deblurring}, with the average dB in brackets. Each curve corresponds to one of the 10 images from the CBSD10 dataset. Note that the scale of (a) is 10 times smaller than the other curves, terminating at 100 instead of 1000. PnP-LBFGS and PnP-DRS have generally more stable convergence, which can be attributed to the smaller Lipschitz constant of $I - D_\sigma$. PnP-LBFGS\textsuperscript{1} also converges in much fewer iterations than the compared methods. The average PSNR between PnP-LBFGS with the two stopping criteria differ by only 0.0013dB.} 
    \label{fig:deblurConvergencePSNR}%
\end{figure}

\begin{figure}[]
    \centering
    \subfloat[\centering PnP-LBFGS\textsuperscript{1} ]{{\includegraphics[height=3.50cm]{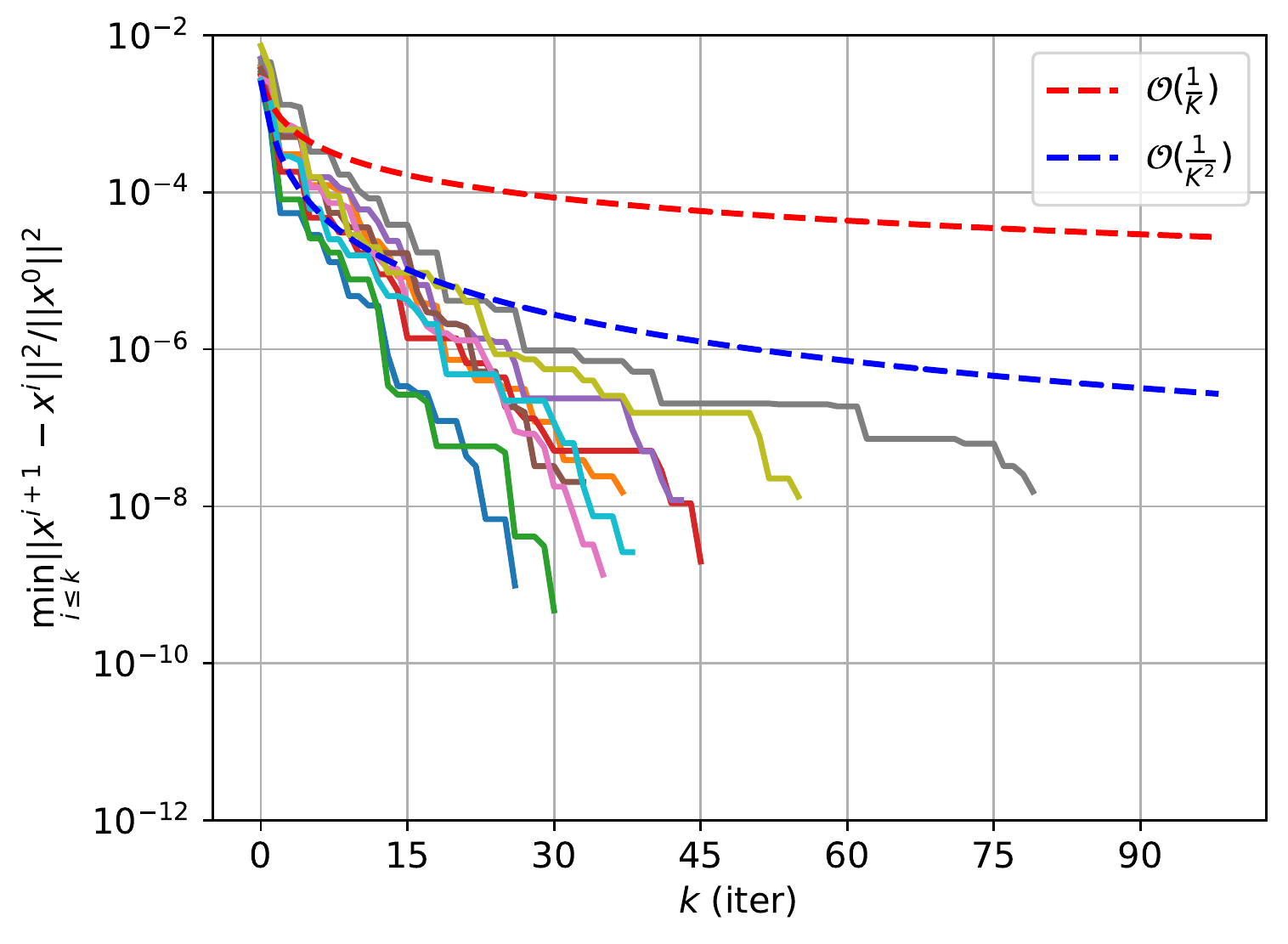} }}%
    \subfloat[\centering PnP-LBFGS\textsuperscript{2} ]{{\includegraphics[height=3.50cm]{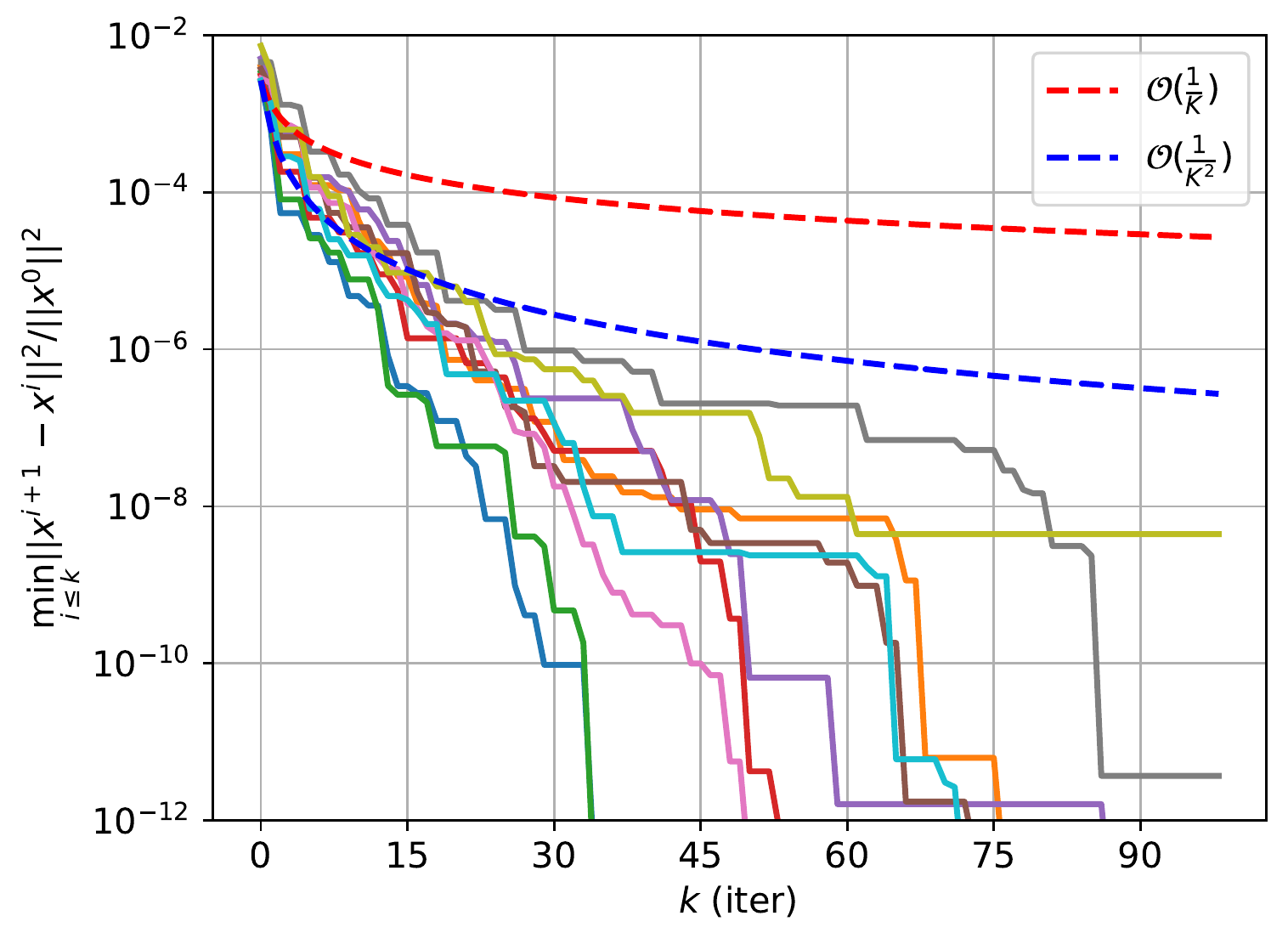} }}%
    \subfloat[\centering DPIR]{{\includegraphics[height=3.50cm]{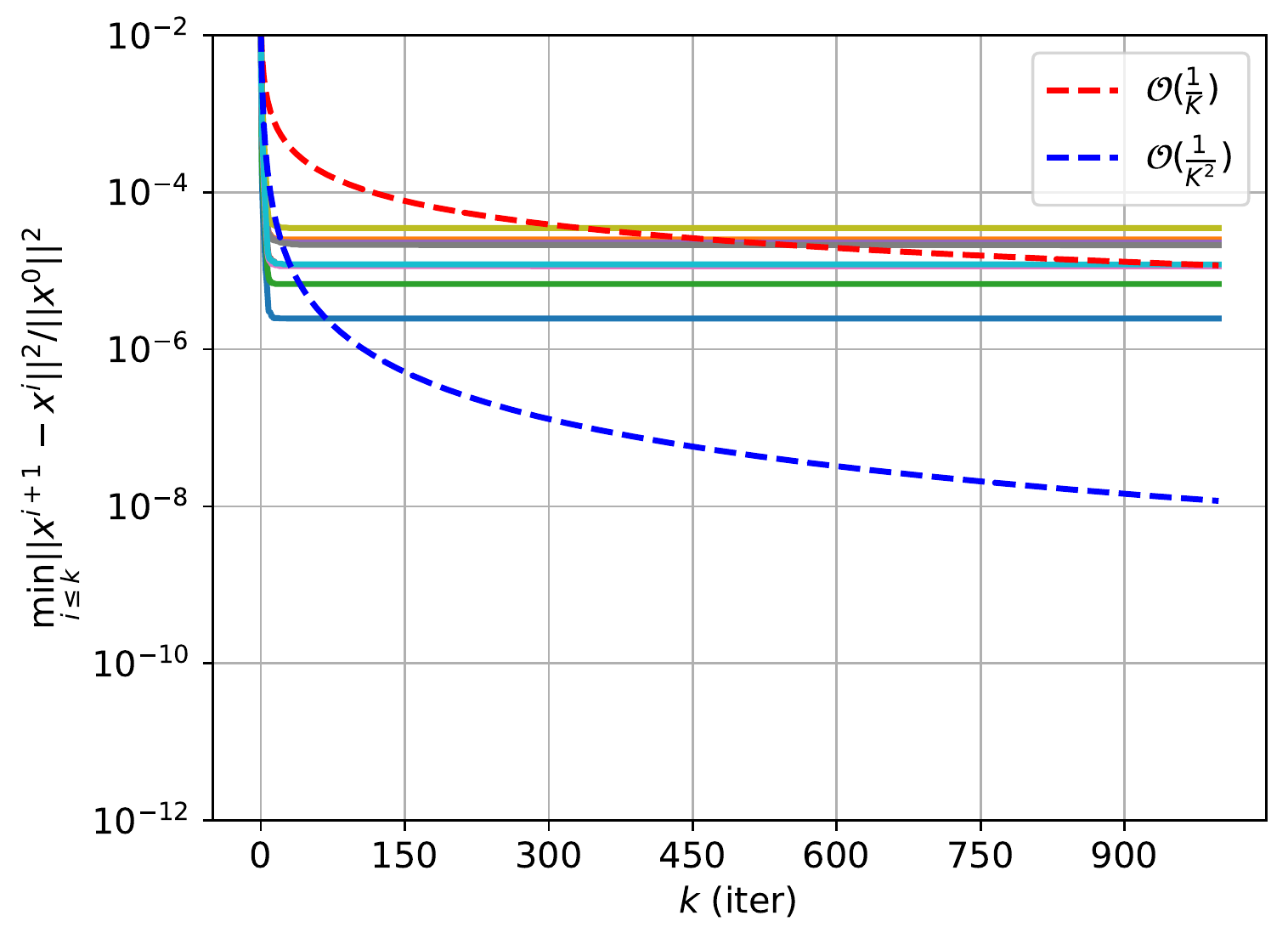}}}%

    \subfloat[\centering PnP-PGD]{{\includegraphics[height=3.50cm]{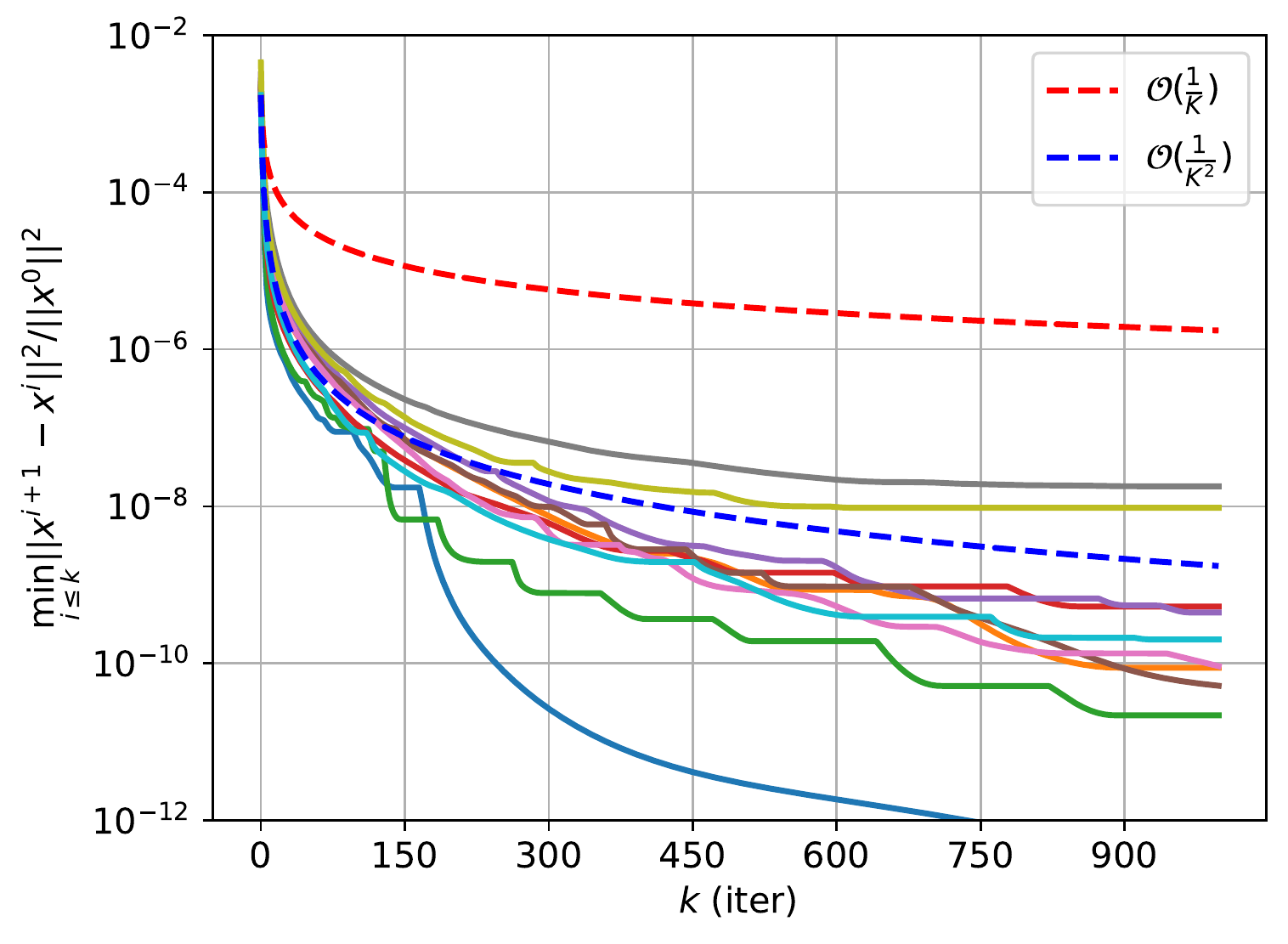}}}%
    \subfloat[\centering PnP-$\hat\alpha$PGD]{{\includegraphics[height=3.50cm]{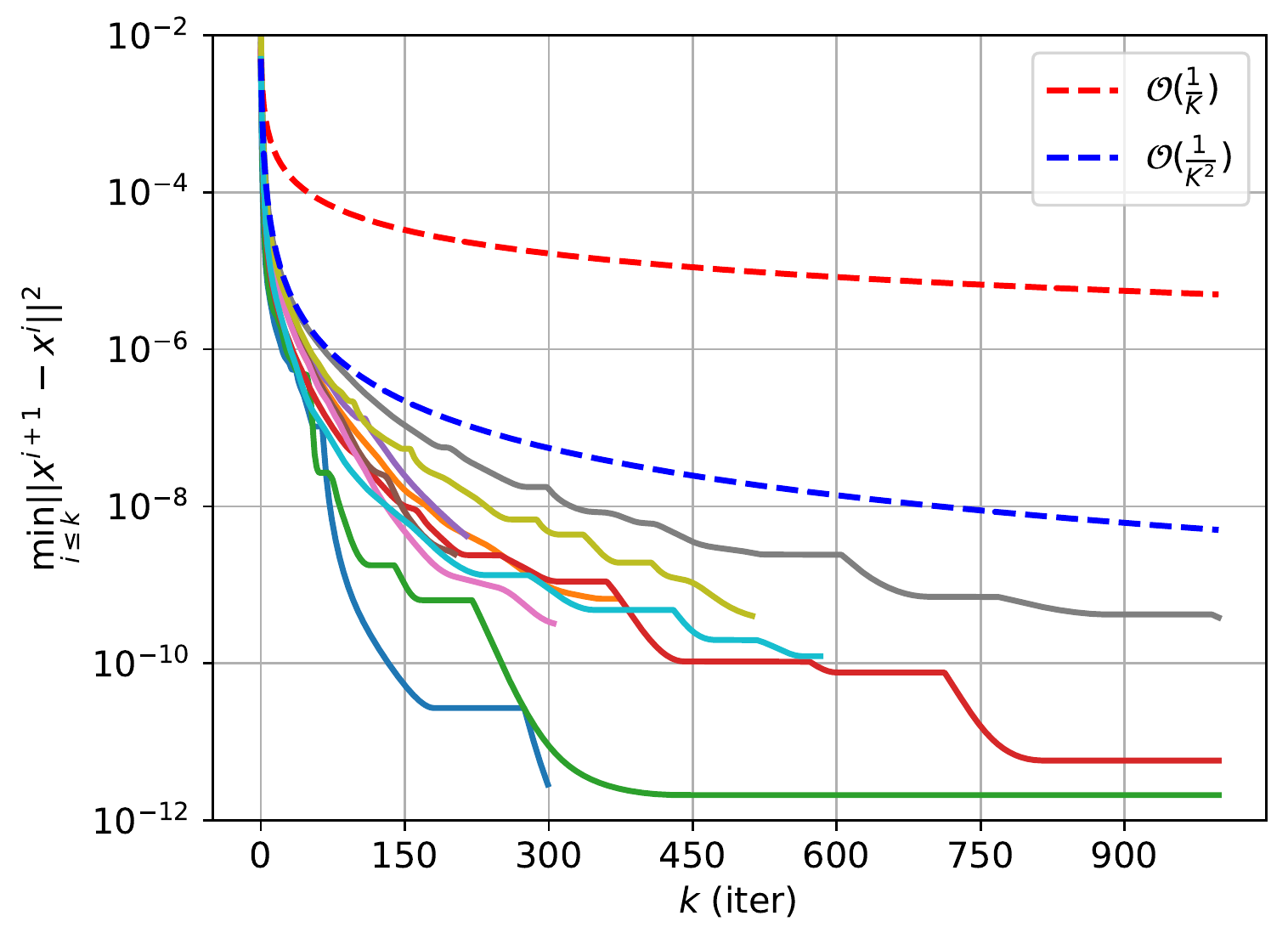}}}
    \subfloat[\centering PnP-FISTA]{{\includegraphics[height=3.50cm]{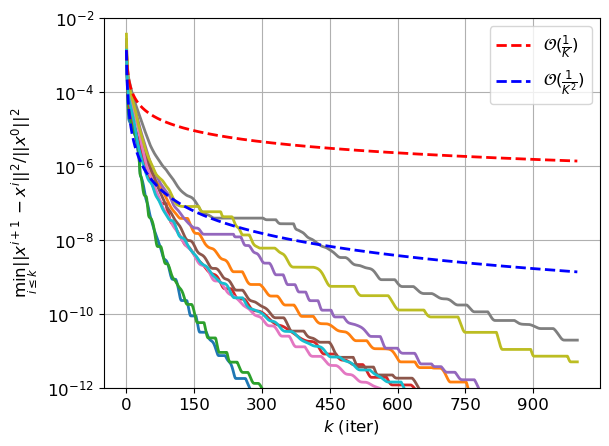}}}%
    
    \subfloat[\centering PnP-DRSdiff]{{\includegraphics[height=3.50cm]{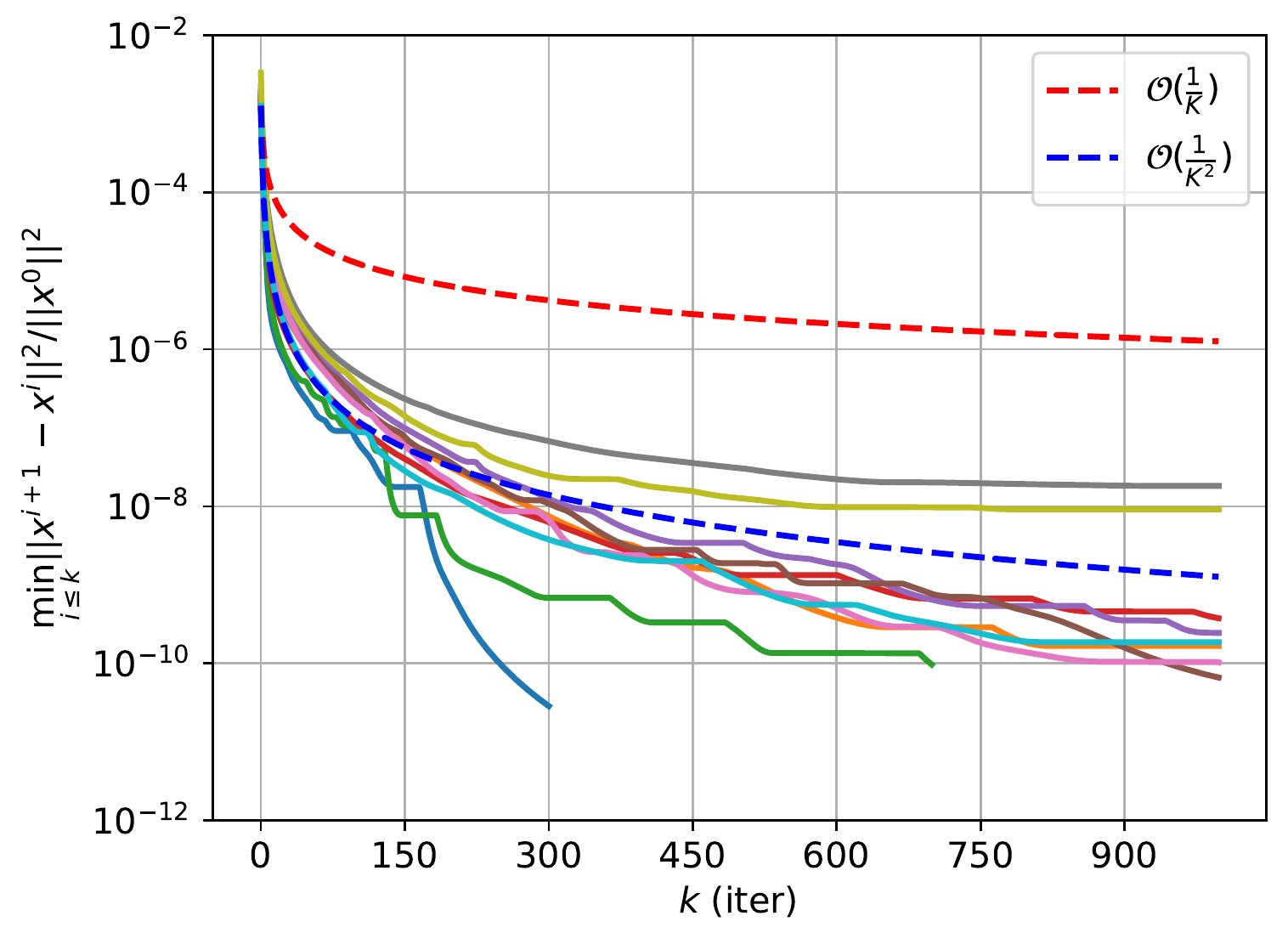}}}%
    \subfloat[\centering PnP-DRS]{{\includegraphics[height=3.50cm]{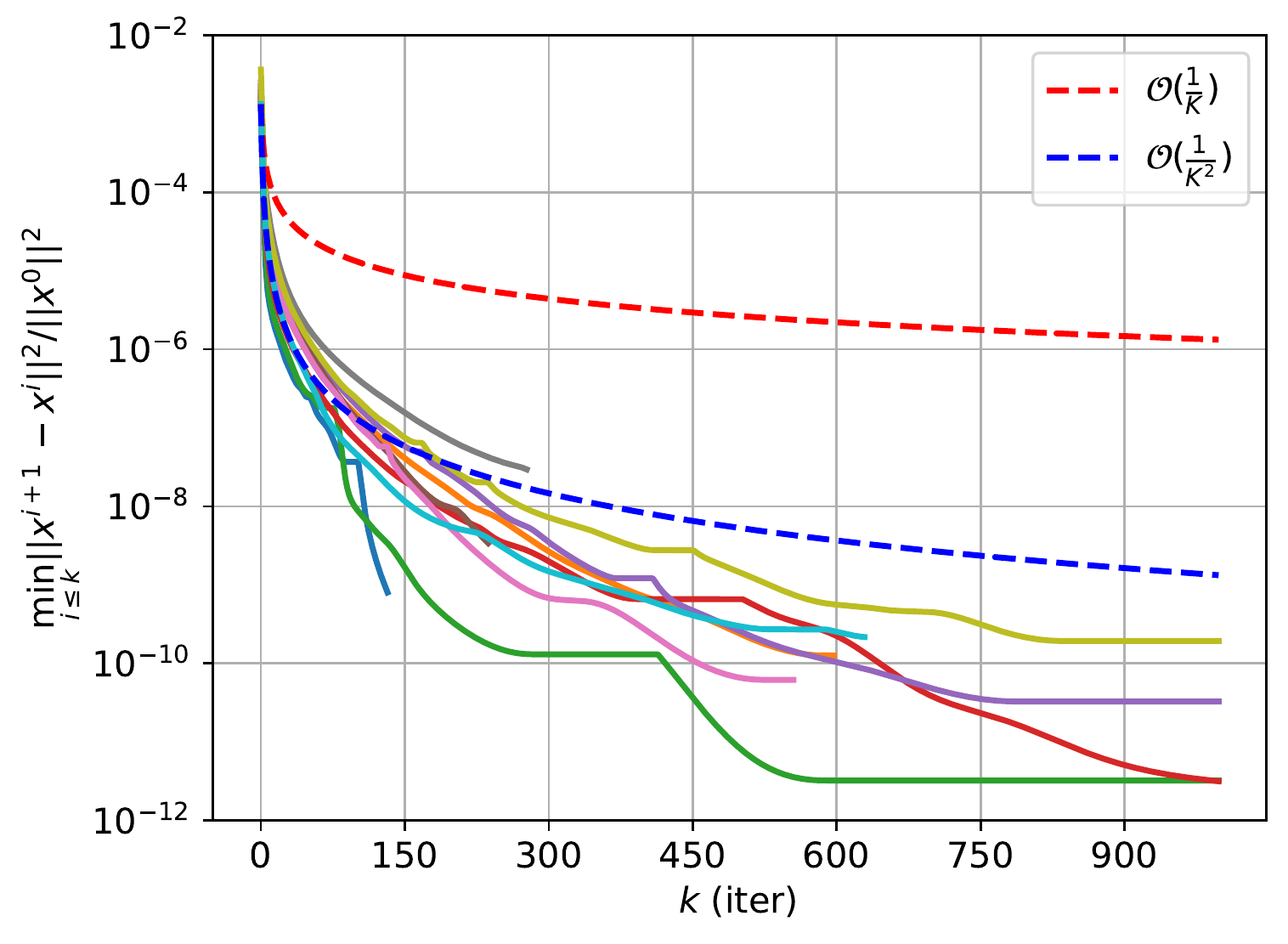}}}%
    \caption{Convergence of the residuals $\min_{i\le k} \|x^{i+1}-x^i\|^2/\|x^0\|^2$ of the various methods \Rev{for deblurring}. Each curve corresponds to one of the 10 images from the CBSD10 dataset, evaluated with the first blur kernel and $\sigma = 7.65$. Note that the x-axis scale of (a) is 10 times smaller than the other curves, terminating at 100 instead of 1000.}
    \label{fig:deblurConvergencelog2}%
\end{figure}%
\begin{figure}[]
    \centering
    \subfloat[\centering $\varphi(x_k)$]{{\includegraphics[height=3.50cm]{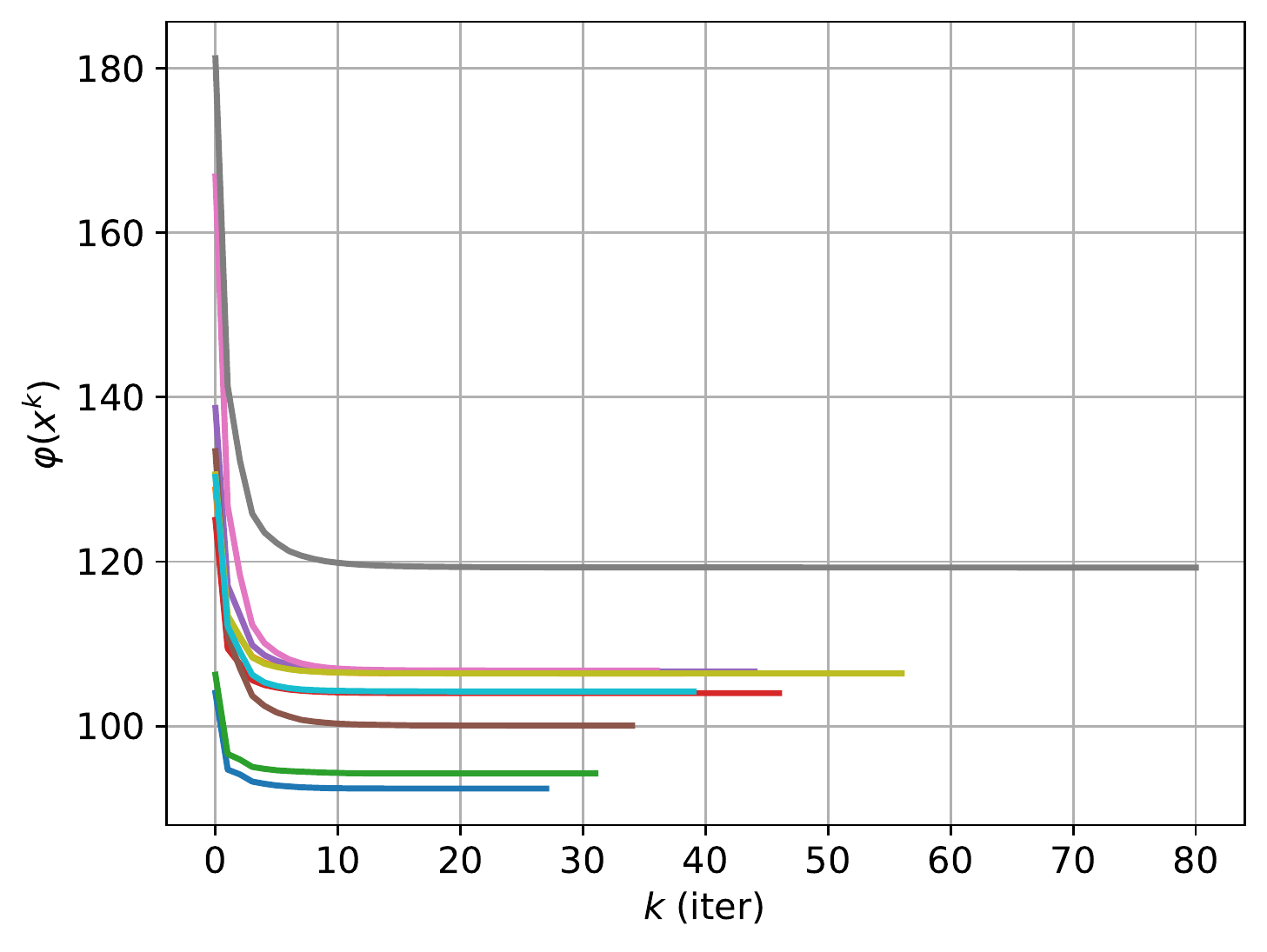}}}%
    \subfloat[\centering $\varphi_\gamma(x_k)$]{{\includegraphics[height=3.50cm]{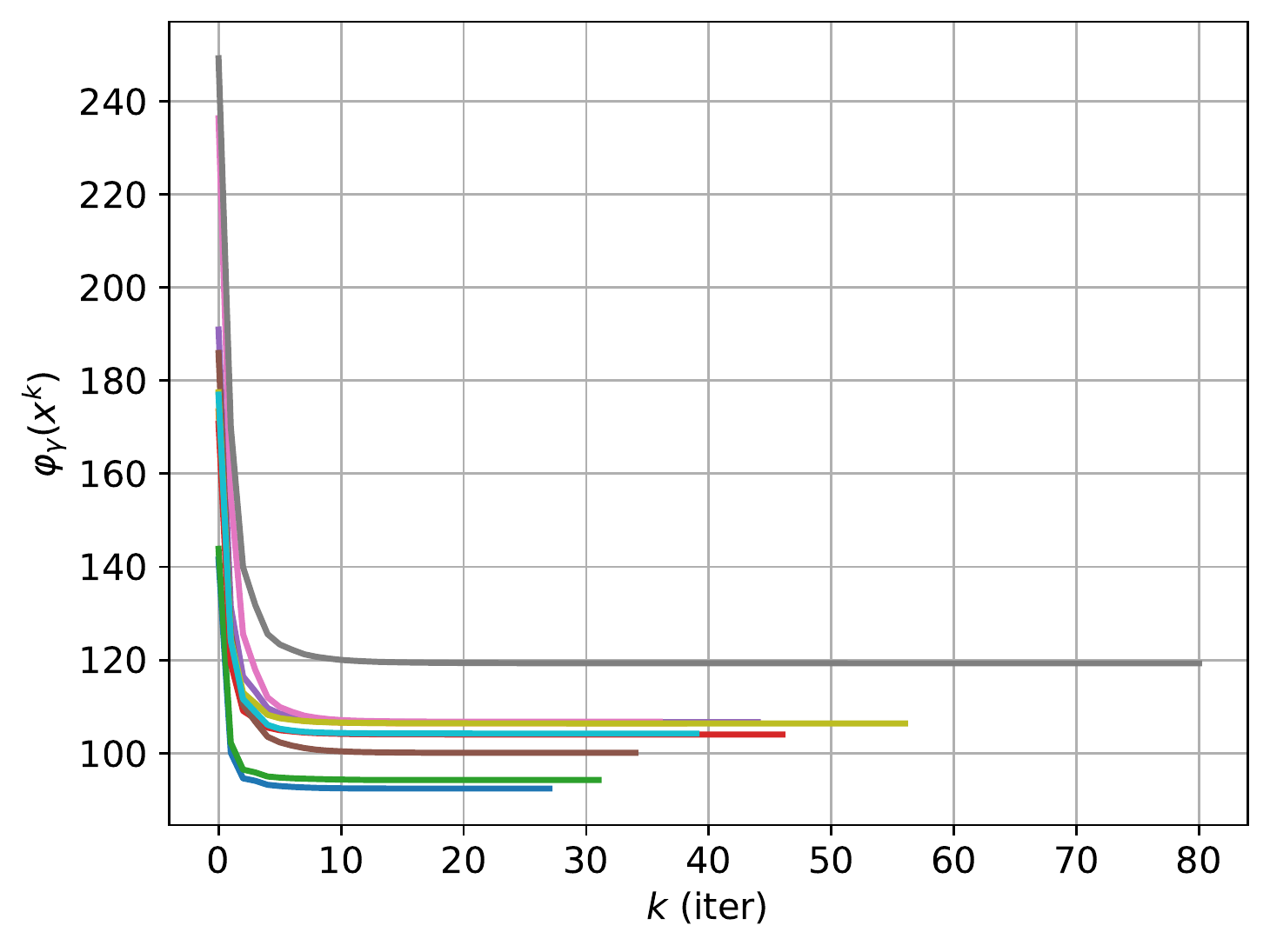}}}%
    \subfloat[\centering $\varphi(x_k) - \varphi_\gamma(x_k)$]{{\includegraphics[height=3.50cm]{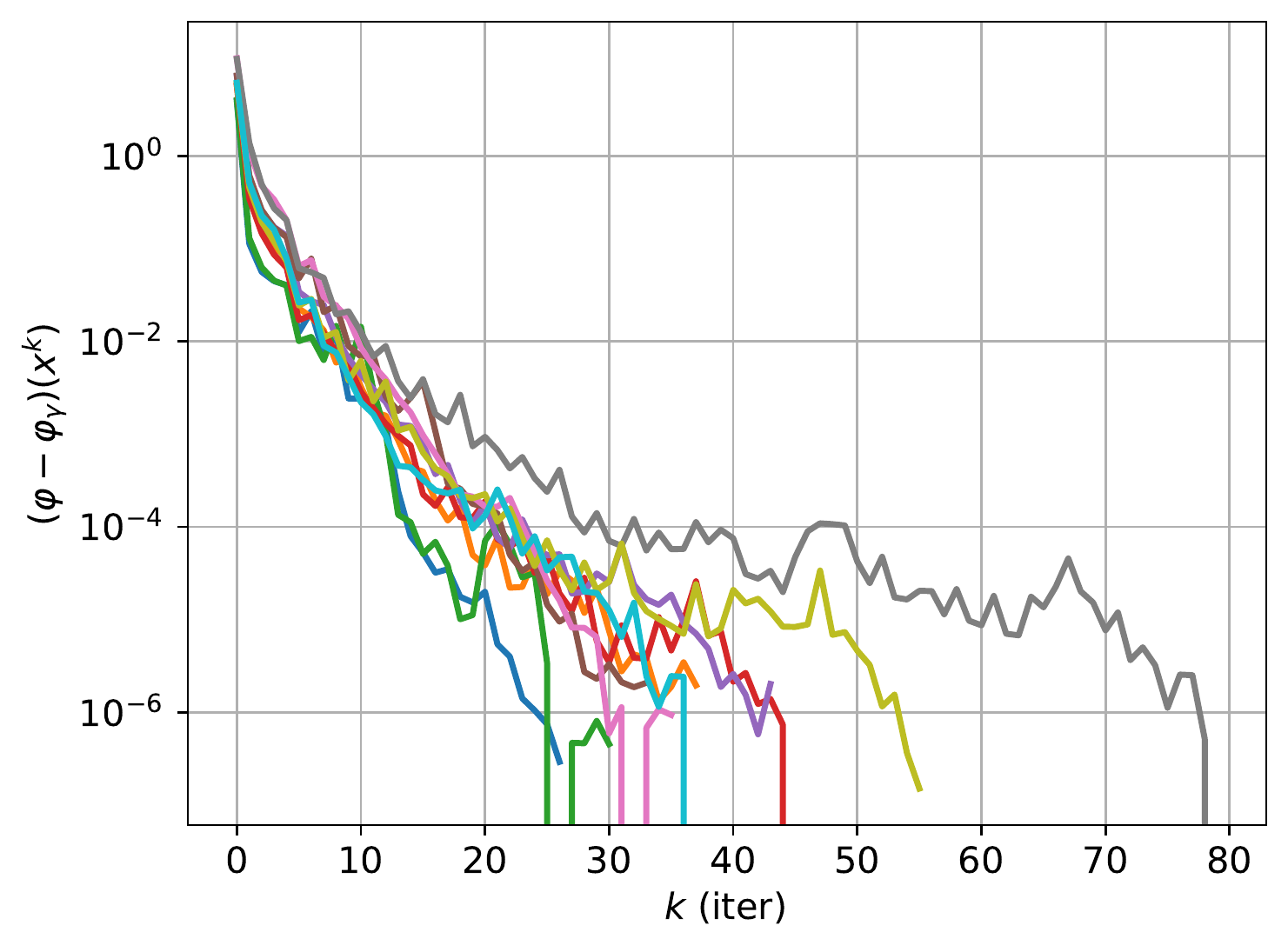}}}%
    \caption{Evolution of the objective $\varphi$, forward-backward envelope $\varphi_\gamma$, and their difference $\varphi - \varphi_\gamma$ for \Rev{deblurring} with PnP-LBFGS\textsuperscript{1}. These values are equal at the true minima, i.e., $\varphi_\gamma(x_*) = \varphi(x_*)$. Each curve corresponds to one of the 10 images from the CBSD10 dataset, evaluated with the first blur kernel and $\sigma = 7.65$.}
    \label{fig:deblur_moreau_difference}
\end{figure}

\subsection{Hyperparameter and Denoiser Choices}
The hyperparameters for the proposed PnP-LBFGS and the existing PnP-$\hat\alpha$PGD methods are as in \Cref{tab:hparams,tab:R1_aPGD_hparams}, respectively, chosen via grid search to maximize the PSNR over the set3c dataset for the respective image reconstruction problems. The hyperparameter grid for PnP-LBFGS is given in the subsequent subsections, while the grid for PnP-$\hat\alpha$PGD is given below. For the denoiser in our experiment, we use the pre-trained network $N_\sigma$ as in \cite{hurault2022proximal}. 

\RRev{The convergence conditions for }PnP-PGD and PnP-DRSdiff are that $g_\sigma$ has $L$-Lipschitz gradient for some $L<1$, and directly using the denoiser $D_\sigma$ maintains theoretical convergence. For PnP-DRS, the condition needs to be strengthened to $L<1/2$. In this case, the denoiser is replaced with an averaged denoiser of the form $(I + D_\sigma)/2 = I - \frac{1}{2}\nabla g_\sigma$, which gives convergence results but changes the underlying optimization problem. For PnP-LBFGS \Rev{and PnP-$\hat\alpha$PGD}, we use an \RRev{averaged} denoiser $D_\sigma^\alpha = I - \alpha \nabla g_\sigma$ which appears to have better performance, with the relaxation parameter $\alpha$ chosen as in \Cref{tab:hparams,tab:R1_aPGD_hparams}. As remarked in the introduction, adding the relaxation parameter $\alpha$ means that the effective Lipschitz constant of the potential gradient $\alpha \nabla g_\sigma$ is $\alpha L$, which alleviates divergence issues when $L>1$. \RRev{In this case, $D_\sigma^\alpha = \prox_{\phi_\sigma^\alpha}$ for some weakly convex $\phi_\sigma^\alpha$, and the previous computations hold with $g_\sigma$ replaced with $\alpha g_\sigma$.}


\Rev{For the parameters of the relaxed PnP-$\hat\alpha$PGD algorithm, we perform a grid search as in \cite{hurault2023RelaxedProxDenoiser}. To obtain the values of the denoiser averaging parameter $\alpha$ and the denoiser strength $\sigma_d$, we do a grid search for the set3c dataset with $\alpha \in \{0.6,0.7,0.8,0.85,0.9,1.0\}$ and $\sigma_d/\sigma \in \{0.5, 0.75, 1.0, 1.5, 2.0\}$, where the noise level is $\sigma=7.65$. The main difficulty in finding these hyperparameters is the dependence between $\alpha$ and $\sigma_d$, leading to poor reconstructions for many of these values. Given the denoiser averaging parameter $\alpha$, the other hyperparameters of PnP-$\hat\alpha$PGD are given by $\lambda = \frac{\alpha+1}{\alpha L_f}, \hat\alpha = \frac{1}{\lambda L_f}$. 

For the Lipschitz constant, we take $L_f=1$ for deblurring and $L_f=1/4$ for super-resolution with $s_{sr}=2,\,3$, as in \Cref{ssec:deblur,ssec:sr}. It appears approximating $L_f=1$ for super-resolution or $L_f=1/9 = 1/s_{sr}^2$ for $s_{sr}=3$ results in divergence, indicating sensitivity to their hyperparameters. We find the best values to be as in \Cref{tab:R1_aPGD_hparams}, with the grid search taken to maximize the PSNR over the set3c dataset. We additionally employ a stopping criterion based on the Lyapunov functional that PnP-$\hat\alpha$PGD minimizes, with the same sensitivity as PnP-DRS and PnP-DRSdiff \cite{hurault2023RelaxedProxDenoiser}. 
}

\Rev{The regularization parameter $\lambda$ for the underlying optimization problem is restricted for PnP-LBFGS in a manner similar to PnP-PGD and PnP-DRS (but not PnP-DRSdiff). For PnP-PGD and PnP-DRS, one condition for convergence is that $\lambda L_f < 1$ \cite{hurault2022proximal}. However, for PnP-LBFGS, \Cref{lem:gammaBddBelowWeakConv} gives the condition that $\gamma < (1-\beta)/(\lambda L_f)$, targeting stationary points of 
\begin{equation*}
    \varphi(x) = \frac{\lambda}{2}\|Ax-y\|^2 + \frac{1}{\gamma}\phi_\sigma.
\end{equation*}
We note that as $\lambda$ increases, the allowed $\gamma$ decreases, which correspondingly increases the smallest allowed coefficient $1/\gamma$ of the prior $\phi_\sigma$ at the same rate as $\lambda$. This puts an upper bound on the ratio between the fidelity term and the regularization term, which may be restrictive for low-noise applications.} 


The memory length for LBFGS was chosen to be $m=20$, with a maximum of 100 iterations per image. The denoiser $D_\sigma^\alpha$ is chosen with denoising strength $\sigma_d$ similar to that used for PnP-DRS as in \cite{hurault2022proximal}. By using different denoising strengths, we are able to further control regularization along with the scaling parameter $\lambda$. The step-sizes $\tau_k$ are chosen using an Armijo line search starting from $\tau_k = 1$, and multiplying by 0.5 if the $\varphi_\gamma$ decrease condition in Step 5 of \Cref{alg:pnpLBFGS} is not met \cite{armijo1966minimization, beck2017first}. 

We additionally introduce a stopping criterion based on the differences between consecutive iterates of the envelope $\varphi_\gamma(x^{k+1}) - \varphi_\gamma(x^k) < 10^{-5}$, as well as the envelope and objective $\varphi(x^k) - \varphi_\gamma(x^k)< 5 \times 10^{-5}$, where we stop if \Rev{at least one} criterion is met for 5 iterations in a row. We note that while the criteria can be strengthened, there is minimal change in the optimization result. \Rev{We label PnP-LBFGS with the envelope-based stopping criterion as PnP-LBFGS\textsuperscript{1}.} For completeness, we also consider the stopping criterion when the relative difference between consecutive function values of $\varphi$ is less than $10^{-8}$. \Rev{We label PnP-LBFGS with the objective change stopping criterion as PnP-LBFGS\textsuperscript{2}.} The PnP-LBFGS algorithms with the two stopping criteria are labeled with superscripts, as PnP-LBFGS\textsuperscript{1} and PnP-LBFGS\textsuperscript{2}, respectively.  We further use PnP-LBFGS without superscripts to refer to both methods together, which share their parameters.

All implementations were done in PyTorch, and the experiments were performed on an AMD EPYC 7352 CPU and a Quadro RTX 6000 GPU with 24GB of memory \cite{PyTorch}. The code for our experiments are publicly available\footnote{\url{https://github.com/hyt35/Prox-qN}}.

\subsection{PnP Methods Without Convergence Guarantees}
For further comparison, we additionally consider two non-provable PnP methods, namely DPIR \cite{zhang2021plug} and PnP-FISTA \cite{kamilov2017plug}. DPIR is based on the half-quadratic splitting, which splits $\prox_{f+g}$ into alternating $\prox_f$ and $\prox_g$ steps, and further replaces $\prox_g$ with a denoising step $D_{\sigma_k}$ in the spirit of PnP. PnP-FISTA is based on the fast iterative shrinkage-thresholding algorithm, which arises by applying a Nesterov-style acceleration to the forward-backward splitting \cite{kamilov2017plug,kamilov2023plug}. We note that neither of these methods correspond to critical points of functions in the existing literature. 

\begin{align}
    &\left\{ \begin{matrix*}[l]
        \alpha_k = \hat\lambda \sigma^2/\sigma_k^2, \\
        x_{k+1} = \prox_{f/2\alpha_k}(z_{k}), \\
        z_{k+1} = D_{\sigma_k}(x_k).
    \end{matrix*}\right. \label{eq:DPIR}\tag{DPIR} \\
    &\left\{ \begin{matrix*}[l]
        x_k = D_\sigma(y_k - \lambda \nabla f(y_k)),\\
        t_{k+1} = \frac{1+\sqrt{1+4t_k^2}}{2},\\
        y_{k+1} = x_k + \frac{t_k-1}{t_{k+1}} (x_k-x_{k-1}).
    \end{matrix*}\right. \label{eq:PnP-FISTA}\tag{PnP-FISTA} 
\end{align}
\subsubsection{DPIR}
To improve the performance, DPIR uses a decreasing noise regime as well as image transformations during iteration \cite[Sec. 4.2]{zhang2021plug}. To extend past eight iterations, we consider using the log-scale noise from $\sigma_d = 49$ to $\sigma_d = \sigma$ over 8 and 24 iterations for deblurring and super-resolution respectively, as recommended in the DPIR paper \cite[Sec. 5.1.1, 5.2]{zhang2021plug}. The scaling for the proximal term is determined by a scaling parameter $\hat\lambda$, which was chosen to be $\hat\lambda=0.23$ in the original work. \Cref{fig:dpir_divergence} shows that while DPIR achieves state-of-the-art performance in the low iteration regime, the PSNR begins to drop when HQS is extended past the number of iterations used in the original DPIR paper \cite{hurault2021gradient}. Moreover, DPIR appears to have poor performance in the low noise regime for the following image reconstruction experiments. \RRev{In the following experiments, we consider DPIR with the suggested 8 and 24 iterations for deblurring and super-resolution respectively, as well as extending up to 1000 iterations to check the convergence behavior.}

\subsubsection{PnP-FISTA}
The denoiser parameters for PnP-FISTA are considered to be either the parameters for PnP-LBFGS or PnP-PGD. \RRev{Proofs for PnP schemes} such as PnP-PGD or PnP-DRS \RRev{generally} rely on classical monotone operator theory, and showing that the denoiser satisfies the necessary assumptions. However, proofs of convergence of FISTA depend heavily on the convexity of the problem \cite{beck2009fast,chambolle2015convergence}, and non-convex proofs additionally require techniques or conditions such as adaptive backtracking \cite{goldstein2014field,ochs2019adaptive} or quadratic growth conditions \cite{aujol2023fista}. These techniques and conditions are difficult to convert and verify in the PnP regime, which translates to difficulties in showing convergence of the associated PnP-FISTA schemes. 

In the following experiments, we run the DPIR and PnP-FISTA methods for 1000 iterations unless stated otherwise to verify the convergence behavior. \Cref{fig:dpir_divergence,fig:fista_divergence} additionally demonstrate some common modes of divergence for DPIR and PnP-FISTA, with DPIR failing for low noise levels and PnP-FISTA failing with artifacts.

\subsection{Deblurring}\label{ssec:deblur}
For deblurring, 10 blur kernels were used, including eight camera shake kernels, a $9\times 9$ uniform kernel, and a $25\times25$ Gaussian kernel with standard deviation $\sigma_{\text{blur}} = 1.6$ \cite{levin2009understanding,hurault2022proximal}. Visualizations of the kernels can be found in the supplementary material. The blurring operator $A$ corresponds to convolution with circular boundary conditions. In this case, the transpose $A^\top$ can be easily implemented using a transposed convolution with circular boundary conditions. The blurring operator was previously scaled to have $\|A^\top A\|_{\text{op}}\approx 0.96$, which was verified using a power iteration. Thus, $\nabla f$ is approximately $0.96 \lambda$-Lipschitz.

We chose hyperparameters of PnP-LBFGS following a grid search maximizing the PSNR on the set3c dataset. The parameter grids are $\alpha \in \{0.5,\, 0.7,\, 0.9,\, 1.0\},\ \lambda \in \{0.8,\, 0.9,\, 1.0\},\ \gamma \in \{0.8,\, 0.85,\, 0.9,\, 1.0\}$, and $ \sigma_d/\sigma \in \{0.5,\, 0.75,\, 1.0,\, 1.5,\, 2.0\}$. Note that this choice obeys $\gamma < \min \{(1-\beta)/L_f,\, 1/(2M)\}$, since $\varphi_\sigma$ is at most $1/2$-weakly convex. We observe empirically that the step-size $\tau = 1$ is also a valid descent almost all of the time, verifying the claim that is required to prove the superlinear convergence as remarked in \Cref{rmk:constantTau}.  The underlying optimization problems are slightly different for PnP-LBFGS and PnP-PGD: for PnP-PGD, the fidelity regularization is chosen to be $\lambda = 0.99$, and the iterates converge to cluster points of $\varphi_{\text{PnP-PGD}}$:

\[\varphi_{\text{PnP-LBFGS}} = \frac{1}{2}\|Ax - y\|^2 + \phi_\sigma^\alpha ,\quad \varphi_{\text{PnP-PGD}} = \frac{0.99}{2} \|Ax - y\|^2 + \phi_\sigma.\]

\noindent We observe in \Cref{tab:PSNR_deblur} that the PnP-PGD and PnP-DRSdiff converge to very similar results since they both minimize the same underlying functional. However, the PnP iterations sometimes do not converge, as demonstrated \Rev{by the steadily decreasing PSNR in subfigures (d) and (g) of \Cref{fig:deblurConvergencePSNR}. This can be attributed} to the Lipschitz constant of $g_\sigma$ being greater than 1 at these iterates. \RRev{The use of the averaged denoiser $D_\sigma^\alpha$} in PnP-DRS and PnP-LBFGS \RRev{reduces divergence}, where we see convergence for these images as well. We generally observe that PnP-$\hat\alpha$PGD has the best performance in terms of PSNR, which can be attributed to the larger allowed value of $\lambda$. Nonetheless, we observe significantly faster convergence for PnP-LBFGS compared to the other methods to comparable PSNR values for each test image. 

Comparing with the non-provable PnP methods, we observe in \Cref{fig:deblurConvergencePSNR} that PnP-FISTA converges to the same PSNR as PnP-LBFGS \RRev{on CBSD10}, but has a worse performance \RRev{when averaged over} all CBSD68 images in \Cref{tab:PSNR_deblur}. This can be attributed to divergence of the method for denoisers where the Lipschitz constant of $\nabla g_\sigma$ is greater than 1. DPIR instead reaches its peak in the first couple of iterations, before decreasing to the fixed point as iterated by the denoiser with the final denoising strength $\sigma_d = \sigma$. \RRev{This results in worse performance of DPIR at iteration $10^3$ as compared to iteration 8, demonstrating the non-convergence and the current gap in performance between provable PnP and non-provable PnP.} 


\Cref{fig:deblurConvergencePSNR} and \Cref{fig:deblurConvergencelog2} additionally demonstrate the difference between the stopping criteria. The stopping criteria of PnP-LBFGS\textsuperscript{1} is sufficient for convergence to a reasonable PSNR, and allows for much earlier stopping. PnP-LBFGS\textsuperscript{2} stops after more iterates and demonstrates the significantly faster convergence of the residuals compared to the other considered PnP methods. Moreover, \Cref{fig:deblur_moreau_difference} shows the convergence curves of the objective $\varphi$ and forward-backward envelope $\varphi_\gamma$, which rapidly converge to the same value, verifying \Cref{prop:sharedMinimizer}. 


\begin{table}[]
\caption{Table of average PSNR (dB) comparing existing \RRev{provable and non-provable }PnP methods evaluated on the CBSD68 dataset compared to the proposed PnP-LBFGS methods. The time shown is the average reconstruction time per image. The PnP-LBFGS\textsuperscript{1} method is significantly faster per image due to the faster convergence \RRev{compared to the other provable PnP methods}. }
\label{tab:PSNR_deblur}
\centering
\begin{tabular}{l|llll}
\hline
$\sigma$     & 2.55  & 7.65  & 12.75 & Time (s) \\ \hline
PnP-LBFGS\textsuperscript{1}        & 31.19 & 27.95 & 26.61 & 5.80     \\
PnP-LBFGS\textsuperscript{2}        & 31.17 & 27.78 & 26.61 & 9.55     \\
PnP-PGD     & 30.57 & 27.80 & 26.61 & 25.93    \\
PnP-DRSdiff & 30.57 & 27.78 & 26.61 & 22.72      \\
PnP-DRS     & 31.54 & 28.07 & 26.60 & 19.26     \\
PnP-$\hat\alpha$PGD     & 31.52 & 28.15 & 26.74 & 15.66     \\ \hline
PnP-FISTA    & 30.24 & 27.15 & 26.60 & 24.32     \\
DPIR (iter $10^3$) & 27.40 & 27.58 & 26.46 & 19.62     \\
DPIR (iter 8)     & 32.01 & 28.34 & 26.86 & 0.55 \\\hline
\end{tabular}
\end{table} 

\begin{figure}%
    \centering
    \subfloat[\centering Ground Truth ]{{\includegraphics[height=2.9cm]{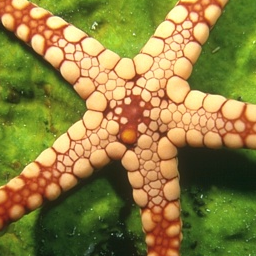}} } %
    \subfloat[\centering PnP-LBFGS\textsuperscript{1} (29.78dB)]{{\includegraphics[height=2.9cm]{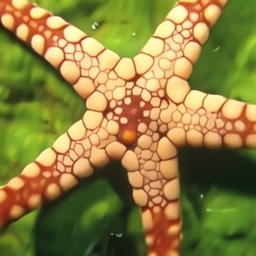}} } %
    \subfloat[\centering DPIR iter 100 (29.31dB)]{{\includegraphics[height=2.9cm]{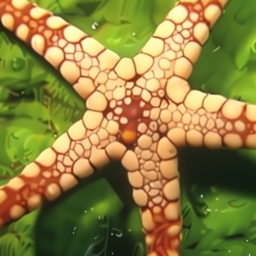}} } %
    \subfloat[\centering DPIR iter 8 (30.13dB)]{{\includegraphics[height=2.9cm]{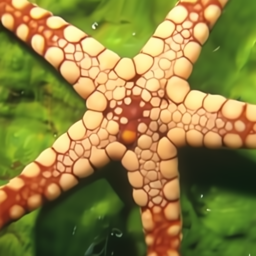}} }%
    \subfloat[\centering PnP-$\hat\alpha$PGD (30.00dB)]{{\includegraphics[height=2.9cm]{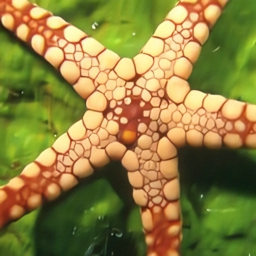}} } %

    \subfloat[\centering Corrupted ]{{\includegraphics[height=2.9cm]{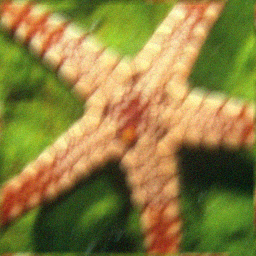}} } %
    \subfloat[\centering PnP-PGD (28.68dB)]{{\includegraphics[height=2.9cm]{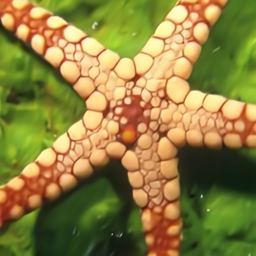}} } %
    \subfloat[\centering PnP-FISTA (29.58dB)]{{\includegraphics[height=2.9cm]{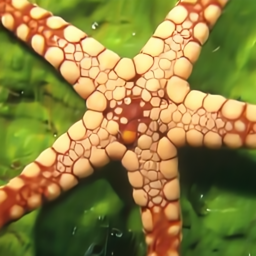}} } %
    \subfloat[\centering PnP-DRSdiff (28.66dB)]{{\includegraphics[height=2.9cm]{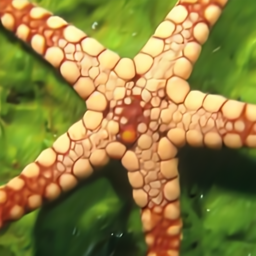}} } %
    \subfloat[\centering PnP-DRS (29.39dB)]{{\includegraphics[height=2.9cm]{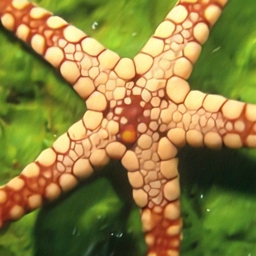}} } %
    \caption{Deblurring visualization using starfish image, with each method limited to a maximum of 100 iterations. Experiments are run with additive Gaussian noise $\sigma= 7.65$. PnP-LBFGS\textsuperscript{1} converges within the first 100 iterations, while the other PnP algorithms take longer to converge. Since the result of PnP-LBFGS\textsuperscript{1} and PnP-LBFGS\textsuperscript{2} are nearly identical, we show only PnP-LBFGS\textsuperscript{1}. \Rev{DPIR starts to decrease in PSNR after 8 iterations, leading to slightly worse performance.}}
    \label{fig:deblurImages}%
\end{figure}

\subsection{Super-resolution}\label{ssec:sr}
For super-resolution, we consider the forward operator with scale $s_{sr} \in \{2,3\}$ as $A = SK:\R^{n\times n} \rightarrow \R^{\lfloor n/s_{sr}\rfloor \times \lfloor n/s_{sr} \rfloor}$, which is a composition of a downsampling operator $S:\R^{n\times n} \rightarrow \R^{\lfloor n/s_{sr}\rfloor \times \lfloor n/s_{sr} \rfloor}$ and a circular convolution $K:\R^{n\times n}\rightarrow \R^{n\times n}$. The convolutions $K$ are Gaussian blur kernels with \Rev{blur strength given by standard deviations} $\sigma_{\text{blur}}=\{0.7,\,1.2,\,1.6,\,2.0\}$ as in \cite{zhang2021plug,hurault2022proximal}. For the PnP-LBFGS parameters, we chose hyperparameters maximizing the PSNR using a grid search on the set3c dataset over the following ranges: $\alpha \in \{0.5,\, 0.7,\, 0.9,\, 1.0\},\ \lambda \in \{1.0,\, 2.0,\, 3.0,\, 4.0\},\ \gamma \in \{0.8,\, 0.85,\, 0.9,\, 1.0\}$, and $\sigma_d/\sigma \in \{0.5,\, 0.75,\, 1.0,\, 1.5,\, 2.0\}$.

The Hessian $\nabla^2 f = \lambda A^\top A = \lambda K^\top S^\top S K$ is easily available, as $S^\top S : \R^{n\times n} \rightarrow \R^{n\times n}$ is a mask operator comprised of setting pixels with index not in $(s_{sr}\mathbb{Z})^2$ to zero, and $K^\top$ is a transposed convolution with circular boundary conditions. Note that on the image manifold, $S^\top S$ is approximately $1/s_{sr}^2$-Lipschitz, as we set $(s_{sr}^2-1)/s_{sr}^2$ of the pixels to zero. With $K$ being approximately 1-Lipschitz, we have that $A^\top A$ is approximately $1/s_{sr}^2$-Lipschitz.

The PnP-LBFGS parameters are $\beta = 0.01, \gamma = 1$, and $\lambda = 2,\, 1.5,\, 1$ for noise levels $\sigma = 2.55,\, 7.65,\, 12.75$ respectively. We can take \RRev{these values of} $\lambda$ since $L_f \approx 1/s_{sr}^2 \le 1/4$ and $\gamma = 1$ still obeys $\gamma < \min \{(1-\beta)/L_f,\, 1/(2M)\}$. The underlying functionals are as follows:
\[\varphi_{\text{PnP-LBFGS}} = \frac{\lambda_{\text{LBFGS}}}{2}\|Ax - y\|^2 + \phi_\sigma^\alpha ,\quad \varphi_{\text{PnP-PGD}} = \frac{0.99}{2} \|Ax - y\|^2 + \phi_\sigma.\]


\RRev{We observe in \Cref{tab:PSNR-superres} that the results for PnP-LBFGS are comparable to the other provable PnP methods, with overall faster wall-clock times. In \Cref{fig:srConvergencePSNR} and \Cref{fig:srConvergencelog2}, we are again able to see the difference between the stopping criteria. For the CBSD10 dataset, PnP-LBFGS\textsuperscript{1} converges on all images in under 40 iterations, while PnP-LBFGS\textsuperscript{2} sometimes requires all 100 iterations, and the other PnP methods take anywhere from 100 to $10^3$ iterations to converge. \Cref{fig:srConvergencelog2} shows again that the convergence of the residuals is significantly faster than the compared PnP methods per iteration. Note that for PnP-LBFGS, PnP-DRS and PnP-$\hat\alpha$PGD, we are allowed to choose larger values of the fidelity regularization term $\lambda$, leading to better reconstructions in the low noise regime compared to PnP-PGD and PnP-DRSdiff.} 

As seen in \Cref{subfig:SRconvlog2DPIR}, DPIR does not converge for super-resolution, and we observe an oscillating behavior of the residuals and PSNR. In contrast, PnP-FISTA is able to converge slightly faster than PnP-PGD, but does not converge for some images as seen by the decreasing PSNR for one curve in \Cref{fig:srConvergencePSNR}. Both PnP-FISTA and DPIR are able to perform reasonably for higher noise levels of $\sigma=12.75$, but have more divergence issues for lower noise levels, leading to reduced performance as seen in \Cref{tab:PSNR-superres}. \RRev{We again observe the gap in performance between DPIR at iteration $10^3$ and at iteration 24 as suggested in the original DPIR work. The performance gap between DPIR and provable PnP methods is less apparent for super-resolution as opposed to deblurring, as observed in \cite{hurault2021gradient}.}

    
\begin{figure}[h]%
    \centering
    \subfloat[\centering PnP-LBFGS\textsuperscript{1} (27.87dB)]{{\includegraphics[height=3.50cm]{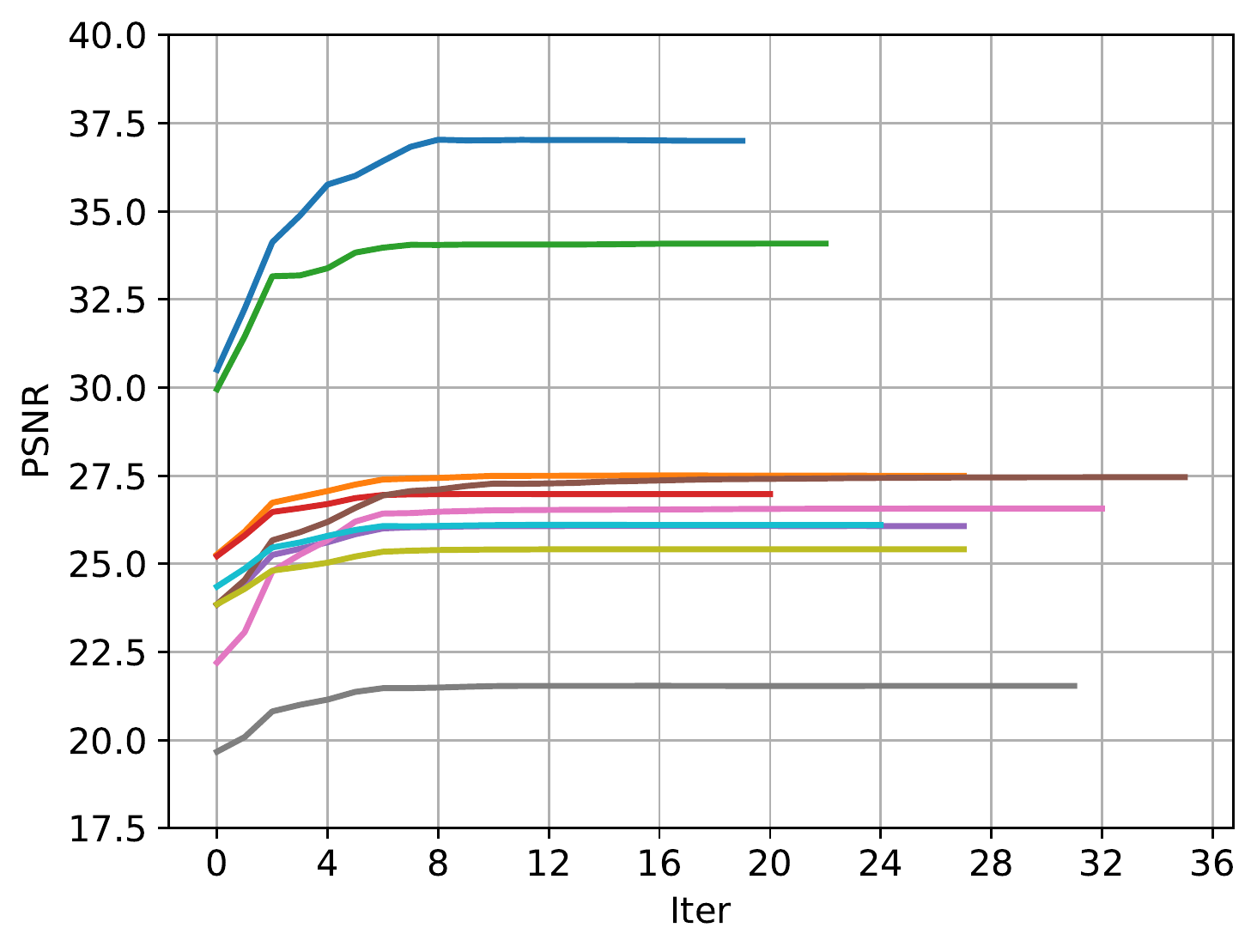} }}%
    \subfloat[\centering PnP-LBFGS\textsuperscript{2} (27.87dB)]{{\includegraphics[height=3.50cm]{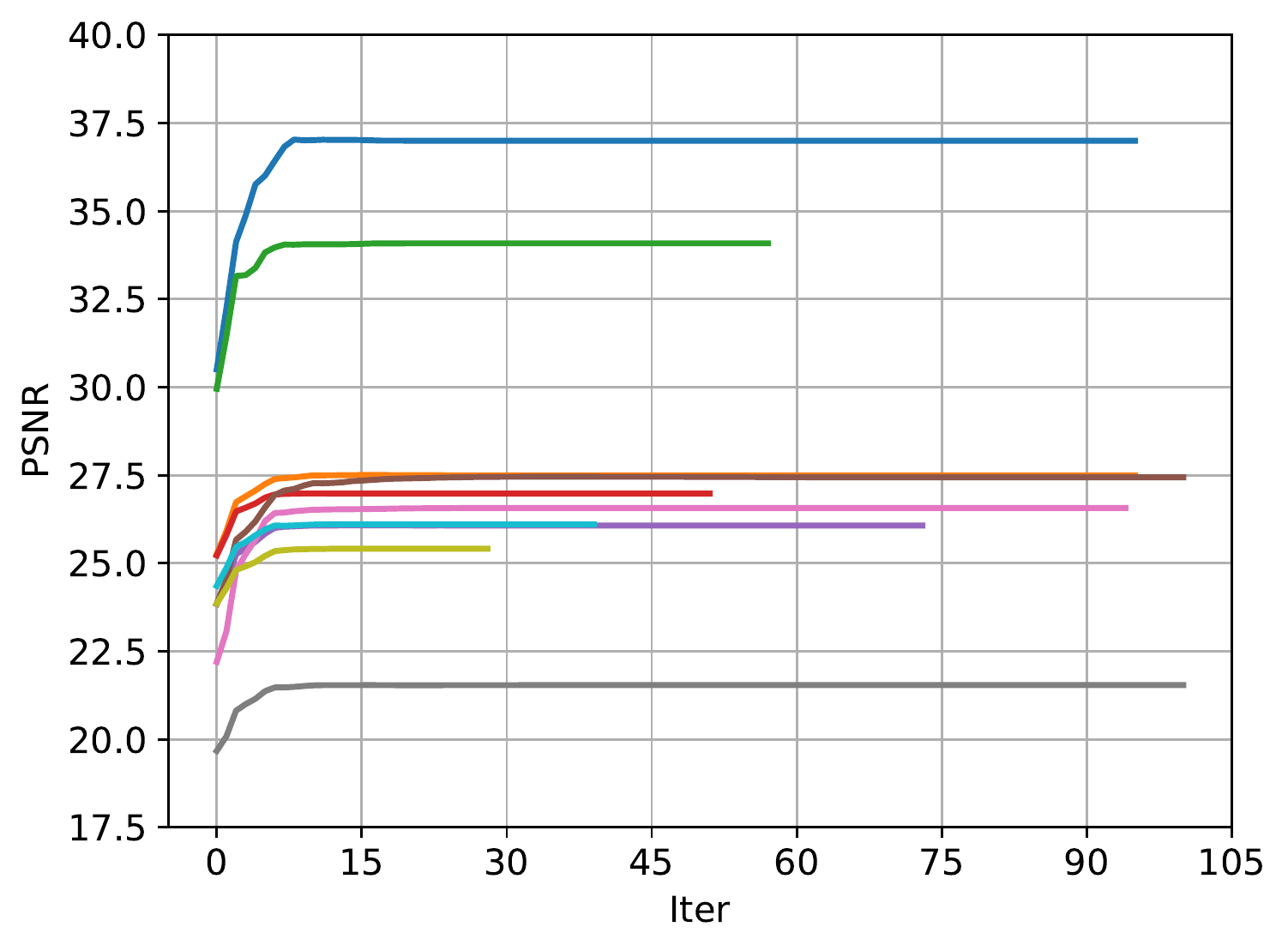} }}%
    \subfloat[\centering DPIR (27.58dB)]{{\includegraphics[height=3.50cm]{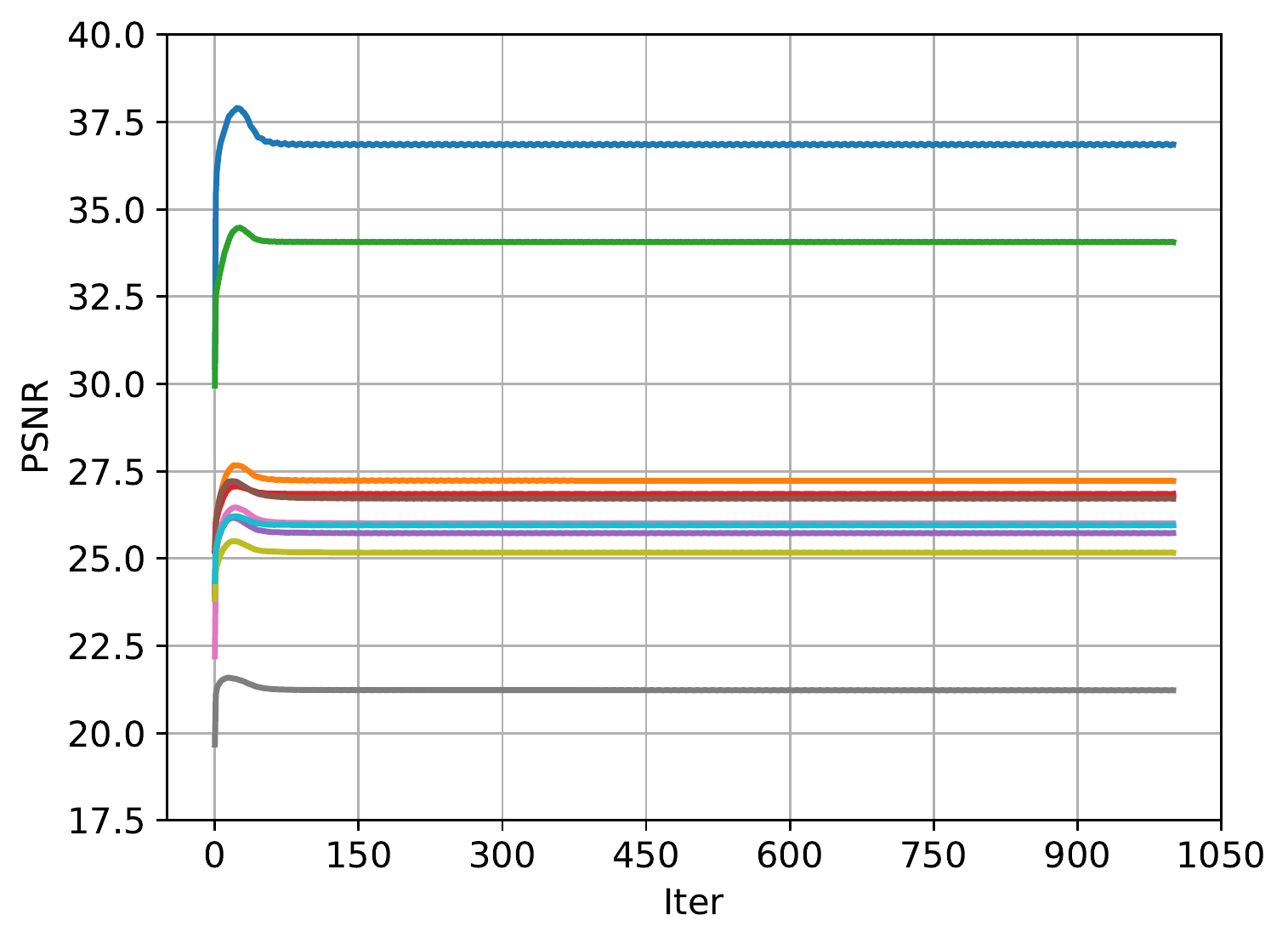} }}%
        
    \subfloat[\centering PnP-PGD (27.82dB)]{{\includegraphics[height=3.50cm]{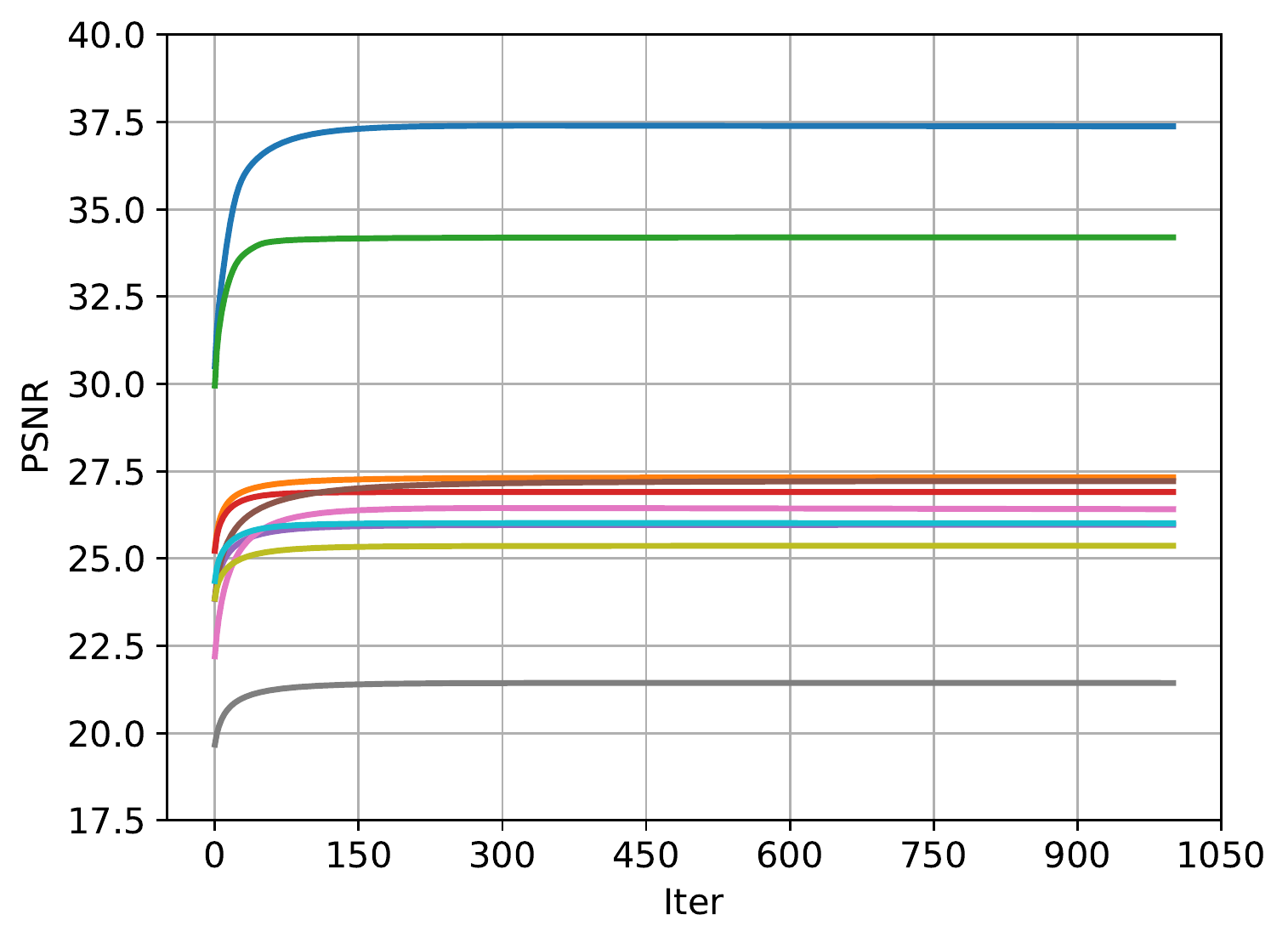}}}
    \subfloat[\centering PnP-$\hat\alpha$PGD (27.86dB)]{{\includegraphics[height=3.50cm]{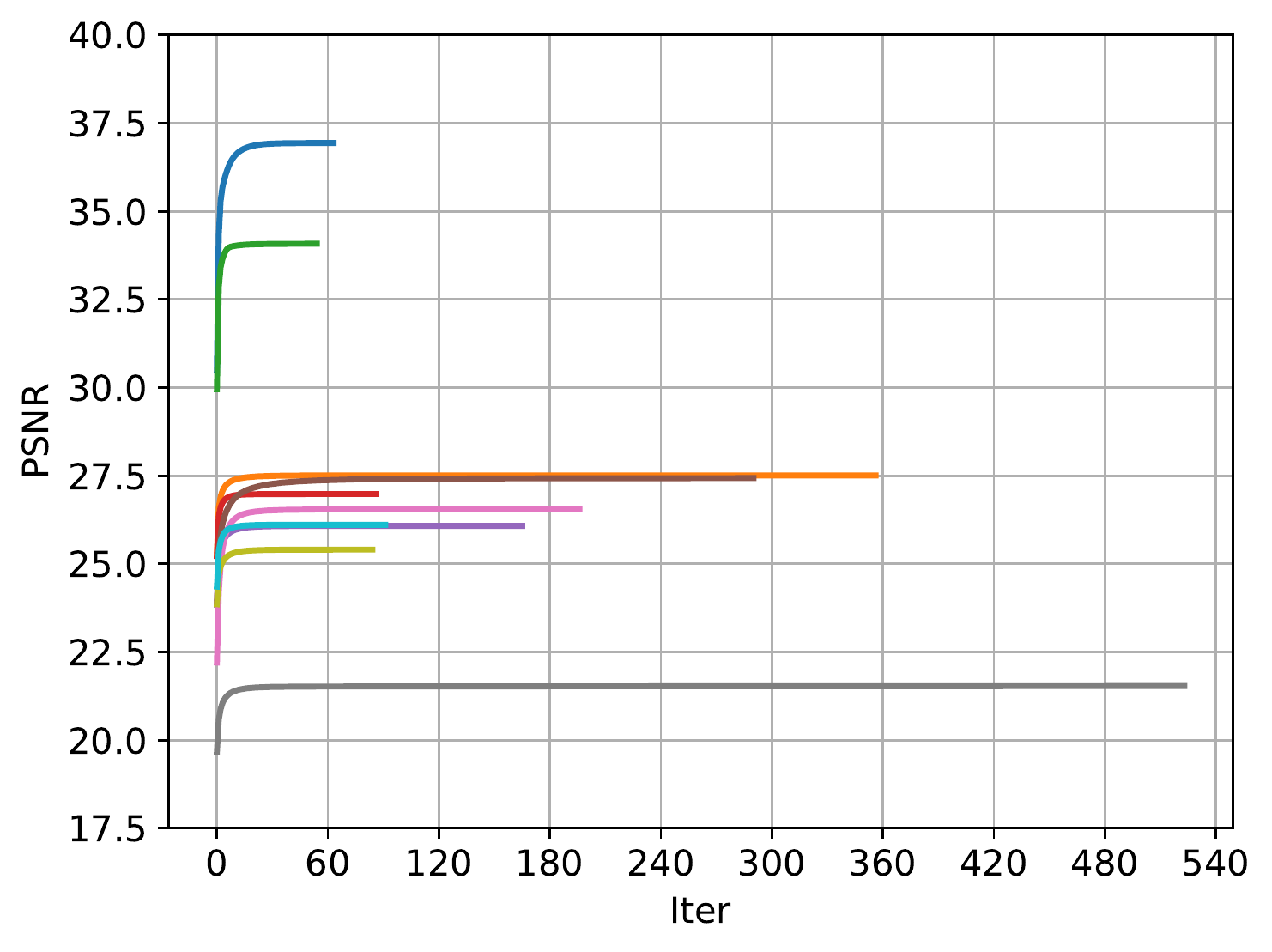}}}
    \subfloat[\centering PnP-FISTA (27.71dB)]{{\includegraphics[height=3.50cm]{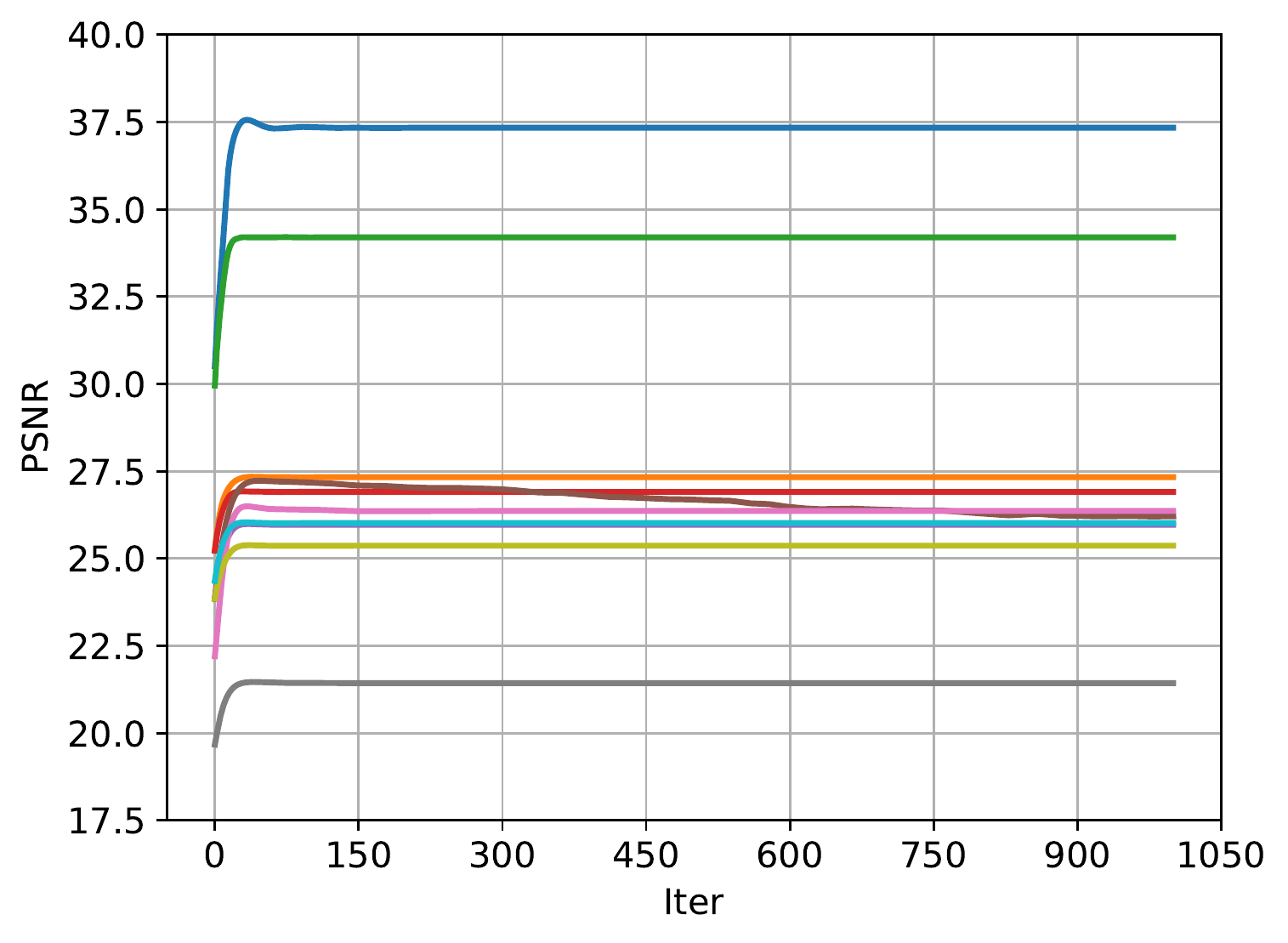}}}
    
    \subfloat[\centering PnP-DRSdiff (27.69dB)]{{\includegraphics[height=3.50cm]{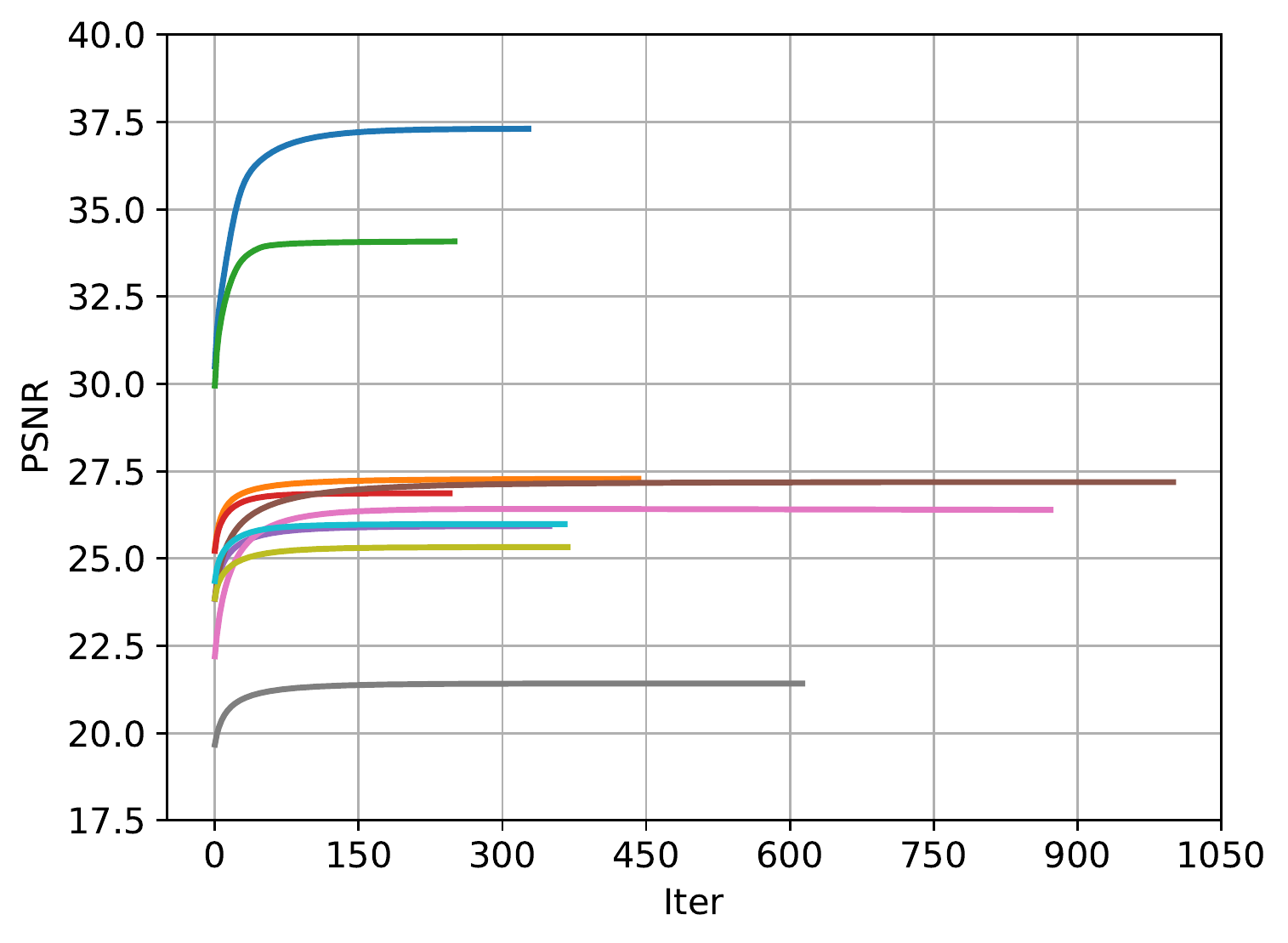}}}%
    \subfloat[\centering PnP-DRS (27.78dB)]{{\includegraphics[height=3.50cm]{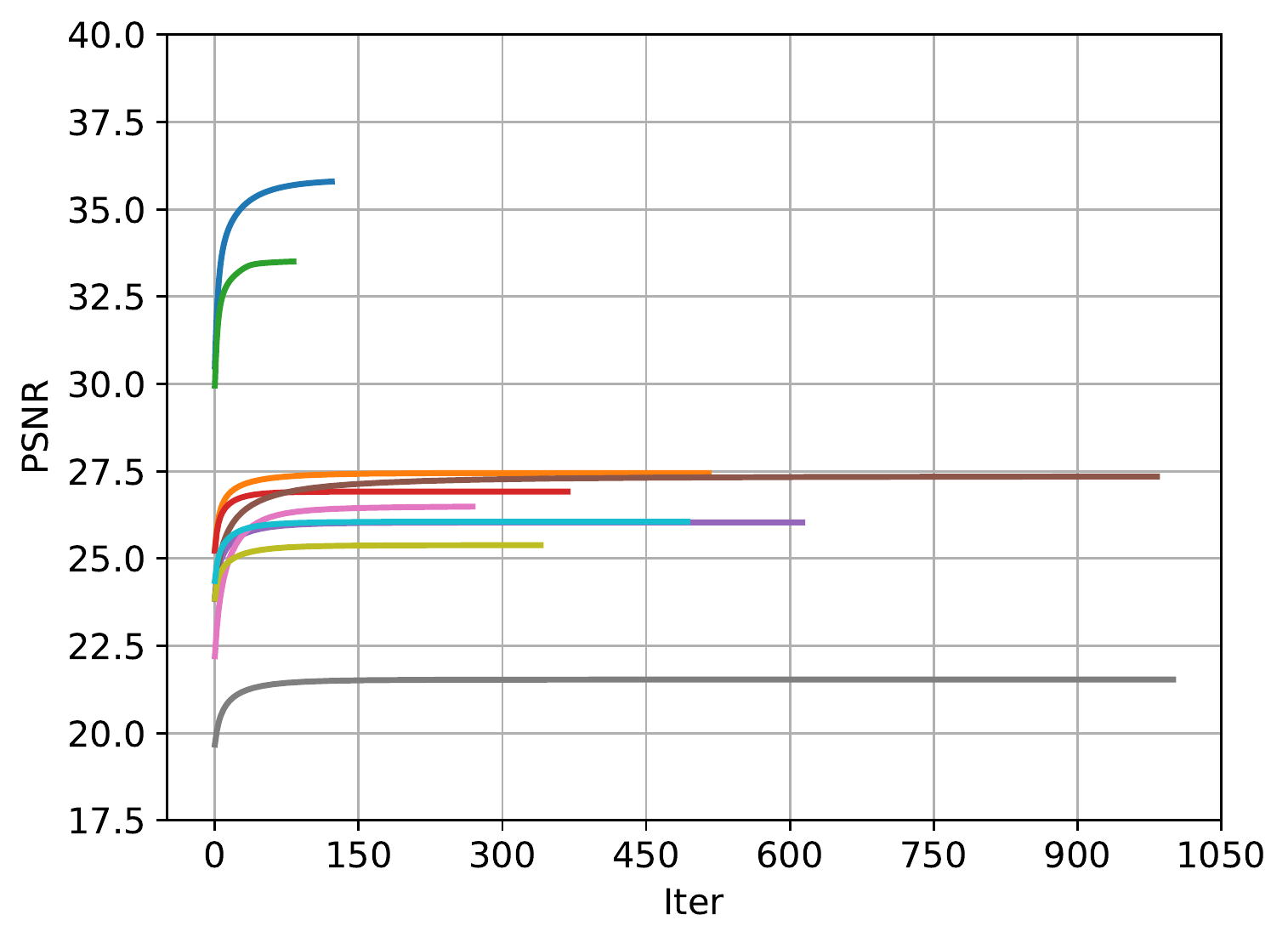}}}%
    \caption{Convergence of the PSNR (dB) of the various curves for super-resolution, with the average dB in brackets. Each curve corresponds to one of the 10 images from the CBSD10 dataset, evaluated with the Gaussian blur kernel with standard deviation $\sigma_{\text{blur}}=1.2$ and additive noise $\sigma = 7.65$, with scale $s_{sr}=2$. We observe the convergence of PSNRs in under 40 iterations for PnP-LBFGS\textsuperscript{1}, much faster than the compared PnP methods.}
    \label{fig:srConvergencePSNR}
\end{figure}
\begin{figure}[h]%
    \centering
    \subfloat[\centering PnP-LBFGS\textsuperscript{1} ]{{\includegraphics[height=3.50cm]{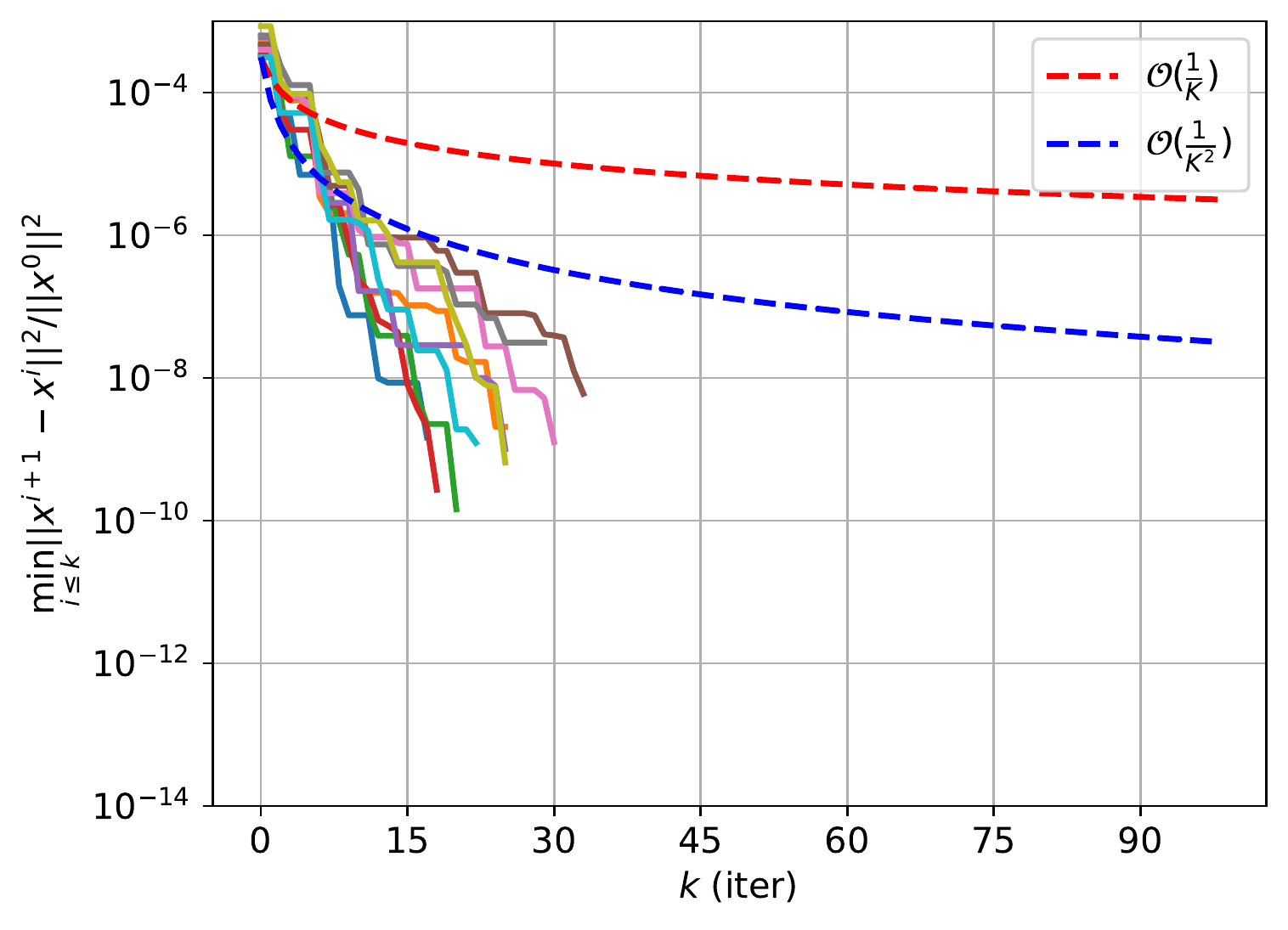}}}%
    \subfloat[\centering PnP-LBFGS\textsuperscript{2} ]{{\includegraphics[height=3.50cm]{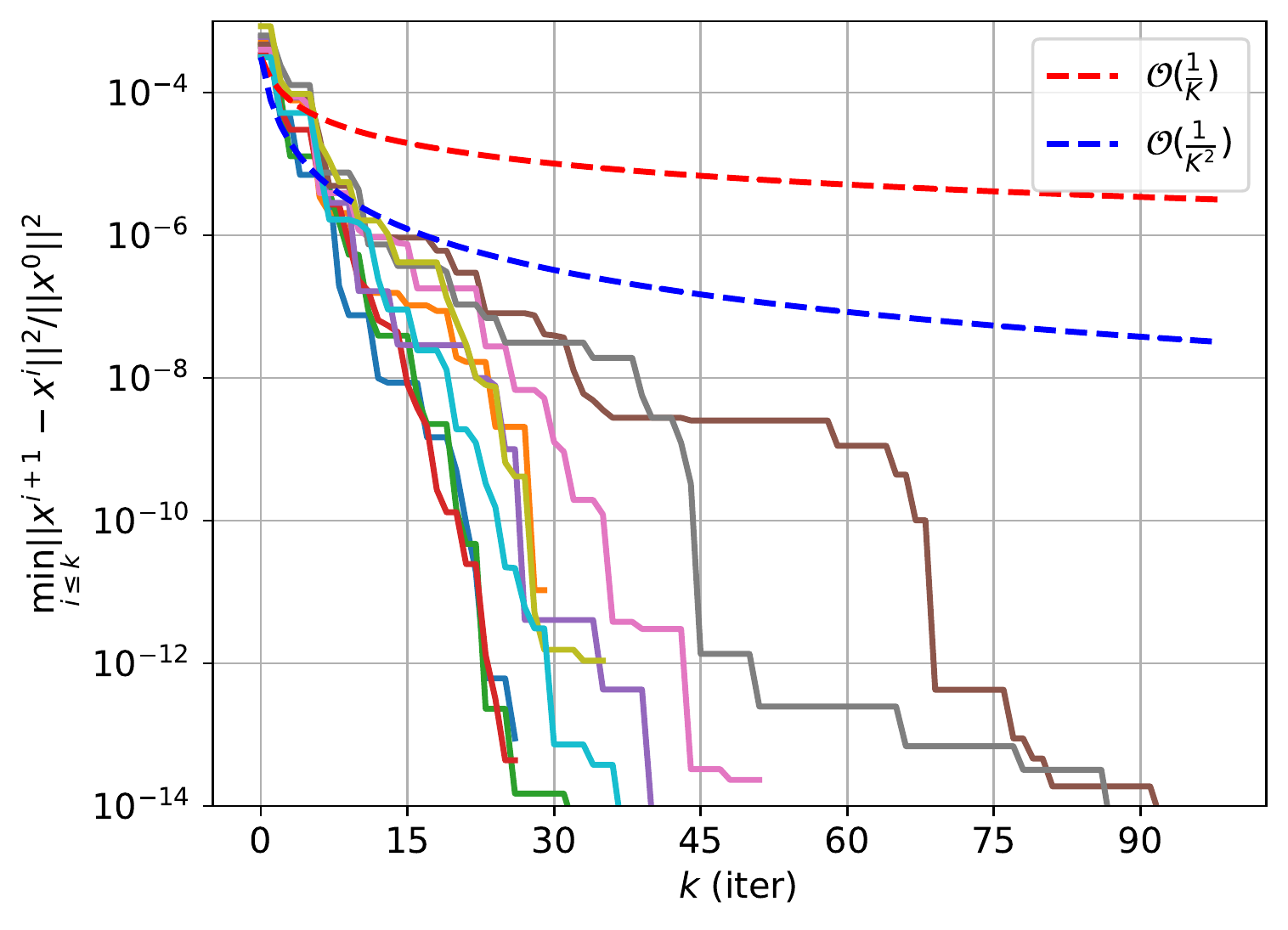}}}%
    \subfloat[\centering DPIR\label{subfig:SRconvlog2DPIR}]{{\includegraphics[height=3.50cm]{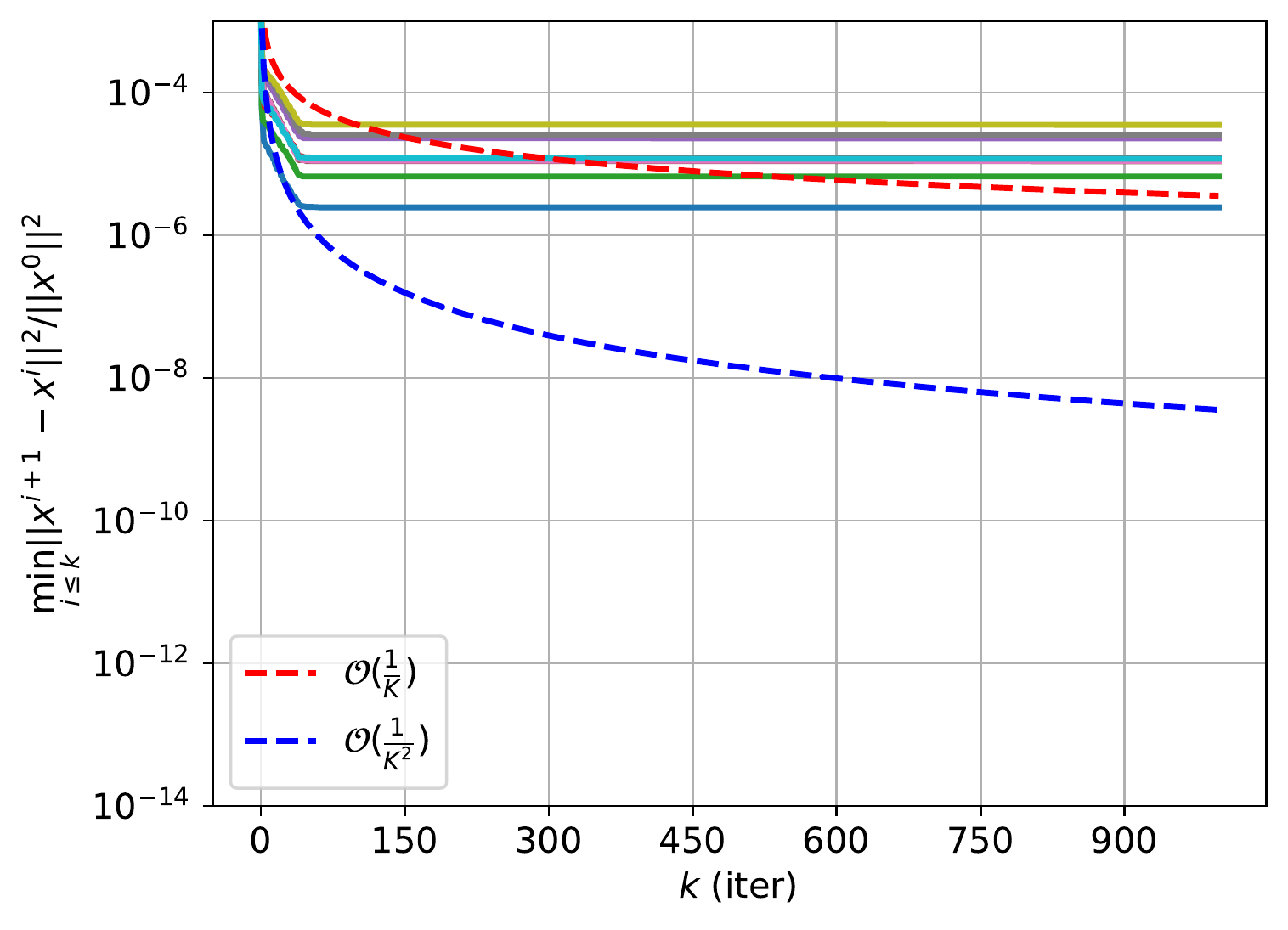}}}%
    
    \subfloat[\centering PnP-PGD]{{\includegraphics[height=3.50cm]{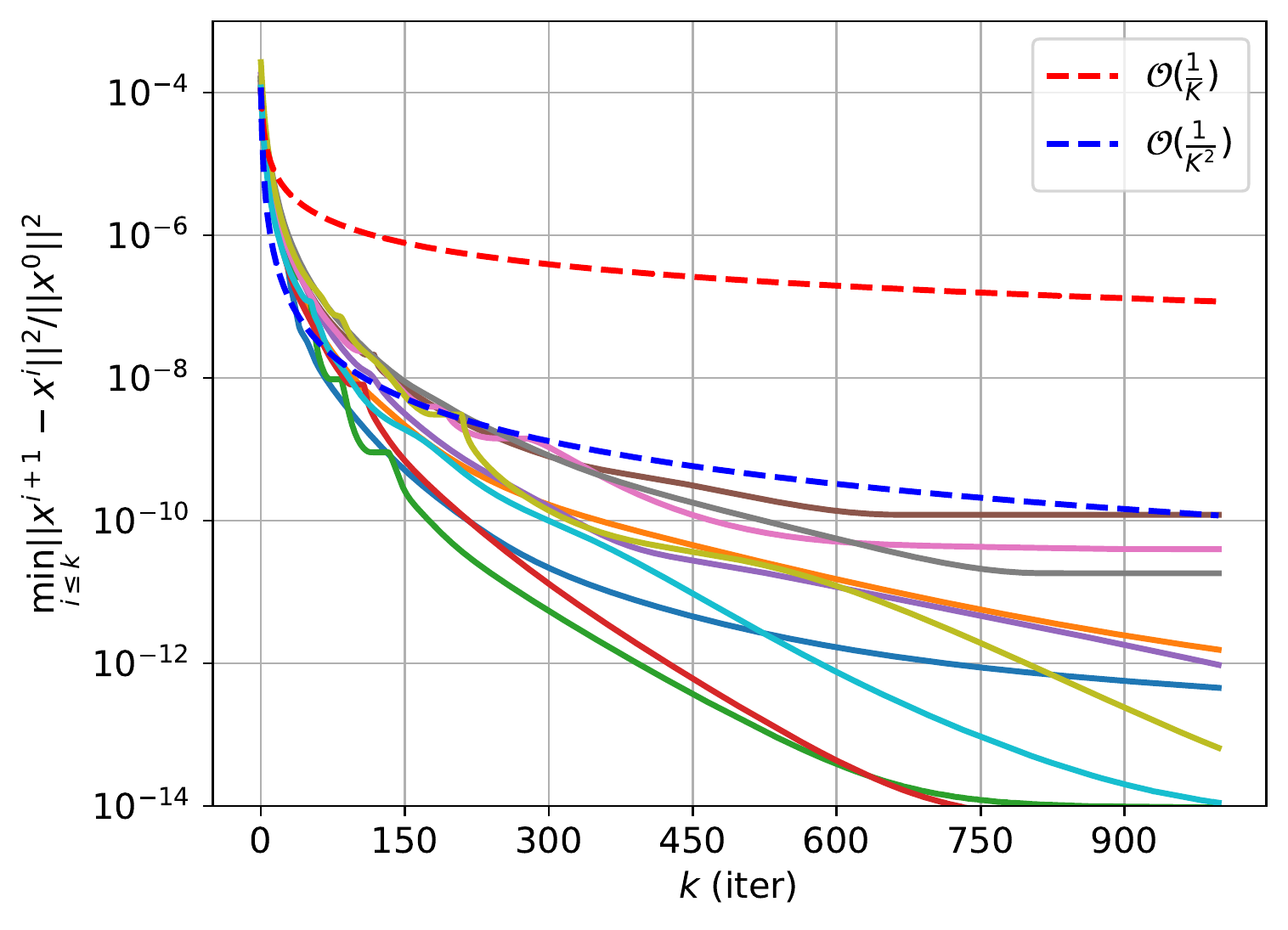}}}%
    \subfloat[\centering PnP-$\hat\alpha$PGD]{{\includegraphics[height=3.50cm]{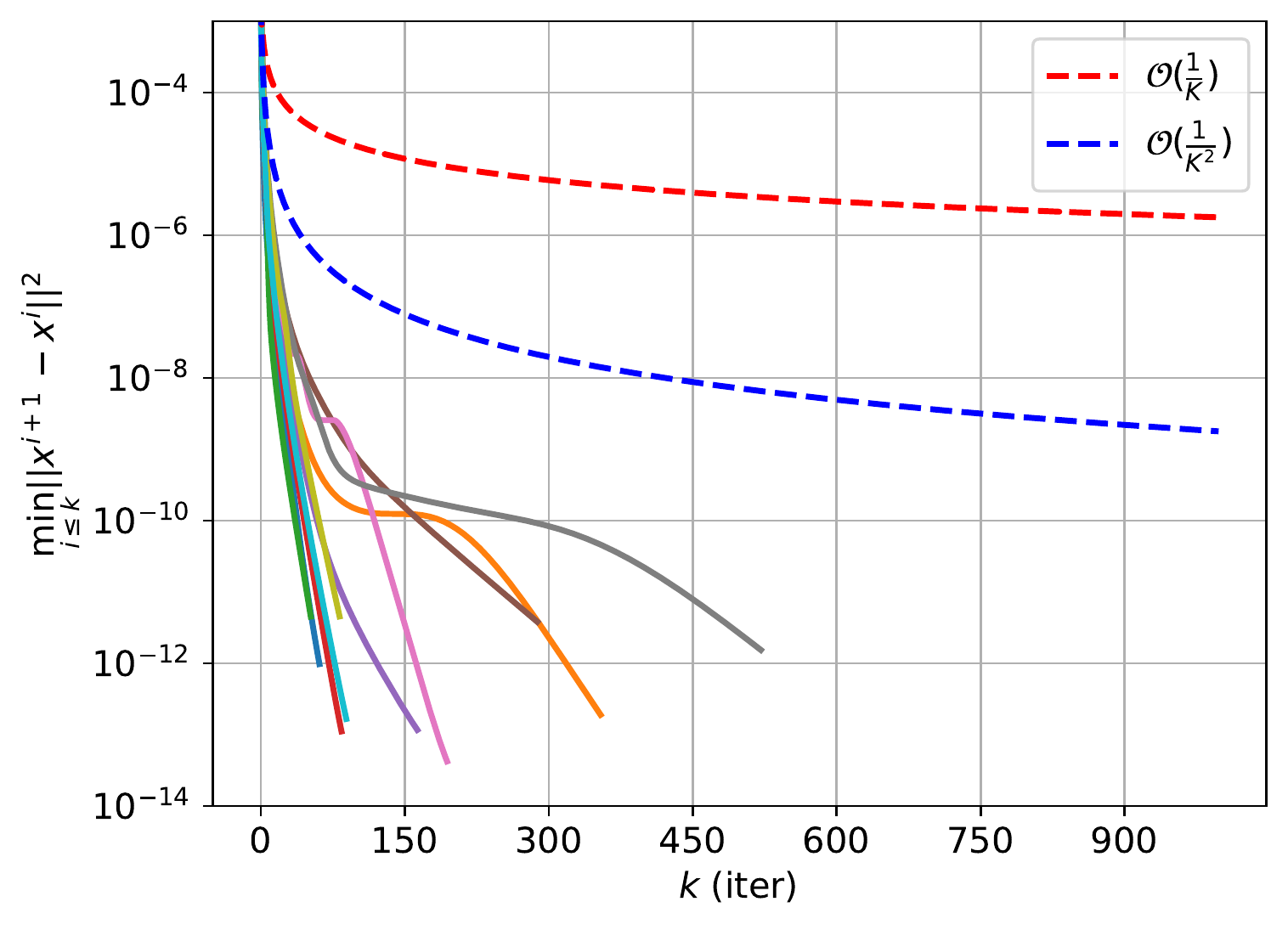}}}%
    \subfloat[\centering PnP-FISTA]{{\includegraphics[height=3.50cm]{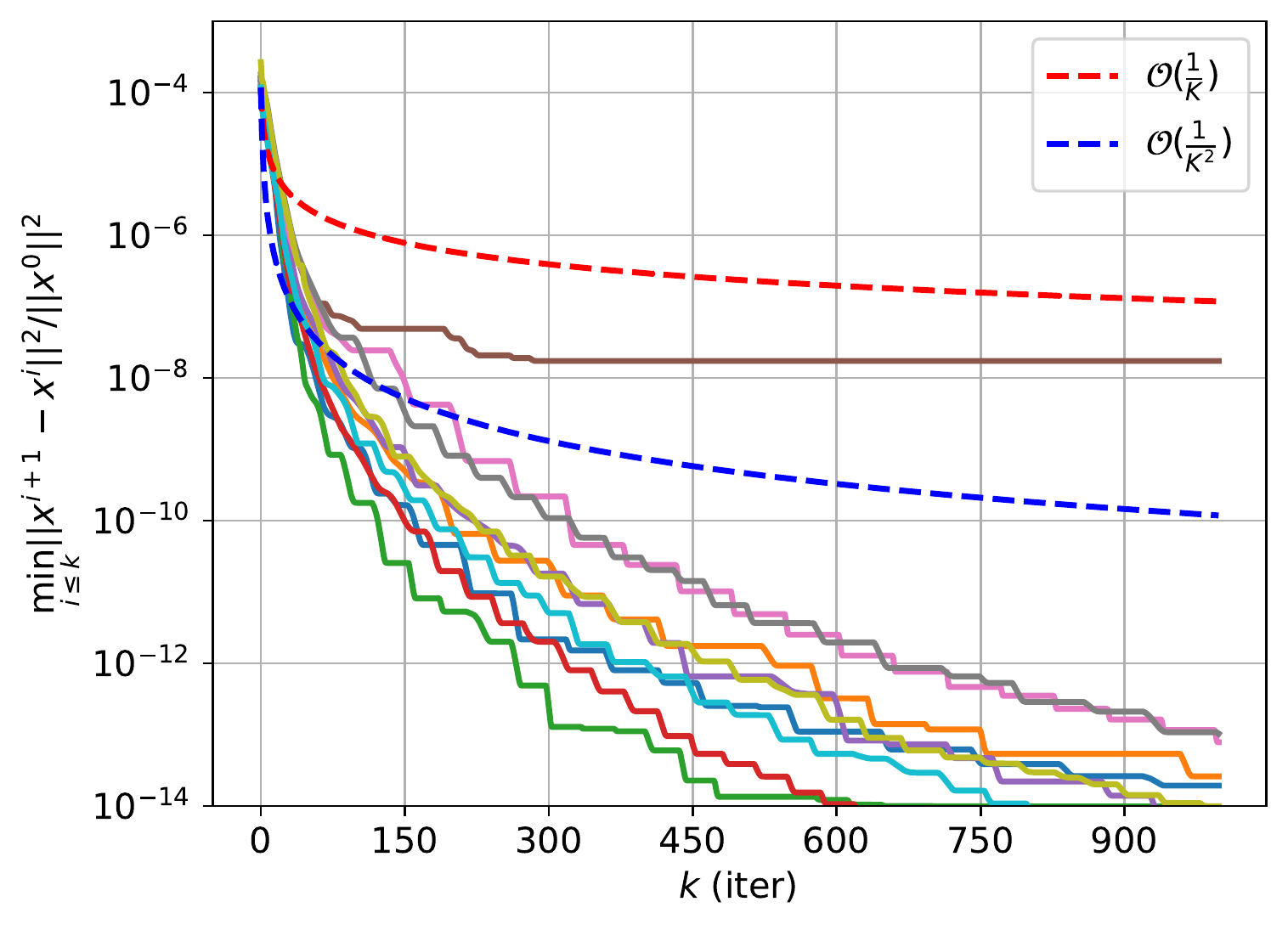}}}%
    
    \subfloat[\centering PnP-DRSdiff]{{\includegraphics[height=3.50cm]{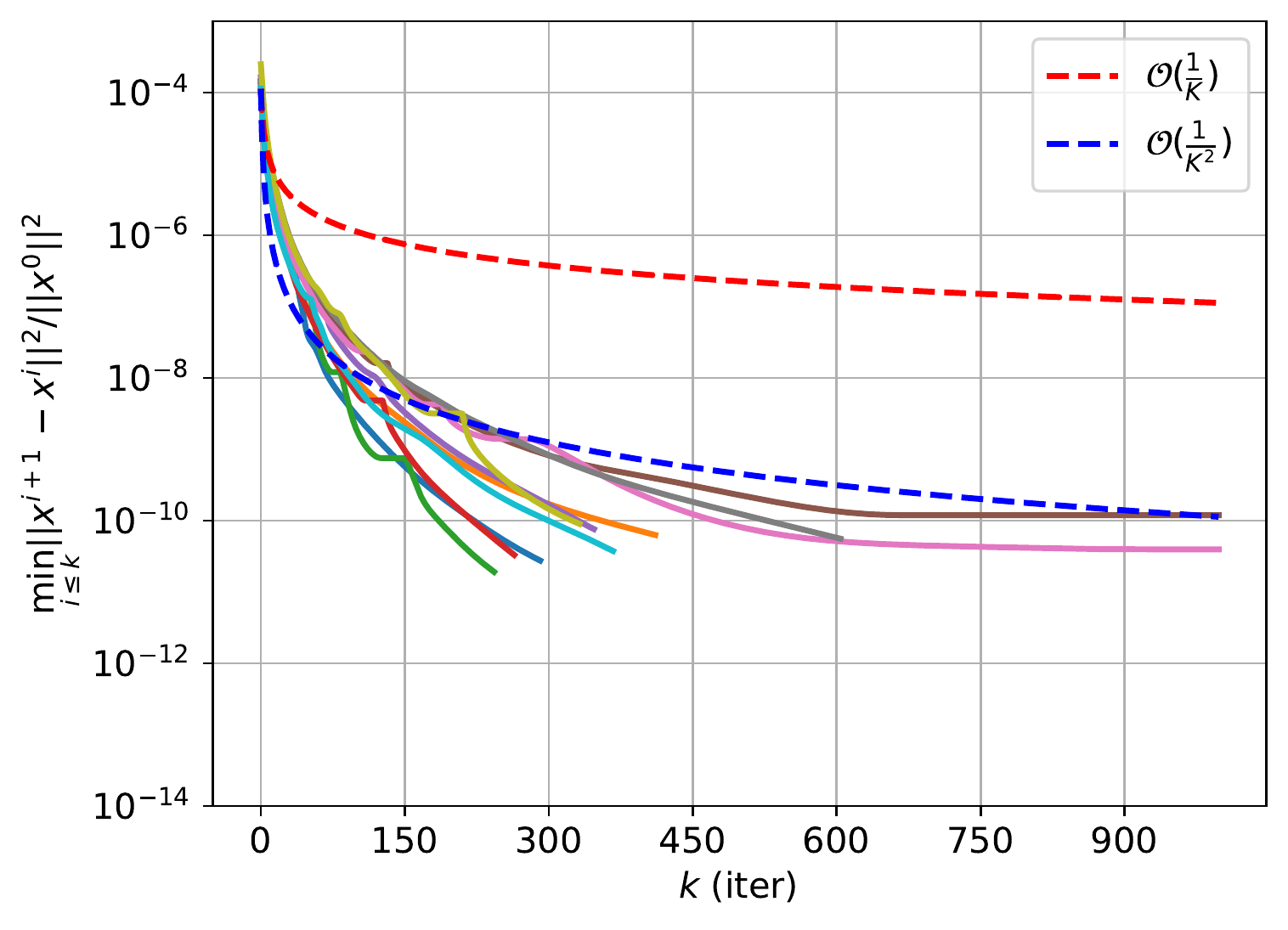}}}%
    \subfloat[\centering PnP-DRS]{{\includegraphics[height=3.50cm]{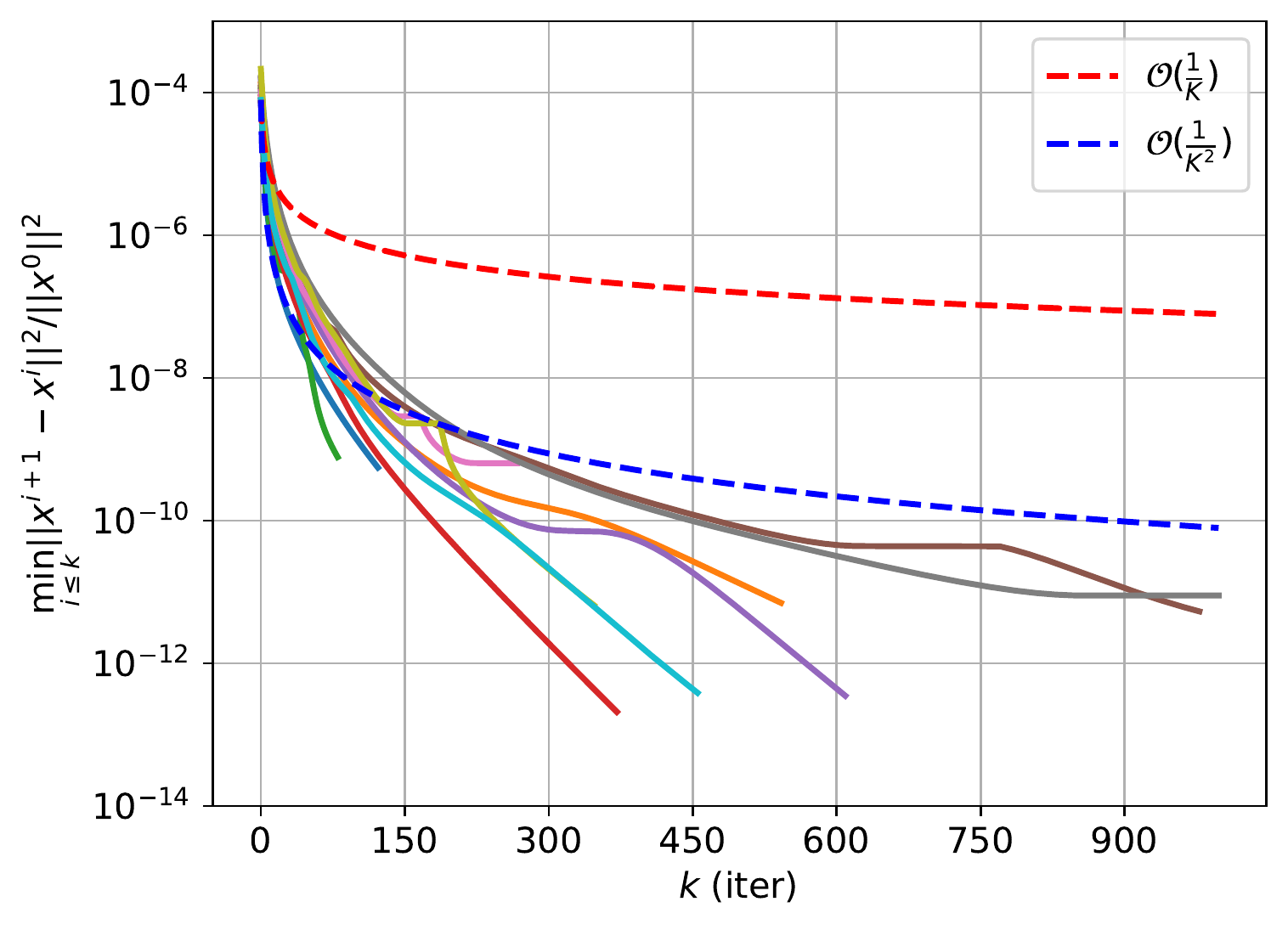}}}%
    \caption{Convergence of the residuals $\min_{i\le k} \|x^{i+1}-x^i\|^2/\|x^0\|^2$ of the various methods \Rev{for super-resolution}. Each curve corresponds to one of the 10 images from the CBSD10 dataset, evaluated with the Gaussian blur kernel with standard deviation $\sigma_{\text{blur}}=1.2$ and additive noise $\sigma/255 = 7.65$, with scale $s_{sr}=2$. PnP-LBFGS\textsuperscript{2} demonstrates significantly faster residual convergence of the proposed method.}
    \label{fig:srConvergencelog2}%
\end{figure}%
\begin{table}
\caption{Table of averaged PSNR (dB) corresponding to the competing PnP methods evaluated on the CBSD68 dataset for super-resolution, as compared with the proposed PnP-LBFGS method. The time is the average reconstruction time per image for $\sigma= 7.65$. The performance of PnP-LBFGS is almost identical to the compared \Rev{provable} PnP methods due to minimizing the same variational form, but with faster convergence. }
\label{tab:PSNR-superres}
\begin{tabular}{lllllllll}
\hline
Scale                                 & \multicolumn{4}{c}{$s=2$}                         & \multicolumn{4}{c}{$s=3$}        \\
$\sigma$                          & 2.55  & 7.65  & 12.75 & Time (s)                  & 2.55  & 7.65  & 12.75 & Time (s) \\ \hline
\multicolumn{1}{l|}{PnP-LBFGS\textsuperscript{1}}        & 27.89 & 26.62 & 25.80 & \multicolumn{1}{l|}{3.19} & 26.12 & 25.32 & 24.68 & 4.80    \\
\multicolumn{1}{l|}{PnP-LBFGS\textsuperscript{2}}        & 27.89 & 26.62 & 25.80 & \multicolumn{1}{l|}{9.81} & 26.12 & 25.30 & 24.68 & 13.15    \\
\multicolumn{1}{l|}{PnP-PGD}     & 27.44 & 26.57 & 25.82 & \multicolumn{1}{l|}{25.99}     & 25.60 & 25.20 & 24.63 & 37.33    \\
\multicolumn{1}{l|}{PnP-DRSdiff} & 27.44 & 26.58 & 25.82 & \multicolumn{1}{l|}{18.24}     & 25.60 & 25.19 & 24.63 & 32.83    \\
\multicolumn{1}{l|}{PnP-DRS}     & 27.93 & 26.61 & 25.79 & \multicolumn{1}{l|}{15.74}     & 26.13 & 25.29 & 24.67 & 27.00    \\ 
\multicolumn{1}{l|}{PnP-$\hat\alpha$PGD}     & 27.94 & 26.62 & 25.72 & \multicolumn{1}{l|}{4.24}     & 26.11 & 25.32 & 24.69 & 8.78    \\ \hline
\multicolumn{1}{l|}{PnP-FISTA}     & 26.38 & 26.44 & 25.79 & \multicolumn{1}{l|}{24.61}     & 24.96 & 25.15 & 24.63 & 33.13    \\
\multicolumn{1}{l|}{DPIR (iter $10^3$)}     & 18.58 & 26.36 & 25.74 & \multicolumn{1}{l|}{19.58}     & 17.53 & 24.96 & 24.55 & 19.67    \\
\multicolumn{1}{l|}{DPIR (iter 24)}     & 27.82 & 26.60 & 25.85 & \multicolumn{1}{l|}{0.98}     & 26.06 & 25.29 & 24.67 & 0.97    \\
\hline
\end{tabular}
\end{table}

\subsection{Computational Complexity}
While each iteration of PnP-LBFGS has increased complexity, we observed convergence in much fewer iterations. In this section, we outline the computational requirements for the number of neural network $N_\sigma$ evaluations, denoising steps $D_\sigma$, as well as computations of $\nabla f$ and $\nabla^2 f$ \RRev{required per iteration}. Note that if a closed form for $\nabla^2 f$ is intractable, \RRev{computations of \eqref{eq:gradVarphiGamma2}} can be replaced with Hessian-vector products, available in many deep learning libraries. 

We can calculate $T_\gamma$ and $R_\gamma$ together using one call each of $\nabla f$ and $D_\sigma$. From \eqref{eqs:pnpMINFBEclosedform}, $\varphi_\gamma$ requires $\nabla f$ and $g_\sigma$, which in turn requires $N_\sigma$. $\nabla \varphi_\gamma$ has a closed form, which requires $R_\gamma$ and an evaluation of $\nabla^2 f$. 

Consider a single iteration of PnP-LBFGS. We first compute $\nabla \varphi_\gamma (x^k)$ and $\varphi_\gamma(x^k)$. Computing $d^k$ using L-BFGS does not require any additional evaluations of $D_\sigma,N_\sigma, \nabla f$ or $\nabla^2 f$, as the secants and differences will have been computed in the previous iteration. For each test of $w^k$, we need to compute a single iteration of $\varphi_\gamma$, which takes one evaluation each of $\nabla f$ and $N_\sigma$.  Once a suitable $w^k$ is found, we compute $T_\gamma(w^k)$ and $R_\gamma(w^k)$ together using the last stored $\nabla f(w^k)$, requiring only one additional $D_\sigma$ operation. For the secant $y^k$, we require an evaluation of $\nabla \varphi_\gamma(w^k)$, which requires only one additional $\nabla^2 f$ evaluation. This concludes one iteration. 

To evaluate the proposed stopping criteria for PnP-LBFGS\textsuperscript{1}, we are also required to compute $\varphi(x^{k+1})$ from \eqref{eq:lastGather}. Note we already have $g_\sigma(w^k - \gamma \nabla f(w^k))$ from computing $\varphi_\gamma(w^k)$, and $T_\gamma(w^k) = x^k$, hence we get $\varphi(x^{k+1})$ with no further evaluations needed.

In total, assuming we need \RRev{$T$ tests for $\tau_k$}, the per iteration-cost is 
\begin{equation} \label{eq:pnplbfgsComplexity}
    \begin{pmatrix}
        \#N_\sigma\\
        \#D_\sigma\\
        \#\nabla f\\
        \#\nabla^2 f
    \end{pmatrix}_{\text{PnP-LBFGS}} = 
    \underbrace{
    \begin{pmatrix}
        1\\1\\1\\1
    \end{pmatrix}}_{\substack{\nabla \varphi_\gamma(x^k),\\ \varphi_\gamma(x^k)}} +\ T 
    \underbrace{\begin{pmatrix}
        1\\0\\1\\0
    \end{pmatrix}}_{\text{test } w^k} + 
    \underbrace{\begin{pmatrix}
        0\\1\\0\\0
    \end{pmatrix}}_{\substack{T_\gamma(w^k), \\ R_\gamma(w^k)}} + 
    \underbrace{\begin{pmatrix}
        0\\0\\0\\1
    \end{pmatrix}}_{\nabla \varphi_\gamma(w^k)} = 
    \begin{pmatrix}
        T+1\\2\\T+1\\2
    \end{pmatrix}.
\end{equation}

At later iterations, the number of tests is only $T=1$, since the step-size $\tau=1$ is accepted almost always. Therefore, later iterations require two of $N_\sigma, D_\sigma, \nabla f$ and $\nabla^2 f$. For comparison, PnP-PGD requires one evaluation each of $D_\sigma$ and $\nabla f$, and the PnP-DRS methods require one evaluation each of $D_\sigma$ and $\prox_f$. Note that for these methods to test their stopping criteria by computing $\varphi$, they also require one evaluation of $g_\sigma$ and hence of $N_\sigma$ \cite{hurault2022proximal}. These methods thus have complexity
\begin{equation*}
    \begin{pmatrix}
        \#N_\sigma\\
        \#D_\sigma\\
        \#\nabla f
    \end{pmatrix}_{\text{PnP-PGD}} =     \begin{pmatrix}
        1\\1\\1
    \end{pmatrix}, \quad
    \begin{pmatrix}
        \#N_\sigma\\
        \#D_\sigma\\
        \#\prox_f
    \end{pmatrix}_{\substack{
        \text{PnP-DRS;}\\
        \text{PnP-DRSdiff}}} = \begin{pmatrix}
        1\\1\\1
    \end{pmatrix}.
\end{equation*}

To compute the asymptotic complexity of PnP-LBFGS, suppose the images have dimension $d$, and that the denoisers have $P$ parameters. From \eqref{eq:pnplbfgsComplexity}, we can read off the complexity of computing one iteration given $d^k$ as $\mathcal{O}(d\times P \times T)$, with $\mathcal{O}(d)$ memory requirement to hold the $x^k, w^k$ and intermediate gradients. To compute $d^k$, the computational complexity of L-BFGS scales linearly with the input dimension and memory length $m$, and requires us to store $m$ secants and differences. The asymptotic complexity per iteration is thus $\mathcal{O}\left(d\times P \times T + md\right)$, where the number of tests $T$ is eventually always 1. The total memory requirement is $\mathcal{O}\left((m+1)\times d\right)$, where we store $m$ differences and secants. 

A similar complexity analysis can be applied to the PnP-PGD, PnP-DRSdiff and PnP-DRS methods to achieve a per-iteration computational complexity of $\mathcal{O}(d\times P)$ and memory requirement of $\mathcal{O}(d)$. However, these three PnP methods do not come with improved convergence rates under additional smoothness assumptions, and come with residual convergence at a rate $\min_{i\le k}\|x^{i+1} - x^i\| = \mathcal{O}(1/k)$. PnP-LBFGS achieves residual convergence $\min_{i\le k}\|R_{\gamma_i}(x^i)\| = \mathcal{O}(1/k)$ from \Cref{thm:residualConvergence}, as well as superlinear convergence under the assumptions of \Cref{thm:Superlinear}. This is summarized in \Cref{tab:complexity}.

\begin{table}[h]
\centering
\caption{Complexity to achieve an $\epsilon$-optimal solution, in terms of the squared residual for PnP-PGD/DRS/DRSdiff, and in terms of the residual $R_{\gamma_i}(x^i)$ for PnP-LBFGS. Under the assumptions of \Cref{thm:Superlinear} for superlinear convergence, the number of tests is eventually always $T=1$, and we are able to achieve at least linear speedup.}
\label{tab:complexity}
\begin{tabular}{@{}rlll@{}}
\toprule
\multicolumn{1}{l}{Complexity} &
  \multicolumn{1}{c}{PnP-PGD/DRS/DRSdiff} &
  \multicolumn{1}{c}{PnP-LBFGS} &
  \multicolumn{1}{c}{PnP-LBFGS superlinear} \\ \midrule
\multicolumn{1}{l|}{Computation} &
  $\mathcal{O}(d P \epsilon^{-1})$ &
  $\mathcal{O}\left((d P T + md)\epsilon^{-1}\right)$ &
  $\mathcal{O}\left((d P + md)\log \epsilon\right)$ \\
\multicolumn{1}{l|}{Memory} &
  $\mathcal{O}(d)$ &
  $\mathcal{O}\left((m+1) d\right)$ &
  $\mathcal{O}\left((m+1) d\right)$ \\ \bottomrule
\end{tabular}
\end{table}

The above complexity analysis shows that the main increase in computational burden for PnP-LBFGS is the requirement of two evaluations of $\nabla^2 f$ at each iteration, as well as at least double the number of neural network evaluations compared to the compared PnP methods. However, assuming only one test for $w^k$ is needed, each iteration only requires one additional evaluation of the denoiser-related networks $N_\sigma, D_\sigma$ and fidelity gradient $\nabla f$ (or $\prox_f$) to the compared PnP methods. In our experiments, $\nabla^2 f$ has a low computational cost due to the closed form. This allows us to trade roughly 2--3$\times$ the per-iteration cost with nearly $10\times$ fewer iterations required as shown in \Cref{fig:deblurConvergencelog2,fig:srConvergencelog2}, resulting in fewer total function calls, and thus the 4--5$\times$ faster reconstruction times as shown in \Cref{tab:PSNR_deblur,tab:PSNR-superres}. 

\section{Conclusion}
In this work, we propose a Plug-and-Play approach to image reconstruction that utilizes descent steps based on the forward-backward envelope. Using the descent formulation, we are able to further incorporate quasi-Newton steps to accelerate convergence. The resulting PnP scheme is provably convergent with a gradient-step assumption on the denoiser by using the Kurdyka-Łojasiewicz property and theoretically achieves superlinear convergence if a Hessian approximation satisfying the Dennis-Moré condition is available. Moreover, properties of the forward-backward envelope allow for additional ways of checking convergence. Our experiments demonstrate that it is able to converge significantly faster in terms of both time and iteration count as well as having highly competitive performance when compared with competing PnP methods with similar convergence guarantees.

For future works, one route is to consider alternative parameterizations of the denoiser $D_\sigma$. For example, consider the objective $\varphi = f + \phi_\sigma$ and the task of learning the regularization term $\phi_\sigma$ \cite{mukherjee2020LearnedConvexreg,mukherjee2022learned}. \Rev{By enforcing convexity of $\phi_\sigma$ through the neural network architecture, such as using input-convex neural networks \cite{amos2017input}, \RRev{(weakly-)} convex ridge regularizers \cite{goujon2022neural,goujon2023learning}, \RRev{firm nonexpansiveness} \cite{pesquet2021learning}, or parametric splines \cite{nguyen2017learning}}, results from \cite{Stella2017} utilizing convexity such as global sublinear convergence and local linear convergence can be applied. \RRev{This may also alleviate divergence problems caused when Lipschitz constraints on the denoisers are violated, as sometimes arises using spectral regularization.} \Rev{One restriction of the proposed method lies in the restriction of the regularization parameter, which imposes a bound on the minimum amount of regularization. Future works could look to loosen this restriction, similarly to \cite{hurault2023RelaxedProxDenoiser}.} In addition, while only simple forward operators such as image deblurring and super-resolution are experimented on in this work, the accelerated convergence rate and model-based interpretation may make this PnP scheme suitable for more complicated forward operators such as CT ray transforms. Future works may explore these practical applications, with a suitably trained ``denoiser" for these domains.

\section*{Acknowledgements}
Hong Ye Tan acknowledges support from GSK.ai and the Masason Foundation. Professor Carola-Bibiane Schönlieb acknowledges support from the Philip Leverhulme Prize, the Royal Society Wolfson Fellowship, the EPSRC advanced career fellowship EP/V029428/1, EPSRC grants EP/S026045/1 and EP/T003553/1, EP/N014588/1, EP/T017961/1, the Wellcome Innovator Awards 215733/Z/19/Z and 221633/Z/20/Z, the European Union Horizon 2020 research and innovation programme under the Marie Skodowska-Curie grant agreement No. 777826 NoMADS, the Cantab Capital Institute for the Mathematics of Information and the Alan Turing Institute.
\bibliographystyle{siamplain}
\bibliography{references}

\begin{thebibliography}{10}

\bibitem{amos2017input}
{\sc B.~Amos, L.~Xu, and J.~Z. Kolter}, {\em Input convex neural networks}, in International Conference on Machine Learning, PMLR, 2017, pp.~146--155.

\bibitem{arjomand2017deep}
{\sc S.~Arjomand~Bigdeli, M.~Zwicker, P.~Favaro, and M.~Jin}, {\em Deep mean-shift priors for image restoration}, Advances in Neural Information Processing Systems, 30 (2017).

\bibitem{armijo1966minimization}
{\sc L.~Armijo}, {\em Minimization of functions having {L}ipschitz continuous first partial derivatives}, Pacific Journal of mathematics, 16 (1966), pp.~1--3.

\bibitem{attouch2010proximal}
{\sc H.~Attouch, J.~Bolte, P.~Redont, and A.~Soubeyran}, {\em Proximal alternating minimization and projection methods for nonconvex problems: An approach based on the {K}urdyka-{{\L}}ojasiewicz inequality}, Mathematics of operations research, 35 (2010), pp.~438--457.

\bibitem{Attouch2013}
{\sc H.~Attouch, J.~Bolte, and B.~F. Svaiter}, {\em Convergence of descent methods for semi-algebraic and tame problems: proximal algorithms, forward--backward splitting, and regularized gauss--seidel methods}, Mathematical Programming, 137 (2013), pp.~91--129.

\bibitem{aujol2023fista}
{\sc J.-F. Aujol, C.~Dossal, and A.~Rondepierre}, {\em Fista is an automatic geometrically optimized algorithm for strongly convex functions}, Mathematical Programming,  (2023), pp.~1--43.

\bibitem{bauschke2011convex}
{\sc H.~Bauschke and P.~Combettes}, {\em Convex Analysis and Monotone Operator Theory in Hilbert Spaces}, CMS Books in Mathematics, Springer New York, 2011.

\bibitem{beck2017first}
{\sc A.~Beck}, {\em First-order methods in optimization}, SIAM, 2017.

\bibitem{beck2009fast}
{\sc A.~Beck and M.~Teboulle}, {\em A fast iterative shrinkage-thresholding algorithm for linear inverse problems}, SIAM journal on imaging sciences, 2 (2009), pp.~183--202.

\bibitem{Bolte2014}
{\sc J.~Bolte, S.~Sabach, and M.~Teboulle}, {\em Proximal alternating linearized minimization for nonconvex and nonsmooth problems}, Mathematical Programming, 146 (2014), pp.~459--494.

\bibitem{buades2011NLM}
{\sc A.~Buades, B.~Coll, and J.-M. Morel}, {\em Non-local means denoising}, Image Processing On Line, 1 (2011), pp.~208--212.

\bibitem{buzzard2018plug}
{\sc G.~T. Buzzard, S.~H. Chan, S.~Sreehari, and C.~A. Bouman}, {\em Plug-and-play unplugged: Optimization-free reconstruction using consensus equilibrium}, SIAM Journal on Imaging Sciences, 11 (2018), pp.~2001--2020.

\bibitem{byrd2016stochastic}
{\sc R.~H. Byrd, S.~L. Hansen, J.~Nocedal, and Y.~Singer}, {\em A stochastic quasi-{N}ewton method for large-scale optimization}, SIAM Journal on Optimization, 26 (2016), pp.~1008--1031.

\bibitem{chambolle2015convergence}
{\sc A.~Chambolle and C.~Dossal}, {\em On the convergence of the iterates of the “fast iterative shrinkage/thresholding algorithm”}, Journal of Optimization theory and Applications, 166 (2015), pp.~968--982.

\bibitem{chan2016plug}
{\sc S.~H. Chan, X.~Wang, and O.~A. Elgendy}, {\em Plug-and-play admm for image restoration: Fixed-point convergence and applications}, IEEE Transactions on Computational Imaging, 3 (2016), pp.~84--98.

\bibitem{cohen2021has}
{\sc R.~Cohen, Y.~Blau, D.~Freedman, and E.~Rivlin}, {\em It has potential: Gradient-driven denoisers for convergent solutions to inverse problems}, Advances in Neural Information Processing Systems, 34 (2021), pp.~18152--18164.

\bibitem{cohen2021regularization}
{\sc R.~Cohen, M.~Elad, and P.~Milanfar}, {\em Regularization by denoising via fixed-point projection (red-pro)}, SIAM Journal on Imaging Sciences, 14 (2021), pp.~1374--1406.

\bibitem{dabov2007bm3d}
{\sc K.~Dabov, A.~Foi, V.~Katkovnik, and K.~Egiazarian}, {\em Image denoising by sparse 3-d transform-domain collaborative filtering}, IEEE Transactions on image processing, 16 (2007), pp.~2080--2095.

\bibitem{daubechies2004iterative}
{\sc I.~Daubechies, M.~Defrise, and C.~De~Mol}, {\em An iterative thresholding algorithm for linear inverse problems with a sparsity constraint}, Communications on Pure and Applied Mathematics: A Journal Issued by the Courant Institute of Mathematical Sciences, 57 (2004), pp.~1413--1457.

\bibitem{dennis1974characterization}
{\sc J.~E. Dennis and J.~J. Mor{\'e}}, {\em A characterization of superlinear convergence and its application to quasi-{N}ewton methods}, Mathematics of computation, 28 (1974), pp.~549--560.

\bibitem{douglas1956numerical}
{\sc J.~Douglas and H.~H. Rachford}, {\em On the numerical solution of heat conduction problems in two and three space variables}, Transactions of the American mathematical Society, 82 (1956), pp.~421--439.

\bibitem{fazlyab2019efficient}
{\sc M.~Fazlyab, A.~Robey, H.~Hassani, M.~Morari, and G.~Pappas}, {\em Efficient and accurate estimation of {L}ipschitz constants for deep neural networks}, Advances in Neural Information Processing Systems, 32 (2019).

\bibitem{gan2023block}
{\sc W.~Gan, S.~Shoushtari, Y.~Hu, J.~Liu, H.~An, and U.~S. Kamilov}, {\em Block coordinate plug-and-play methods for blind inverse problems}, arXiv preprint arXiv:2305.12672,  (2023).

\bibitem{goldstein2014field}
{\sc T.~Goldstein, C.~Studer, and R.~Baraniuk}, {\em A field guide to forward-backward splitting with a fasta implementation}, arXiv preprint arXiv:1411.3406,  (2014).

\bibitem{goujon2022neural}
{\sc A.~Goujon, S.~Neumayer, P.~Bohra, S.~Ducotterd, and M.~Unser}, {\em A neural-network-based convex regularizer for image reconstruction}, arXiv preprint arXiv:2211.12461,  (2022).

\bibitem{goujon2023learning}
{\sc A.~Goujon, S.~Neumayer, and M.~Unser}, {\em Learning weakly convex regularizers for convergent image-reconstruction algorithms}, arXiv preprint arXiv:2308.10542,  (2023).

\bibitem{gribonval2020characterization}
{\sc R.~Gribonval and M.~Nikolova}, {\em A characterization of proximity operators}, Journal of Mathematical Imaging and Vision, 62 (2020), pp.~773--789.

\bibitem{helgason1980radon}
{\sc S.~Helgason and S.~Helgason}, {\em The radon transform}, vol.~2, Springer, 1980.

\bibitem{hertrich2021convolutional}
{\sc J.~Hertrich, S.~Neumayer, and G.~Steidl}, {\em Convolutional proximal neural networks and plug-and-play algorithms}, Linear Algebra and its Applications, 631 (2021), pp.~203--234.

\bibitem{huang2015broyden}
{\sc W.~Huang, K.~A. Gallivan, and P.-A. Absil}, {\em A {B}royden class of quasi-{N}ewton methods for {R}iemannian optimization}, SIAM Journal on Optimization, 25 (2015), pp.~1660--1685.

\bibitem{hurault2023RelaxedProxDenoiser}
{\sc S.~Hurault, A.~Chambolle, A.~Leclaire, and N.~Papadakis}, {\em A relaxed proximal gradient descent algorithm for convergent plug-and-play with proximal denoiser}, 2023.

\bibitem{hurault2021gradient}
{\sc S.~Hurault, A.~Leclaire, and N.~Papadakis}, {\em Gradient step denoiser for convergent plug-and-play}, arXiv preprint arXiv:2110.03220,  (2021).

\bibitem{hurault2022proximal}
{\sc S.~Hurault, A.~Leclaire, and N.~Papadakis}, {\em Proximal denoiser for convergent plug-and-play optimization with nonconvex regularization}, 2022.

\bibitem{jin2022sharpened}
{\sc Q.~Jin, A.~Koppel, K.~Rajawat, and A.~Mokhtari}, {\em Sharpened quasi-{N}ewton methods: Faster superlinear rate and larger local convergence neighborhood}, in International Conference on Machine Learning, PMLR, 2022, pp.~10228--10250.

\bibitem{jin2023non}
{\sc Q.~Jin and A.~Mokhtari}, {\em Non-asymptotic superlinear convergence of standard quasi-{N}ewton methods}, Mathematical Programming, 200 (2023), pp.~425--473.

\bibitem{kaipio2006statistical}
{\sc J.~Kaipio and E.~Somersalo}, {\em Statistical and computational inverse problems}, vol.~160, Springer Science \& Business Media, 2006.

\bibitem{kamilov2023plug}
{\sc U.~S. Kamilov, C.~A. Bouman, G.~T. Buzzard, and B.~Wohlberg}, {\em Plug-and-play methods for integrating physical and learned models in computational imaging: Theory, algorithms, and applications}, IEEE Signal Processing Magazine, 40 (2023), pp.~85--97.

\bibitem{kamilov2017plug}
{\sc U.~S. Kamilov, H.~Mansour, and B.~Wohlberg}, {\em A plug-and-play priors approach for solving nonlinear imaging inverse problems}, IEEE Signal Processing Letters, 24 (2017), pp.~1872--1876.

\bibitem{lee2014proximal}
{\sc J.~D. Lee, Y.~Sun, and M.~A. Saunders}, {\em Proximal {N}ewton-type methods for minimizing composite functions}, SIAM Journal on Optimization, 24 (2014), pp.~1420--1443.

\bibitem{levin2009understanding}
{\sc A.~Levin, Y.~Weiss, F.~Durand, and W.~T. Freeman}, {\em Understanding and evaluating blind deconvolution algorithms}, in 2009 IEEE conference on computer vision and pattern recognition, IEEE, 2009, pp.~1964--1971.

\bibitem{li2001modified}
{\sc D.-H. Li and M.~Fukushima}, {\em A modified {BFGS} method and its global convergence in nonconvex minimization}, Journal of Computational and Applied Mathematics, 129 (2001), pp.~15--35.

\bibitem{li2001global}
{\sc D.-H. Li and M.~Fukushima}, {\em On the global convergence of the {BFGS} method for nonconvex unconstrained optimization problems}, SIAM Journal on Optimization, 11 (2001), pp.~1054--1064.

\bibitem{liu1989limited}
{\sc D.~C. Liu and J.~Nocedal}, {\em On the limited memory {BFGS} method for large scale optimization}, Mathematical programming, 45 (1989), pp.~503--528.

\bibitem{lunz2018adversarial}
{\sc S.~Lunz, O.~{\"O}ktem, and C.-B. Sch{\"o}nlieb}, {\em Adversarial regularizers in inverse problems}, Advances in neural information processing systems, 31 (2018).

\bibitem{martin2001database}
{\sc D.~Martin, C.~Fowlkes, D.~Tal, and J.~Malik}, {\em A database of human segmented natural images and its application to evaluating segmentation algorithms and measuring ecological statistics}, in Proceedings Eighth IEEE International Conference on Computer Vision. ICCV 2001, vol.~2, IEEE, 2001, pp.~416--423.

\bibitem{miyato2018spectral}
{\sc T.~Miyato, T.~Kataoka, M.~Koyama, and Y.~Yoshida}, {\em Spectral normalization for generative adversarial networks}, arXiv preprint arXiv:1802.05957,  (2018).

\bibitem{mokhtari2015global}
{\sc A.~Mokhtari and A.~Ribeiro}, {\em Global convergence of online limited memory {BFGS}}, The Journal of Machine Learning Research, 16 (2015), pp.~3151--3181.

\bibitem{moreau1965proximite}
{\sc J.-J. Moreau}, {\em Proximit{\'e} et dualit{\'e} dans un espace hilbertien}, Bulletin de la Soci{\'e}t{\'e} math{\'e}matique de France, 93 (1965), pp.~273--299.

\bibitem{mukherjee2020LearnedConvexreg}
{\sc S.~Mukherjee, S.~Dittmer, Z.~Shumaylov, S.~Lunz, O.~Öktem, and C.-B. Schönlieb}, {\em Learned convex regularizers for inverse problems}, 2020.

\bibitem{mukherjee2022learned}
{\sc S.~Mukherjee, A.~Hauptmann, O.~{\"O}ktem, M.~Pereyra, and C.-B. Sch{\"o}nlieb}, {\em Learned reconstruction methods with convergence guarantees}, arXiv preprint arXiv:2206.05431,  (2022).

\bibitem{nair2021fixed}
{\sc P.~Nair, R.~G. Gavaskar, and K.~N. Chaudhury}, {\em Fixed-point and objective convergence of plug-and-play algorithms}, IEEE Transactions on Computational Imaging, 7 (2021), pp.~337--348.

\bibitem{neumayer2023approximation}
{\sc S.~Neumayer, A.~Goujon, P.~Bohra, and M.~Unser}, {\em Approximation of {L}ipschitz functions using deep spline neural networks}, SIAM Journal on Mathematics of Data Science, 5 (2023), pp.~306--322.

\bibitem{nguyen2017learning}
{\sc H.~Q. Nguyen, E.~Bostan, and M.~Unser}, {\em Learning convex regularizers for optimal {B}ayesian denoising}, IEEE Transactions on Signal Processing, 66 (2017), pp.~1093--1105.

\bibitem{NoceWrig06}
{\sc J.~Nocedal and S.~J. Wright}, {\em Numerical Optimization}, Springer, New York, NY, USA, 2e~ed., 2006.

\bibitem{ochs2019adaptive}
{\sc P.~Ochs and T.~Pock}, {\em Adaptive fista for nonconvex optimization}, SIAM Journal on Optimization, 29 (2019), pp.~2482--2503.

\bibitem{PyTorch}
{\sc A.~Paszke, S.~Gross, F.~Massa, A.~Lerer, J.~Bradbury, G.~Chanan, T.~Killeen, Z.~Lin, N.~Gimelshein, L.~Antiga, A.~Desmaison, A.~Kopf, E.~Yang, Z.~DeVito, M.~Raison, A.~Tejani, S.~Chilamkurthy, B.~Steiner, L.~Fang, J.~Bai, and S.~Chintala}, {\em Pytorch: An imperative style, high-performance deep learning library}, in Advances in Neural Information Processing Systems 32, H.~Wallach, H.~Larochelle, A.~Beygelzimer, F.~d\textquotesingle Alch\'{e}-Buc, E.~Fox, and R.~Garnett, eds., Curran Associates, Inc., 2019, pp.~8024--8035, \url{http://papers.neurips.cc/paper/9015-pytorch-an-imperative-style-high-performance-deep-learning-library.pdf}.

\bibitem{pesquet2021learning}
{\sc J.-C. Pesquet, A.~Repetti, M.~Terris, and Y.~Wiaux}, {\em Learning maximally monotone operators for image recovery}, SIAM Journal on Imaging Sciences, 14 (2021), pp.~1206--1237.

\bibitem{reehorst2018regularization}
{\sc E.~T. Reehorst and P.~Schniter}, {\em Regularization by denoising: Clarifications and new interpretations}, IEEE transactions on computational imaging, 5 (2018), pp.~52--67.

\bibitem{Rockafellar1972CA}
{\sc R.~Rockafellar}, {\em Convex Analysis}, Princeton mathematical series ; 28, Princeton University Press, Princeton, NJ, 1972.

\bibitem{rodomanov2021greedy}
{\sc A.~Rodomanov and Y.~Nesterov}, {\em Greedy quasi-{N}ewton methods with explicit superlinear convergence}, SIAM Journal on Optimization, 31 (2021), pp.~785--811.

\bibitem{rodomanov2021rates}
{\sc A.~Rodomanov and Y.~Nesterov}, {\em Rates of superlinear convergence for classical quasi-{N}ewton methods}, Mathematical Programming,  (2021), pp.~1--32.

\bibitem{ruderman94FourierReg}
{\sc D.~L. Ruderman}, {\em The statistics of natural images}, Network: Computation in Neural Systems, 5 (1994), pp.~517--548.

\bibitem{rudin1992nonlinear}
{\sc L.~I. Rudin, S.~Osher, and E.~Fatemi}, {\em Nonlinear total variation based noise removal algorithms}, Physica D: nonlinear phenomena, 60 (1992), pp.~259--268.

\bibitem{ryu2019plug}
{\sc E.~Ryu, J.~Liu, S.~Wang, X.~Chen, Z.~Wang, and W.~Yin}, {\em Plug-and-play methods provably converge with properly trained denoisers}, in International Conference on Machine Learning, PMLR, 2019, pp.~5546--5557.

\bibitem{schraudolph2007stochastic}
{\sc N.~N. Schraudolph, J.~Yu, and S.~G{\"u}nter}, {\em A stochastic quasi-{N}ewton method for online convex optimization}, in Artificial intelligence and statistics, PMLR, 2007, pp.~436--443.

\bibitem{sreehari2016plug}
{\sc S.~Sreehari, S.~V. Venkatakrishnan, B.~Wohlberg, G.~T. Buzzard, L.~F. Drummy, J.~P. Simmons, and C.~A. Bouman}, {\em Plug-and-play priors for bright field electron tomography and sparse interpolation}, IEEE Transactions on Computational Imaging, 2 (2016), pp.~408--423.

\bibitem{Stella2017}
{\sc L.~Stella, A.~Themelis, and P.~Patrinos}, {\em Forward--backward quasi-{N}ewton methods for nonsmooth optimization problems}, Computational Optimization and Applications, 67 (2017), pp.~443--487.

\bibitem{sun2020async}
{\sc Y.~Sun, J.~Liu, Y.~Sun, B.~Wohlberg, and U.~S. Kamilov}, {\em Async-red: A provably convergent asynchronous block parallel stochastic method using deep denoising priors}, arXiv preprint arXiv:2010.01446,  (2020).

\bibitem{sun2019online}
{\sc Y.~Sun, B.~Wohlberg, and U.~S. Kamilov}, {\em An online plug-and-play algorithm for regularized image reconstruction}, IEEE Transactions on Computational Imaging, 5 (2019), pp.~395--408.

\bibitem{sun2021scalable}
{\sc Y.~Sun, Z.~Wu, X.~Xu, B.~Wohlberg, and U.~S. Kamilov}, {\em Scalable plug-and-play admm with convergence guarantees}, IEEE Transactions on Computational Imaging, 7 (2021), pp.~849--863.

\bibitem{tang2022accelerating}
{\sc J.~Tang}, {\em Accelerating plug-and-play image reconstruction via multi-stage sketched gradients}, arXiv preprint arXiv:2203.07308,  (2022).

\bibitem{tjoa2020survey}
{\sc E.~Tjoa and C.~Guan}, {\em A survey on explainable artificial intelligence (xai): Toward medical xai}, IEEE transactions on neural networks and learning systems, 32 (2020), pp.~4793--4813.

\bibitem{vellido2020importance}
{\sc A.~Vellido}, {\em The importance of interpretability and visualization in machine learning for applications in medicine and health care}, Neural computing and applications, 32 (2020), pp.~18069--18083.

\bibitem{venkatakrishnan2013plug}
{\sc S.~V. Venkatakrishnan, C.~A. Bouman, and B.~Wohlberg}, {\em Plug-and-play priors for model based reconstruction}, in 2013 IEEE Global Conference on Signal and Information Processing, IEEE, 2013, pp.~945--948.

\bibitem{wang2017stochastic}
{\sc X.~Wang, S.~Ma, D.~Goldfarb, and W.~Liu}, {\em Stochastic quasi-{N}ewton methods for nonconvex stochastic optimization}, SIAM Journal on Optimization, 27 (2017), pp.~927--956.

\bibitem{wen2018survey}
{\sc F.~Wen, L.~Chu, P.~Liu, and R.~C. Qiu}, {\em A survey on nonconvex regularization-based sparse and low-rank recovery in signal processing, statistics, and machine learning}, IEEE Access, 6 (2018), pp.~69883--69906.

\bibitem{zhang2021plug}
{\sc K.~Zhang, Y.~Li, W.~Zuo, L.~Zhang, L.~Van~Gool, and R.~Timofte}, {\em Plug-and-play image restoration with deep denoiser prior}, IEEE Transactions on Pattern Analysis and Machine Intelligence, 44 (2021), pp.~6360--6376.

\bibitem{zhang2017learning}
{\sc K.~Zhang, W.~Zuo, S.~Gu, and L.~Zhang}, {\em Learning deep cnn denoiser prior for image restoration}, in IEEE Conference on Computer Vision and Pattern Recognition, 2017, pp.~3929--3938.

\end{thebibliography}

\newpage
\appendix
\section{Convergence under Kurdyka-Łojasiewicz (KL) Property}
We wish to prove \Cref{thm:KLConvergence}, which we restate below. The proof will be similar to that in \cite[Thm 3.10]{Stella2017}, which we reproduce the proofs below for completeness.

\begin{theorem}
    Suppose that $f$ satisfies the KL condition and $g$ is semialgebraic, and both are bounded from below. Suppose further that there exists constants $\bar\tau, c>0$ such that $\tau_k<\bar\tau$ and $\|d^k\| \le c \|R_{\gamma_k}(x^k)\|$, $\beta>0$, and that $\varphi$ is coercive or has compact level sets. Then the sequence of iterates is finite and ends with $R_{\gamma_k}(x^k)=0$ or converges to a critical point of $\varphi$.
\end{theorem}

Recall the definition of the KL property.

\begin{definition}[KL Property \cite{Attouch2013,Bolte2014}] \label{def:KL}
    Suppose $\varphi:\R^d \rightarrow \overline{\R}$ is proper and lower semi-continuous. $h$ satisfies the \emph{Kurdyka-Łojasiewicz (KL) property} at a point $x_*$ in $\dom \partial \varphi$ if there exists $\eta \in (0,+\infty]$, a neighborhood $U$ of $x_*$ and a continuous concave function $\Psi:[0,\eta) \rightarrow [0,+\infty)$ such that:
    \begin{enumerate}
        \item $\Psi(0) = 0$;
        \item $\Psi$ is $\mathcal{C}^1$ on $(0,\eta)$
        \item $\Psi'(s)>0$ for $s \in (0,\eta)$;
        \item For all $u \in U \cap \{\varphi(x_*) < \varphi(u) < \varphi(x_*) + \eta\}$, we have
        \[\psi'(\varphi(u) - \varphi(x_*)) \dist (0, \partial \varphi(u)) \ge 1.\]
    \end{enumerate}
    We say that $\varphi$ is a KL function if the KL property is satisfied at every point of $\dom \partial \varphi$.
\end{definition}

Let us state the main assumption that will be used in the following lemmas.
\begin{assumption}\label{ass:1}
    $\varphi = f+g$, where $f$ is $\mathcal{C}^1$ with $L_f$-Lipschitz gradient, and $g$ is $M$-weakly convex, proper and lower semicontinuous. Furthermore, assume $\gamma_0 < 1/M$.
\end{assumption}
\begin{lemma}[{\cite[Lem 6.6]{Stella2017}}] \label{lem:6.6}
    Assume \Cref{ass:1}, and consider the sequences $(x^k),\ (w^k)$ generated by MINFBE. If there exist $\bar \tau, c>0$ such that $\tau_k \le \bar \tau$ and $\|d^k\|\le c \|R_{\gamma_k}(x^k)\|$, then
    \begin{equation}\label{eq:6.2}
        \|x^{k+1} - x^k\| \le \gamma_k \|R_{\gamma_k}(w^k)\| + \bar \tau c \|R_{\gamma_k}(x^k)\| \quad \forall k \in \mathbb{N}
    \end{equation}
    and for $k$ large enough,
    \begin{equation}\label{eq:6.3}
        \|x^{k+1} - x^k\| \le \gamma_k \|R_{\gamma_k}(w^k)\| + \bar \tau c (1+\gamma_k L_f)\|R_{\gamma_{k-1}}(w^{k-1})\|.
    \end{equation}
\end{lemma}
\begin{proof}
    \eqref{eq:6.2} follows from the triangle inequality, the definition of $R_{\gamma}$ and $\tau_k \le \bar \tau$:
    \begin{align*}
    \|x^{k+1} - x^k\| &= \|x^{k+1} - w^k + \tau_k d^k\| \\
    & \le \|w^k - T_{\gamma_k}(w^k) \|+ \|\tau_k d^k\| \\
    &\le \gamma_k \|R_{\gamma_k}(w^k)\| + \bar \tau c \|R_{\gamma_k}(x^k)\|.
    \end{align*} 
    For $k$ sufficiently large s.t. $\gamma_k = \gamma_{k-1} = \gamma_\infty > 0$ (c.f. \Cref{lem:gammaBddBelowWeakConv}), dropping the subscripts of the $\gamma$,
    \begin{align*}
        \|R_\gamma(x^k)\| &= \gamma^{-1}\|x^k - T_\gamma(x^k)\| \\
        &= \gamma^{-1} \|T_\gamma(w^{k-1}) - T_\gamma(x^k)\| \\
        &\le \gamma^{-1} \|w^{k-1} - \gamma \nabla f(w^{k-1}) - x^k + \gamma \nabla f(x^k)\| \\
        &\le \gamma^{-1} \|w^{k-1} - x^k\| + \|\nabla f(w^{k-1}) - \nabla f(x^k)\| \\
        &\le (1+\gamma L_f) \|R_\gamma(w^{k-1})\|.
    \end{align*}
    where the first inequality follows from non-expansivity of $\prox_{\gamma g}$ as $\gamma < 1/M$, the second from triangle inequality, and the final from $\nabla f$ being $L_f$-Lipschitz. Adding back the subscripts yields \eqref{eq:6.3}.
\end{proof}

\begin{lemma}[{\cite[Lem 6.7]{Stella2017}}] \label{lem:6.7}
    Let $(\beta_k)_{k \in \mathbb{N}}$ and $(\delta_k)_{k \in \mathbb{N}}$ be real sequences with $\beta_k, \delta_k\ge 0,\, \delta_{k+1} \le \delta_k$ and $\beta^2_{k+1} \le (\delta_k - \delta_{k+1}) \beta_k$ for all $k \in \mathbb{N}$. Then $\sum_{k=0}^\infty \beta_k < \infty$.
\end{lemma}
\begin{proof}
    As in \cite{Stella2017}.
\end{proof}

\begin{lemma}[{\cite[Lem 5]{Bolte2014}}] \label{lem:clusterpts}
    Suppose $(z^k)$ is a bounded sequence with $\|z^{k+1} - z^k\| \rightarrow 0$ as $k \rightarrow \infty$. Then
    \begin{enumerate}[i.]
        \item $\lim_{k\rightarrow \infty} \dist_{\omega (z^0)}(z^k) = 0$;
        \item $\omega(z^0)$ is nonempty compact and connected.
    \end{enumerate}
\end{lemma}
\begin{proof}
    \textbf{(i).} Follows from the fact that a sequence converges if and only if every subsequence has a convergent subsequence. In this case, taking the sequence to be $(\dist_{\omega (z^0)}(z^k))_{k \in \mathbb{N}}$ and using boundedness of $z^k$ to get a convergent subsequence of any subsequence yields the result.
    \textbf{(ii).} Follows from Lemma 5 and Remark 5 in \cite{Bolte2014}.
\end{proof}

\begin{prop}[{\cite[Prop 6.8]{Stella2017}}] \label{prop:6.8}
Suppose \Cref{ass:1} is satisfied, and $\varphi$ is bounded from below. Consider the sequences generated by MINFBE. If $\beta \in (0,1)$ and there exist $\bar \tau, c>0$ such that $\tau_k \le \bar \tau$ and $\|d^k\|\le c \|R_{\gamma_k}(x^k)\|$, then 
\begin{equation}
    \sum_{k=0}^\infty \|x^{k+1} - x^k\|^2 < \infty.
\end{equation}
If additionally $(x^k)_{k \in \mathbb{N}}$ is bounded, then 
\begin{equation} \label{eq:distlimitzero}
    \lim_{k\rightarrow \infty} \dist_{\omega (x^0)}(x^k) = 0,
\end{equation}
where $\omega(x^0)$ is the set of cluster points of the sequence $(x^k)$. Moreover, $\omega(x^0)$ is a nonempty compact connected subset of $\zer \partial \varphi$ over which $\varphi$ is constant.
\end{prop}
\begin{proof}
    Since $\varphi$ is lower bounded, by \Cref{thm:phiDecreases}(ii) and (iv), $\|R_{\gamma_k}(x^k)\|$ and $\|R_{\gamma_k}(w^k)\|$ are square-summable. By \eqref{eq:6.2}, we have that $\|x^{k+1} - x^k\|$ is upper bounded by a sum of square-summable sequences, hence is square-summable.

    If $(x^k)$ is bounded, \Cref{lem:clusterpts} yields \eqref{eq:distlimitzero} and (nonempty) compactness and connectedness. $\omega(x^0) \subset \zer \partial \varphi$ follows from \Cref{thm:phiDecreases}(iii).

    To show $\varphi$ is constant over $\omega(x^0)$, denote by $\varphi^*$ the finite limit of $\varphi(x^k)$. This exists since $\varphi(x^k)$ is monotonically nonincreasing by \Cref{thm:phiDecreases}(i), and bounded below. Take $x^*$ in $\omega(x^0)$, which has subsequence $x^{k_j} \rightarrow x_*$. Then by (lower semi-)continuity of $\varphi$, we have $\varphi(x_*) \le \liminf \varphi(x^{k_j}) = \varphi^*$. So $\varphi(x_*) = \varphi^*$.
\end{proof}

\begin{lemma}[{\cite[Lem 6.9]{Stella2017}}, {\cite[Lem 6]{Bolte2014}}]\label{lem:uniformKL}
    Let $K \subset \R^n$ be a compact set and suppose the proper lower semi-continuous function $\varphi:\R^n \rightarrow \overline{\R}$ is constant on $K$ and satisfies the KL property at every $x_*$ in $K$. Then there exists $\varepsilon, \eta>0$ and a continuous concave function $\psi:[0,\eta] \rightarrow [0,+\infty)$ s.t. properties \ref{def:KL}(i-iii) hold and additionally
    \begin{enumerate}[i'.]
        \setcounter{enumi}{3}
        \item for all $x_*\in K$ and $x$ such that $\dist_K(x) < \varepsilon$ and $\varphi(x_*) < \varphi(x) < \varphi(x_*) + \eta$,
        \begin{equation} \label{eq:uniformKL}
            \psi'(\varphi(x) - \varphi(x_*)) \dist(0, \partial \varphi(x)) \ge 1.
        \end{equation}
    \end{enumerate}
\end{lemma}
\begin{proof}
    As in \cite[Lem 6]{Bolte2014}, cited without proof in \cite[Lem 6.9]{Stella2017}.
\end{proof}

We are ready to prove the main result now. We follow the proof of \cite[Thm 3.10]{Stella2017}.

\begin{proof}[Proof of \Cref{thm:KLConvergence}]
If $(x^k)$ terminates, then we are done. Suppose that $(x^k)$ is infinite, then $R_{\gamma_k}(x^k) \ne 0$ for all $k$. By \Cref{thm:phiDecreases}(i) and $\gamma_k < 1/M$, we have that $\varphi(x^{k+1}) < \varphi(x^k)$. Note that since $g$ is semialgebraic and hence KL, so is $\varphi = f + g$. Moreover, $\varphi$ is bounded below since both $f$ and $g$ are. By \Cref{prop:6.8}, the KL property \ref{def:KL} and \Cref{lem:uniformKL} for $\varphi$ applied with the compact cluster set $K = \omega(x^0)$, there exists $\varepsilon, \eta > 0$ and continuous concave $\psi:[0,\eta] \rightarrow [0,+\infty)$ such that for any $x$ with $\dist_{\omega(x^0)}(x)< \varepsilon$ and $\varphi(x_*) < \varphi(x) < \varphi(x_*) + \eta$, we have
\[\psi'(\varphi(x) - \varphi(x_*)) \dist(0, \partial \varphi(x)) \ge 1.\]

By \Cref{prop:6.8}, for sufficiently large $k$, $\dist_{\omega(x^0)}(x^k)< \varepsilon$. Moreover, since $\varphi(x^k) \rightarrow \varphi(x_*) = \varphi^*$, we have for sufficiently large $k$ that $\varphi(x^k) < \varphi(x_*) + \eta$. Let $\bar k$ be sufficiently large that for $k\ge \bar k$, both of these conditions hold. Then for every $k \ge \bar k$, we have
\[\psi'(\varphi(x^k) - \varphi(x_*)) \dist(0, \partial \varphi(x^k)) \ge 1.\]

From \Cref{thm:phiDecreases}(i), and $\gamma < 1/M$, we have
\begin{equation}\label{eq:weakSuffDecr}
    \varphi(x^{k+1}) \le \varphi(x^k) - \frac{\beta \gamma_k}{2} \|R_{\gamma_k}(w^k)\|^2.
\end{equation}

Define for $k>0$, $\tilde \nabla \varphi(x^k) = \nabla f(x^k) - \nabla f(w^{k-1}) + R_{\gamma_{k-1}}(w^{k-1})$. Since $R_{\gamma_{k-1}}(w^{k-1}) \in \nabla f(w^{k-1}) + \partial g(x^k)$, we have $\tilde \nabla \varphi(x^k) \in \partial \varphi(x^k)$, and
\begin{align*}
    \|\tilde \nabla \varphi(x^k)\| & \le \|\nabla f(x^k) - \nabla f(w^{k-1})\| + \|R_{\gamma_{k-1}}(w^k)\|\\
    &\le (1 + \gamma_{k-1}L_f) \|R_{\gamma_{k-1}}(w^{k-1})\|,
\end{align*}
where the second inequality follows from $\nabla f$ being $L_f$-Lipschitz.

From \eqref{eq:uniformKL}, since $\tilde \nabla \varphi(x^k) \in \partial \varphi(x^k)$,
\begin{align*}
    \psi'(\varphi(x^k) - \varphi(x_*)) &\ge \frac{1}{\dist(0, \partial\varphi(x^k))} \\
    & \ge \frac{1}{\|\tilde \nabla \varphi(x^k)\|}\\
    & \ge \frac{1}{(1 + \gamma_{k-1} L_f) \|R_{\gamma_{k-1}}(w^{k-1})\|}.
\end{align*}

Now let $\Delta_k = \psi(\varphi(x^k) - \varphi(x_*))$. By concavity of $\psi$ and \eqref{eq:weakSuffDecr},

\begin{align*}
    \Delta_k - \Delta_{k+1} &\ge \psi'(\varphi(x^k) - \varphi(x_*)) (\varphi(x^k) - \varphi(x^{k+1}) \\
    &\ge \frac{\beta  \gamma_k}{2(1+\gamma_{k-1}L_f)} \frac{\|R_{\gamma_k}(w^k)\|^2}{\|R_{\gamma_{k-1}}(w^{k-1})\|}\\
    & \ge \frac{\beta  \gamma_{\text{min}}}{2(1+\gamma_0 L_f)} \frac{\|R_{\gamma_k}(w^k)\|^2}{\|R_{\gamma_{k-1}}(w^{k-1})\|},
\end{align*}
where $\gamma_{\text{min}} = \min \left\{ \gamma_0, \xi (1-\beta)/L_f, 1/M\right\}$ as in \Cref{lem:gammaBddBelowWeakConv}. Rearranging, we have
\begin{equation}
    \|R_{\gamma_k}(w^k)\|^2 \le \alpha (\Delta_k - \Delta_{k+1}) \|R_{\gamma_{k-1}}(w^{k-1})\|,
\end{equation}
where $\alpha = 2(1+\gamma_0 L_f)/(\beta \gamma_{\text{min}})$. Applying \Cref{lem:6.7} with 
\[\delta_k = \alpha \Delta_k,\quad \beta_k = \|R_{\gamma_{k-1}}(w^{k-1})\|,\]
we get absolute summability of $\|R_{\gamma_k}(w^k)\|$. By \eqref{eq:6.3}, we get absolute summability of $(x^{k+1} - x^k)$. Hence $(x^k)$ is Cauchy and converges, and the limit is a critical point of $\varphi$ by \Cref{thm:phiDecreases}(iii).
\end{proof}

\section{Experiments: Blur kernels}
For the deblurring and super-resolution experiments, we use the same blur kernels as in \cite{hurault2022proximal}. This allows for a fair comparison of the algorithms by using the same initializations. For deblurring, the kernels consist of eight camera shake kernels, a $9\times 9$ uniform kernel and a $25\times25$ Gaussian kernel with standard deviation $\sigma_{\text{blur}} = 1.6$ \cite{levin2009understanding}. For super-resolution, we use Gaussian blur kernels with strengths $\{0.7,\,1.2,\,1.6,\,2.0\}$.

    \begin{figure}[H] \centering
    \begin{tabular}{c}
    \includegraphics[width=0.08\textwidth]{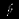}\;
    \includegraphics[width=0.08\textwidth]{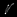}\;
    \includegraphics[width=0.08\textwidth]{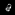}\;
    \includegraphics[width=0.08\textwidth]{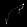}\;
    \includegraphics[width=0.08\textwidth]{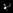}\\
    \includegraphics[width=0.08\textwidth]{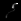}\;
    \includegraphics[width=0.08\textwidth]{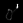}\;
    \includegraphics[width=0.08\textwidth]{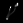}\;
    \includegraphics[width=0.08\textwidth]{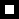}\;
    \includegraphics[width=0.08\textwidth]{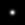}
    \end{tabular}
    \caption{The 10 blur kernels used for deblurring evaluation.\vspace{-0.3cm}
    }
    \end{figure}

    \begin{figure}[H] \centering
    \includegraphics[width=0.08\textwidth]{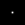}\;
    \includegraphics[width=0.08\textwidth]{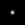}\;
    \includegraphics[width=0.08\textwidth]{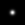}\;
    \includegraphics[width=0.08\textwidth]{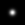}\;
    \caption{The 4 blur kernels used for super-resolution evaluation.\vspace*{-0.2cm}
    }
    \end{figure}

\end{document}